\documentclass[12pt]{article}
\usepackage{amsfonts,amsthm,amsmath,amssymb,amscd,graphicx}
\usepackage{epsfig}
\usepackage{color}

\theoremstyle{plain}
\newtheorem{lemma}{Lemma}[section]
\newtheorem{theorem}{Theorem}

\newtheorem{lm}{Lemma}
\newtheorem{prop}[lemma]{Proposition}

\theoremstyle{definition}
\newtheorem{df}{Definition}

\newtheorem{remark}{Remark}[section]

\title{
On dynamics and bifurcations of area-preserving maps with homoclinic
tangencies.} 
\author{
Amadeu Delshams$^{(1)}$, Marina Gonchenko$^{(2)}$
and Sergey Gonchenko$^{(3)}$
} 
\date{}

\begin{document}
\maketitle
\begin{center}
(1) Departament de Matem\`atica Aplicada I, Universitat
Polit\`ecnica de Catalunya,

Diagonal 647, 08028 Barcelona (Spain)

{\tt Amadeu.Delshams@upc.edu}

(2) Institut f\"ur Mathematik,
Technische Universit\"at Berlin\\
Strasse des 17. Juni 136, 
10623 Berlin (Germany)\\

{\tt  gonchenk@math.tu-berlin.de}

(3) Institute of Applied Mathematics and Cybernetics,
Nizhny Novgorod State University

Ulyanova St 10, 603005 Nizhny Novgorod (Russia)

{\tt gosv100@uic.nnov.ru}
\end{center}

\begin{abstract}
 We study bifurcations of area-preserving maps, both orientable (symplectic) and non-orientable, with quadratic homoclinic tangencies.
We consider one
and two parameter general unfoldings and establish results related to the appearance of elliptic
periodic orbits. In particular, we find conditions {for} such maps {to} have infinitely many generic (KAM-stable) elliptic periodic orbits of all
successive periods
starting
at some number.
\end{abstract}

\section{Introduction.}
The present paper is devoted to the study of bifurcations of area-preserving maps (APMs) with quadratic homoclinic tangencies. The case of
two-dimensional symplectic {(area-preserving and orientable)} maps  was analyzed in the papers~{\cite{MR97,GS01,GS03,Dua08,GG09}}.
Closely related bifurcation problems were considered in the papers~\cite{B87,BSh89,Ler91,Ler00,LerK00}
where bifurcations of conservative flows
with a homoclinic loop of a saddle-focus equilibrium were studied. In the works~\cite{B87,BSh89} the case of
three-dimensional
divergence-free flows was considered, while in ~\cite{Ler91,Ler00,LerK00} {the} dynamical behaviour and bifurcations in two degrees of freedom
Hamiltonian systems were analyzed.

In the present paper {we do not restrict ourselves to symplectic maps, but we also consider}
the new case of area-preserving and
non-orientable maps. First, we give a classification of APMs with quadratic homoclinic tangencies and, further, prove certain theorems on the existence of infinitely many bifurcations (cascades)  leading to the appearance of generic (KAM-stable) elliptic periodic orbits.

We recall that for dissipative systems, the related
problems are quite traditional and
many results obtained here
{are of fundamental importance in} the theory of dynamical chaos. One of such
results, known as {\em theorem on cascade of periodic sinks}, goes back to the famous papers of
Gavrilov and Shilnikov~\cite{GaS73} and Newhouse~\cite{N74}, see also~\cite{G83,PV94}.  This
theorem  {deals with}
the so-called sectionally dissipative case, i.e., when a homoclinic tangency is
associated to a saddle fixed (periodic) point with multipliers $\lambda_1,...,\lambda_n, \gamma$ such that $|\lambda_i|<1, |\gamma|>1$  and
the \emph{saddle value} $\sigma \equiv  |\gamma|\cdot \max\limits_i|\lambda_i|$ is less than 1.
In this case,
bifurcations of the homoclinic tangency lead to the appearance of asymptotically stable  periodic
orbits (periodic sinks). Moreover, in any one parameter general unfolding $f_\mu$, such orbits exist for values of
$\mu$ {forming} an infinite sequence (cascade) of intervals that do not intersect and accumulate to
$\mu=0$.

A very nontrivial extension of this (quite simple) result was made by S.~Newhouse~\cite{N79}, who proved that, for any such one parameter
general unfolding $f_\mu$, there exist intervals  in which there are dense values of $\mu$ such that the corresponding diffeomorphism
$f_\mu$ has a homoclinic tangency. Together with the theorem on cascade of periodic sinks, this implies that the values of $\mu$
 {where} $f_\mu$
possesses {\em infinitely many periodic sinks} {form} residual subsets of these intervals, i.e., subsets {which are}
the intersection of a countable number of open and dense sets. Thus, this {\em Newhouse phenomenon} should be
generic for chaotic sectionally dissipative systems allowing homoclinic tangencies.
Later, the existence of Newhouse regions (where systems with homoclinic tangencies are dense) was proved for any dimension,
\cite{PV94,GST93b,Rom95}, as well as for conservative systems,~{\cite{Dua99, D00}, see also~\cite{MR97a}}.

\begin{figure} [htb]
\centerline{\epsfig{file=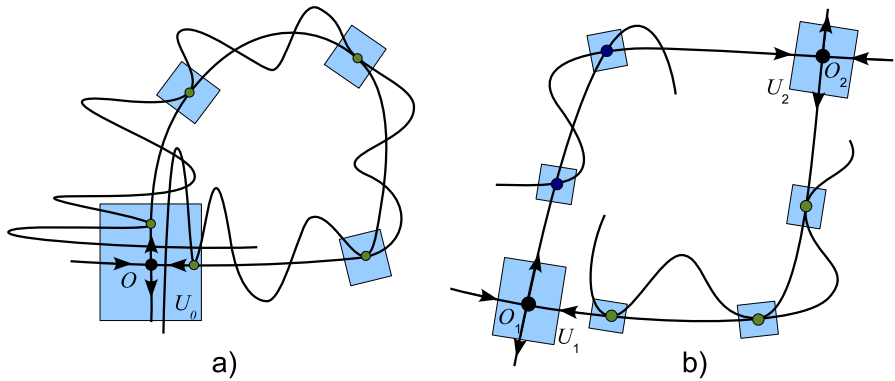,
height=5.5cm }} \caption{{\footnotesize Examples of planar diffeomorphisms  with (a) a (quadratic)
homoclinic tangency at the points of some homoclinic orbit; (b) a nontransversal heteroclinic cycle
containing two saddle fixed points and two selected heteroclinic orbits one of which is nontransversal.
Here it is shown also small neighbourhoods of (a) the nontransversal homoclinic orbit and (b) the heteroclinic cycle which look as a union of small
rectangles. }}
\label{fig:DGG1}
\end{figure}

In principle, the theorem on cascade of periodic sinks admits various extensions even in the case of two-dimensional diffeomorphisms,
see e.g.~\cite{GS07}.

Thus, the main bifurcations of quadratic homoclinic tangencies with $\sigma=1$
were studied
in~\cite{GG00,GG04} where the so-called {\em generalized H\'enon map}
$$
\bar x = y, \; \bar y = M_1 - M_2  x - y^2 + a_1 xy + a_2 y^3
$$
was derived as the normal form for the first return maps. In this map the parameters $M_1$ and $M_2$
are, in fact, the rescaled initial parameters, which control the splitting of the invariant
manifolds and the saddle value,  respectively.
 {The small coefficients $a_1$ and $a_2$
are important: if $a_1\neq 0$, then
the Andronov-Hopf
bifurcation of a fixed point (with multipliers $e^{\pm i\psi}$)  is non-degenerate;  if
$a_1\neq 0$ and  $a_2\neq 0$, then  bifurcations at the strong resonant case $\psi = \pi/2$ are non-degenerate, see
\cite{GKM05,MGon05}.}
Naturally, in the area-preserving case, we have that $|M_2|\equiv 1$,
$a_1\equiv 0$, and
$a_2$ is important again.

Another important extension of the theorem on cascade of periodic sinks concerns the case of
two-dimensional diffeomorphisms having nontransversal heteroclinic cycles, i.e., there are several
saddle fixed (periodic) points which form a cycle due to intersections of their invariant manifolds,
and some of these intersections are nontransversal (see Figure~\ref{fig:DGG1}~b).
If the saddle values of all these  {(saddle)} points are less than 1 (or greater than 1), then the results look quite similar
to the homoclinic case with {saddle value} $\sigma<1$ (respectively, $\sigma>1$), although the intervals of existence of
periodic sinks (sources) can
intersect here~\cite{GMS01}. However,
if the heteroclinic cycle contains at least two saddle points such that $\sigma_1<1$ and $\sigma_2>1$,
then, as it was shown in~\cite{GST97}, a new dynamical phenomenon {called} {\em mixed dynamics} occurs. The
essence of this phenomenon consists in the fact that, first, global bifurcations
of such systems lead to the appearance of infinitely many coexisting hyperbolic periodic points of all
possible types,
i.e., saddle, stable and completely unstable
 (as well as stable and unstable invariant circles~\cite{GSSt02,GSSt06});
and, second, periodic orbits {of one type} are not separated {from the ones of another type}, i.e., the closures
of the sets of  periodic orbits of different types have nonempty intersections.
Note also that
the mixed dynamics
is a generic phenomenon~\cite{GST97}, i.e.,
it takes place on residual subsets in some open (Newhouse) regions.
Especially, this is significant for reversible systems for which the (reversible) mixed dynamics {gives rise} to the
coexistence of {\em infinitely many saddle, attracting, repelling and elliptic periodic orbits}
appearing generically when symmetric homoclinic or heteroclinic structures are involved,
see~\cite{LS04,DGGTS13}. Note that the phenomenon of reversible mixed dynamics is frequently observed in applications,
for example, in a model of coupled rotators~\cite{PT02}, in mechanical models such as  nonholonomic models of a Celtic stone~\cite{GGK13}
and a rubber-body~\cite{K13}, etc.

Concerning {the corresponding} results in the conservative
case, we mention, above all, the well-known theorem of S.~Newhouse~\cite{N77} on the {emergence} of
1-elliptic periodic orbits ({with} only one pair of multipliers {on the unit circle} $e^{\pm i\varphi}$
with $\varphi\neq 0,\pi$) under bifurcations of homoclinic tangencies of
multidimensional symplectic maps. Note that the Newhouse theorem
does not give
answer whether these 1-elliptic orbits are generic.\footnote{The birth of 2-elliptic generic
periodic orbits was proved in~\cite{GST98,GST04} for the case of four-dimensional symplectic maps
with homoclinic tangencies to saddle-focus fixed points. Recall that a periodic orbit is 2-elliptic
if it has two pairs of multipliers $e^{\pm i\phi}$ and $e^{\pm i\psi}$ with $\phi\neq\psi$ and
$\phi,\psi\neq\{0,\pi\}$. The genericity means, in particular, that
$\phi,\psi\neq\{\pi/2,2\pi/3\},\phi\neq 2\psi,\phi\neq 3\psi$, etc.} However, this fact is
{very} important in the two-dimensional case where an 1-elliptic point is elliptic and the
genericity means the KAM-stability of such point. Such a
problem was considered in~\cite{B87,BSh89} when studying bifurcations of three-dimensional
divergence free flows with a homoclinic loop of a saddle-focus equilibrium, and in~\cite{MR97}
{when}  studying bifurcations of two-dimensional symplectic maps with quadratic homoclinic
tangencies. However, a more or less complete description of related bifurcation diagrams (including
the questions about the coexistence of elliptic points of different periods) was not obtained in these papers.
This was done in~\cite{GG09} for the symplectic case.  In this
connection, we note that {in~\cite{GS01,GS03}} it was discovered  that APMs
with quadratic homoclinic tangencies (at $\mu=0$) can possess infinitely many coexisting elliptic
periodic orbits, and, moreover, these orbits have successive periods $k_0,k_0+1,\dots$, starting
at some integer $k_0$. Thus, such APMs {display} the {\em  phenomenon of global resonance}
leading to strict
ordering even in the   structure of elliptic points.\footnote{It is interesting to note that maps
with infinitely many generic elliptic periodic points are dense in the space of {APMs}
with nontransversal heteroclinic cycles,~\cite{GS97,GS00}.
Moreover, the conditions for the existence of such orbits are closely related to certain arithmetic
properties of some numerical invariants ($\Omega$-moduli), whose set includes even
the first Birkhoff coefficients from the normal forms of local maps near saddle points.
{Note also that in \cite{Dua94} it was proved that in the standard map family
there exists a residual set of parameter values for which the map has infinitely many elliptic islands accumulating to
a locally maximal hyperbolic set.
An analogous result for the so-called cyclicity-one elliptic islands was proved recently in~\cite{DeS13}.
}}

In the present paper
the results of~\cite{GG09} and~\cite{GS01,GS03} {are significantly extended}, in particular including into consideration non-orientable
APMs with quadratic homoclinic tangencies. Note that such systems can {be} either planar maps, like the non-orientable conservative
H\'enon map $\bar x =y,\; \bar y = M + x -y^2$, or area-preserving diffeomorphisms on non-orientable {surfaces}.

Our paper is organized as follows.

In {\em Section}~\ref{sec:stofproblem} we state the problem and give the necessary geometric constructions {as well as} the general technical results.
In particular, we formulate, in form of lemmas, several important results on normal forms of  {saddle APMs}
including rather new results (e.g. Lemma~\ref{lemma:FSNF} on the $n$-th order normal form) {for}
the non-orientable case. {In fact, we extend the well-known analytical
Birkhoff-Moser normal form (see formula (\ref{eq:nfBM})) to the finite-smooth case}.

In {\em Section}~\ref{sec3}  we give
a classification of APMs with quadratic homoclinic tangencies according to the type of the semi-local dynamics, i.e., the type of the structure
of the set $N$ of orbits entirely lying in a small neighbourhood $U$ of the contour $O\cup\Gamma_0$, where $O$ is a saddle fixed point and
$\Gamma_0$ is a homoclinic orbit at whose points the manifolds $W^u(O)$ and $W^s(O)$ have a quadratic tangency. Note that $U$ is represented as
a union of a small disk $U_0$ containing the point $O$ and a finite number of disks surrounding those homoclinic points of the orbit $\Gamma_0$ which
do not belong to $U_0$, see Figure~\ref{fig:DGG1}. {Thus}, $U_0$ contains infinitely many points of $\Gamma_0$ lying in $W^s_{loc}(O)$ and
$W^s_{loc}(O)$ and accumulating to $O$. We divide the APMs with quadratic homoclinic tangencies into three classes. In the first class, the set $N$
has always a trivial structure: $N=\{O;\Gamma_0\}$; in the second class, $N$ is nontrivial and allows always a complete description in terms of
the symbolic dynamics,
 see Section~\ref{sec:1-2cl}. In the third class, the structure of $N$ can be both trivial ($N=\{O;\Gamma_0\}$) and nontrivial ($N$ contains
 nontrivial hyperbolic subsets) depending not only on the geometry of the manifolds $W^u(O)$ and $W^s(O)$ near a point of
 homoclinic tangency\footnote{In the case of quadratic homoclinic tangencies such a geometry  is completely {determined} by the signs of
 4 parameters:
 the two multipliers of the point $O$ and {two more} parameters $c$ and $d$ that characterize the mutual position and orientation of the curves $W^u(O)$ and
 $W^s(O)$ near a tangency point, see Section~\ref{sec:globmap}.}, as in the case of tangencies in the first and second classes, but also on other
 invariants of the homoclinic structure. In particular, the structure of $N$ depends on an invariant $\tau$ (see formula~(\ref{eq:tau})) whose
 variation {near $\tau=0$}
(without splitting the tangency) implies that the set $N$ changes the structure.
{See the corresponding
propositions in Section~\ref{sec:3cl}}.

The central part of the paper, {\em Sections}~4, 5 and 6,
is devoted to the study of {the main} bifurcations in parameter families $f_\varepsilon$ of APMs which unfold generally the initial homoclinic
tangency.
First of all, we are interested in bifurcations of the so-called {\em single-round} periodic orbits, i.e., those which pass only once along
the neighbourhood $U(O\cup\Gamma_0)$, see Definition~\ref{definition:p-round}. Every point of such an orbit can be considered as a fixed point of
the corresponding first return map $T_k$ defined in some  {domain}
near a homoclinic point.
In this paper we construct these first return maps as certain compositions $T_k = T_1T_0^k$ of the local map $T_0$ and the global map $T_1$,
where $k$ runs {along} all sufficiently large integer numbers. The local map is, in fact, a conservative saddle map which is defined by orbits of
the diffeomorphism $f_\varepsilon$ on a small neighbourhood (a disk) $U_0\subset U$ containing the point $O_\varepsilon$, thus,
$T_0 = f_\varepsilon |_{U_0}$. The global map $T_1$ is a map acting by the orbits of $f_\varepsilon$ from a small neighbourhood,
say $\Pi^-$, of a homoclinic point, $M^-$,  belonging to $W^u_{loc}\cap U_0$,  to a small
neighbourhood, $\Pi^+$, of another  homoclinic point, $M^+$, belonging to $W^s_{loc}\cap U_0$. Then one can write $T_1 = f_\varepsilon^q |_{\Pi^-}$,
where $q$ is a number such that $M^+= f_0^q(M^-)$, see Figure~\ref{fig:fretm}.

We assume here that the set $\varepsilon$ of the governing parameters include always the parameter $\mu$ of the splitting {between} the manifolds
$T_1(W^u_{loc})$ and $W^s_{loc}$ near the homoclinic point $M^+$. Then we show, see the Rescaling Lemma~\ref{henmain} of
Section~\ref{sec:bifsingleround},   that every first return map $T_k$, for sufficiently large $k$ and small $\mu$, can be written
in the unified rescaled form
\begin{equation}
\bar x = y + o (\lambda^k),\;\; \bar y = M - \nu_1 x - y^2 + \nu_2 \lambda^k y^3 + o(\lambda^k),
\label{eq:1}
\end{equation}
where the rescaled coordinates $(x,y)$ and the parameter $M$ can take values on a ball $\|(x,y,M)\| \leq L_k$, where $L_k\to\infty$ as $k\to\infty$;
$\nu_1$ is the index equal to $+1$ or $-1$ depending on the orientability of the map $T_k$;
$\nu_2$ is some invariant of the homoclinic structure. In fact, the map~(\ref{eq:1}) is a generalized conservative H\'enon map whose bifurcations
are well known. Therefore, we know the bifurcations {that} single-round periodic orbits undergo: the list of these bifurcations coincides
(up to some small details) with the list of bifurcations of fixed points in the map
(\ref{eq:1}), see Sections~\ref{sec:421} and~\ref{sec:422}.  However, this does not mean that we have studied {completely} the homoclinic
bifurcations,
since we need to construct the {\em bifurcation diagram}, which includes not only the list of bifurcations of the first return maps $T_k$, but
also shows a disposition of these bifurcations in the parameter space. Since we are interested in the bifurcations leading to the appearance of
elliptic periodic orbits, first of all, we need to answer the question ``Can elliptic orbits of different periods coexist?''

Our first result on this theme, Theorem~\ref{th:1parcasc} from Section~\ref{sec:fmu}, shows that in
the main case, when $\tau \neq 0$ and {the homoclinic tangency takes place for
$\mu=0$}, in the family $f_\mu$, the intervals of values of $\mu$ corresponding to the
existence of single-round elliptic periodic orbits of period $(k+q)$ (or double-round ones of period $2(k+q)$ when the maps $T_k$ are non-orientable)
are not crossed for different sufficiently large $k$.
However, the ``globally resonant case'' $\tau=0$ is much more interesting. Here, the pointed out intervals can intersect and, moreover, they all
can be {\em nested}, i.e., all the intervals contain the point $\mu=0$. The corresponding results, Theorems~\ref{th:sympinf},~\ref{th:gnorinf}
and~\ref{th:lnorinf}, are presented
in Section~\ref{sec:2parfam} and both formulated and proved in a context of two parameter general unfoldings.

In {\em Section}~\ref{newinv} we prove the invariance of certain quantities which play a
{very} important {role} for the description of the dynamical
phenomena at
the ``global resonance''.

In {\em Section}~\ref{appendix} we prove Lemma~\ref{lemma:FSNF}.

\section{Statement of the problem and preliminary geometric constructions.} \label{sec:stofproblem}

Consider a $C^r$-smooth ($r\geq 3$) area-preserving map $f_0$ satisfying the following conditions.
\begin{itemize}
\item[{\bf A.}]  $f_0$ has a saddle fixed (or periodic) point $O$ with multipliers
$\lambda$ and $\gamma$, where $0<|\lambda|<1<|\gamma|$ and $|\lambda\gamma| = 1$~. Moreover, we
will consider two different
cases: 
\begin{itemize}
\item[{\bf A}.1] the saddle is {\it orientable}, i.e.,  $\lambda=\gamma^{-1}$;
\item[{\bf A}.2] the saddle is {\it non-orientable}, i.e., $\lambda= -\gamma^{-1}$.
\end{itemize}

\item[{\bf B.}] The stable and unstable invariant manifolds of the saddle $O$ have
a quadratic tangency at the points of some homoclinic orbit $\Gamma_0$
(see Figure~\ref{fig:DGG1}(a)).
\end{itemize}

Let ${\cal H}$ be a (codimension one) bifurcation {manifold} composed of area-preserving  $C^r$-maps
close to $f_0$ and such that every map of ${\cal H}$ has a nontransversal
homoclinic orbit close to $\Gamma_0$. Let $f_{\varepsilon}$ be a family of area-preserving
$C^r$-maps that contains the map $f_0$ at $\varepsilon =0$. We suppose that the family depends
smoothly on parameters $\varepsilon = (\varepsilon_1,...,\varepsilon_m)$ and satisfies the
following condition.

\begin{itemize}
\item[{\bf C.}] The family $f_{\varepsilon}$ is transverse to ${\cal
H}$.
\end{itemize}

Let $U$ be a small neighbourhood of the set $O\cup\Gamma_0$. Note that $U$ consists of a small disk
$U_0$ containing the point $O$ and a number of small disks containing those points of $\Gamma_0$
that do not belong to $U_0$ (see Figure~\ref{fig:DGG1}(a)).

\begin{df}
{\em A periodic or homoclinic orbit entirely lying in $U$ is called {\sl p-round}  if it has
exactly $p$ intersection points with any disk of the set $U\backslash U_0$.}
\label{definition:p-round}
\end{df}

\begin{figure}[htb]
\centerline{\epsfig{file=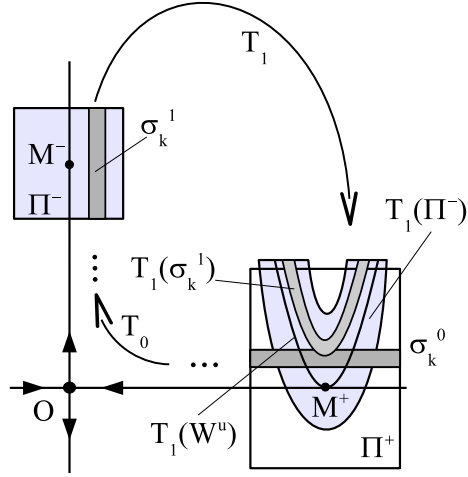, width=7cm }} 
\caption{{\footnotesize{Geometric properties of the local and global maps $T_0$ and $T_1$.} }} \label{fig:fretm}
\end{figure}

In this paper we study bifurcations of {\em single-round ($p=1$) periodic orbits} in the families
$f_\varepsilon$. Note that every point of such an orbit can be considered as a fixed point of the
corresponding \emph{first return map}. Such a map is usually constructed as a superposition $T_k=
T_1T_0^k$ of two maps $T_0\equiv T_0(\varepsilon)$ and $T_1\equiv T_1(\varepsilon)$, see
Figure~\ref{fig:fretm}. The map $T_0$ is called \emph{local map} and it is defined as the
restriction of $f_\varepsilon$ onto $U_0$, i.e., $T_0(\varepsilon)\equiv f_\varepsilon\bigl|{U_0}$.
The map $T_1$ is called \emph{global map} and it is defined as $T_1 \equiv f_\varepsilon^q$  and
acts from a small neighbourhood $\Pi^-\subset U_0$ of some
point $M^-\in W^u_{loc}(O)$ of the orbit $\Gamma_0$ into a neighbourhood $\Pi^+\subset U_0$ of
another
point $M^+\in W^s_{loc}(O)$ of $\Gamma_0$, where $q$ is an integer such that $f_0^q(M^-)=M^+$.
Thus, any fixed point of $T_k$ is a point of a single-round periodic orbit for $f_\varepsilon$ with
period $k+q$. We will study maps $T_k$ for all sufficiently
large integer $k$. Therefore,
it is very important to have good coordinate representations for both
maps $T_0$ and $T_1$.

\subsection{Finite-smooth normal
forms of  {saddle APMs}.} \label{sec:1p1}

The area-preserving map $T_0(\varepsilon)$ has a saddle fixed point $O_\varepsilon$ for all
sufficiently small $\varepsilon$. The simplest form for $T_0$ might be the linear one: $\bar x =
\lambda(\varepsilon) x, \;\;\bar y = \gamma(\varepsilon) y$, where $|\lambda| = |\gamma|^{-1}$,
however, it is non-applicable {since only $C^1$-linearization
can be ensured here}. In the real-analytical case with $\gamma = \lambda^{-1}>0$ we can use
the well-known \emph{Birkhoff-Moser normal form}~\cite{M56}
\begin{equation}
\begin{array}{l}
\bar{x} = B(xy,\varepsilon)x = \lambda(\varepsilon) x\left(1 +
\sum\limits_{i=1}^\infty \beta_i(\varepsilon) \cdot (xy)^i \right),  \\
\bar y = B(xy,\varepsilon)^{-1}y = \lambda^{-1}(\varepsilon) y\left(1 +\sum\limits_{i=1}^\infty
\tilde \beta_i(\varepsilon) \cdot (xy)^i \right),
\end{array}
\label{eq:nfBM}
\end{equation}
where $B(xy,\varepsilon)$ is a real-analytic function (of the variable $u=xy$) well-defined in a
small fixed neighbourhood of $u=0$ for all sufficiently small $\varepsilon$.
The coefficients $\beta_i$ are called {\em Birkhoff coefficients}, the coefficients $\tilde\beta_i$ depend on those in such a way that
$\tilde\beta_i$ is a single-valued functions of
$\beta_1,...,\beta_i$. For example, $\tilde\beta_1 = - \beta_1$, $\tilde\beta_2 = \beta_1^2 -
\beta_2$, etc.

In the smooth case, following
\cite{GS87,GS90,GST07},  we can apply the so-called {\it finitely smooth normal forms}
of the saddle map. {\em The main normal
form} (of the first order) for the saddle map $T_0(\varepsilon)$ is given by the following lemma.
\begin{lm}
\label{lemma:NF1order} {\rm~\cite{GST07}.} Let $T_0(\varepsilon)$ be $C^r$ with $r\geq 3$.   Then
there exists a canonical {$C^r$-change} of coordinates under which
$T_0(\varepsilon)$ takes the form
\begin{equation}
\begin{array}{l}
\bar{x} = \lambda(\varepsilon) x\left(1 + \beta_1(\varepsilon)  xy\right) + o(x^2y), \;\;
\bar y = \gamma(\varepsilon) y\left(1 - \beta_1(\varepsilon)  xy \right) + o(xy^2),
\end{array}
\label{eq:nf1}
\end{equation}
where $\beta_1\equiv 0$ in the case $\lambda\gamma=-1$. The change is $C^{r-2}$ with respect to the
parameters.
\end{lm}

The following lemma {concerns}
{\em the $n$-th order normal form}.
\begin{lm}
\label{lemma:FSNF}
For any integer $n\geq 2$ such that $n < r/2$ (if $r = \infty$, then $n$ is arbitrary), there
exists a canonical $C^{r-2n +1}$
change of coordinates  under which $T_0(\varepsilon)$ takes
the form
\begin{equation}
\begin{array}{l}
\bar{x} = \lambda(\varepsilon) x\left(1 + \beta_1(\varepsilon) \cdot xy + ... +
\beta_n(\varepsilon) \cdot (xy)^n\right) + o(x^{n+1}y^n), \\
\bar y = \gamma(\varepsilon) y\left(1 + \tilde \beta_1(\varepsilon) \cdot xy + ... +
\tilde\beta_n(\varepsilon) \cdot (xy)^n\right) + o(x^{n}y^{n+1}).
\end{array}
\label{eq:nfn}
\end{equation}
Moreover,
in the case $\lambda\gamma=-1$, $\beta_{i} = \tilde\beta_{i}\equiv 0$ for odd $i$.
\end{lm}

\begin{remark}
1) We refer Lemma~\ref{lemma:NF1order} to the paper~\cite{GST07}, where it was proved for the case
$|\lambda\gamma|=1$ and the proof (using canonical transformations) covers also the area-preserving
case, independently, whether the map $T_0(\varepsilon)$ is orientable or not.
We also note that a  version of Lemma~\ref{lemma:NF1order} with the existence of a $C^{r-1}$-change of coordinates
was proved in~\cite{MR97,GS00}
for the symplectic case and in~\cite{GS90} for the case $|\lambda\gamma|=1$.

2) Note that analogous to~(\ref{eq:nfn}) finite-smooth local normal forms for two-dimensional flows having a
saddle equilibrium with eigenvalues $-\rho$ and $\rho$, where
$\rho>0$,
were derived by E.A.~Leontovich
\cite{Le-51,Le-88}. When proving Lemma~\ref{lemma:FSNF} we {follow} closely to the Leontovich method
with some modifications proposed by V.S.~Afraimovich~\cite{Afr84}.
\label{rem:nf}
\end{remark}

One of the advantages of the pointed out normal forms is that they allow {us} to obtain a quite simple
coordinate expression for the
iterations $T_0^k$ for all integer $k$.  Namely, let $(x_i,y_i)\in U_0, i=0,\dots,k-1,$ be
points {such} that $(x_{i+1},y_{i+1}) = T_0(x_{i},y_{i})$. If $T_0$ is linear, then, evidently, $x_k=\lambda^k x_0,\; y_k=\gamma^k y_0$.
We can rewrite the last formula in the so-called {\em cross-form} $x_k = \lambda^k x_0, y_0 = \gamma^{-k}y_k$.
An analogous cross-form exists also in the nonlinear case.
In the case of $T_0$ in the Birkhoff-Moser normal form~(\ref{eq:nfBM}),
the map $T_0^k$ can be written as follows~\cite{GS97}
\begin{equation}
\begin{array}{l}
x_k = \lambda^k x_0 \cdot R^{(k)}(x_0y_k,\varepsilon), \;\;
y_0 = \lambda^{k} y_k \cdot R^{(k)}(x_0y_k,\varepsilon),
\end{array}
\label{03b-BM}
\end{equation}
where
\begin{equation}
\begin{array}{l}
R^{(k)}(x_0y_k,\varepsilon) \equiv 1 + \sum\limits_{i=1}^\infty  \hat\beta_1(k) \lambda^{ik} (x_0y_k)^i
\end{array}
\label{002a}
\end{equation}
and $\hat\beta_i(k)$ are some $i$-th degree polynomials of $k$ with coefficients depending on
$\beta_1,...,\beta_{i}$, in particular,
\begin{equation}
\begin{array}{l}
\hat\beta_1(k,\varepsilon) = \beta_1(\varepsilon) \cdot k, \;\; \hat\beta_2(k,\varepsilon) =
\beta_1^2(\varepsilon) \cdot k^2 + \beta_2(\varepsilon)\cdot k\;.
\end{array}
\label{hatbet}
\end{equation}

In the case of finitely smooth normal forms the following results hold.
\begin{lm} {\rm~\cite{GST07}}
If $T_0$ takes the first order normal form {\rm~(\ref{eq:nf1})}, then $T_0^k$ can be written as
follows
\begin{equation}
\begin{array}{l}
(x_k,y_0) = (\lambda^k x_0,\gamma^{-k} y_k)(1  + \beta_1 k\lambda^{k}x_0y_k) + \lambda^{2k}(P_1(x_0,y_k,\varepsilon),Q_1(x_0,y_k,\varepsilon))
\end{array}
\label{03bnf1}
\end{equation}
where the functions $P_1$ and $Q_1$ are uniformly bounded along with all {their} derivatives up to order $(r-2)$  and the following
estimates take place
for the last {two} derivatives
$$
\|(x_k,y_0)\|_{C^{r-1}} = O(|\lambda|^k),\; \|(x_k,y_0)\|_{C^{r}} = o(1)_{k\to\infty}.
$$
\label{lemma:lmTknf1}
\end{lm}

\begin{lm} {\rm~\cite{GG09}}
If $T_0$ takes the $n$-th order normal form {\rm~(\ref{eq:nfn})}, then $T_0^k$ can be written as
\begin{equation}
\begin{array}{l}
x_k = \lambda^k x_0 \cdot R_n^{(k)}(x_0y_k,\varepsilon) +
\lambda^{(n+1)k}P_n^{(k)}(x_0,y_k,\varepsilon) , \;\;\\
y_0 = \gamma^{-k} y_k \cdot R_n^{(k)}(x_0y_k,\varepsilon) +
\lambda^{(n+1)k}Q_n^{(k)}(x_0,y_k,\varepsilon) \;,
\end{array}
\label{03b}
\end{equation}
where $\displaystyle R_n^{(k)}(x_0y_k,\varepsilon) = 1 + \sum\limits_{i=1}^n  \hat\beta_1(k) \lambda^{ik} (x_0y_k)^i$ ({a finite sum}
of series {\rm~(\ref{002a})}).
The functions $P_n^{(k)} = o(x_0^{n+1}y_k^n),
Q_n^{(k)} = o(x_0^{n}y_k^{n+1})$ are uniformly bounded {in} $k$ along with all {their}
derivatives with respect to $x_0$ and $y_k$ up to the order $(r-2n-1)$ (and up to the order  $(r-2n-1)$ with respect to the derivatives by parameters)
and  $\|(x_k,y_0)\|_{C^{r-2n}} = O(|\lambda|^k),\; \|(x_k,y_0)\|_{C^{r-2n+1}} = o(1)_{k\to\infty}$.
\label{lemma:lm5}
\end{lm}

Lemmas~\ref{lemma:NF1order} and~\ref{lemma:lmTknf1} were proved in~\cite{GST07} (see also~\cite{GS00}). Lemmas~\ref{lemma:FSNF} and~\ref{lemma:lm5}
were proved in~\cite{GG09} for the symplectic case. The
proof of Lemma~\ref{lemma:lm5} for the non-orientable area-preserving case
is practically the same and, therefore, we omit it.  Thus, only Lemma~\ref{lemma:FSNF} is {really} new (in its part related to the
non-orientable case) and we prove it in Section~\ref{appendix}. For the convenience of the reader
we give the
complete proof of this lemma considering both symplectic and non-orientable cases.

\begin{remark}
In our calculations (see e.g. the proof of Lemma~\ref{henmain}) we will also use the {\em second order normal form} for $T_0$
\begin{equation}
\begin{array}{l}
\bar{x} = \lambda x\left(1 + \beta_1  xy + \beta_2 (xy)^2\right) +
O[(|x|^{3}|y|^2(|x| + |y|)]  \;,\;  \\
\bar y = \gamma y\left(1 - \beta_1 xy + \tilde\beta_2
(xy)^2\right) + O[|x|^{2}|y|^{3}(|x|+|y|)]  \;,
\end{array}
\label{eq:nfn2}
\end{equation}
that is given by~(\ref{eq:nfn}) for $n=2$,
where  $\tilde\beta_2 =
\beta_1^2 - \beta_2$; in the case $\lambda\gamma=-1$ we have
that $\beta_1\equiv 0$ and  $\tilde\beta_2 = -
\beta_2$. Then  formulae~(\ref{03b}) for $T_0^k$ {with} $n=2$ can be written as
\begin{equation}
\begin{array}{l}
x_k = \lambda^k x_0(1 + k \beta_1 \lambda^k x_0y_k) + O(k^2\lambda^{3k}), \;\; \\
y_0 = \gamma^{-k} y_k(1 + k \beta_1 \lambda^k x_0y_k)  + O(k^2\lambda^{3k}),
\end{array}
\label{eq:03bnf1}
\end{equation}
moreover, in the case $\lambda\gamma=-1$, {they take} a simpler
form
\begin{equation}
\begin{array}{l}
x_k = \lambda^k x_0 + O(k^2\lambda^{3k}),\;\;
y_0 = \gamma^{-k} y_k  + O(k^2\lambda^{3k}).
\end{array}
\label{03bnf1-1}
\end{equation}
\label{rem:03bnf1}
\end{remark}

\subsection{Properties of the global map $T_1(\varepsilon)$.}\label{sec:globmap}

In what follows, we will use in $U_0$ the local normal form coordinates $(x,y)$  introduced in Section~\ref{sec:1p1}.
In these coordinates both $W^s_{loc}$ and
$W^u_{loc}$ are straightened {out} and, hence, we can put $M^+=(x^+,0),M^-=(0,y^-)$, where $x^{+}>0$ and
$y^{-}>0$.
Then the global map $T_{1}(\varepsilon)\equiv f^q(\varepsilon):\Pi^{-}\rightarrow \Pi^{+}$ can be
written as follows
\begin {equation}
\begin {array}{l}
\overline{x} -x^{+}  =  F(x,y-y^{-},\varepsilon),\;\;
\overline{y} = G(x,y-y^{-},\varepsilon),
\end {array}
\label{eq:t1}
\end{equation}
where $F(0)=0,G(0)=0$. Besides, one has that $G_y(0)=0, G_{yy}(0)=2d\neq 0$ which follows from the
fact (condition B) that at $\varepsilon=0$
the curve $T_1(W^u_{loc}):\{\overline{x} -x^{+} = F(0,y-y^{-},0), \overline{y} = G(0,y-y^{-},0)\}$
has a quadratic tangency with $W_{loc}^s:\{\bar y =0\}$ at $M^+$. When parameters vary this
tangency can split and, moreover, we can introduce the corresponding splitting parameter $\mu\equiv G(0,0,\varepsilon)$.
{By condition \textbf{C}, we can assume that the parameter $\mu$ belongs to the set of
parameters $\varepsilon$.}
Accordingly, we can write the following Taylor expansions for the functions~$F$ and~$G$
\begin {equation}
\begin {array}{rcl}
F(x,y-y^{-},\varepsilon)  &=&
ax + b(y-y^{-}) + e_{20}x^2 + e_{11}x(y-y^{-}) + e_{02}(y-y^{-})^2 + \mbox{h.o.t} \;, \\
G(x,y-y^{-},\varepsilon)  &=& \mu + cx+d(y-y^{-})^{2} +  f_{20}x^2
+ f_{11}x(y-y^{-}) +  f_{30}x^3 \\
&&+ f_{21}x^2(y-y^{-}) + f_{12}x(y-y^{-})^2 + f_{03}(y-y^{-})^3 + \mbox{h.o.t}  \;,
\end {array}
\label{eq:t1ext}
\end{equation}
where the coefficients $a, b, \ldots, f_{03}$ (as well as $x^+$ and $y^-$) depend smoothly on
$\varepsilon$.

In the area-preserving case, the Jacobian $J(T_1)$ of $T_1$ is equal identically to $\pm 1$  for all
values of $\varepsilon$. In particular, this implies that
\begin {equation}
\begin {array}{l}
|bc| \equiv  1\;\;\mbox{and}\;\;
R = 2ad - b f_{11} - 2 c e_{02}\equiv 0,
\end {array}
\label{bcR}
\end{equation}
since $J(T_1)\bigl|_{M^-}= -bc\;\;$ and $\;\displaystyle \frac{\partial J(T_1)}{\partial
y}\Bigl|_{M^-} = R$.
\begin{figure}[htb]
\centerline{\epsfig{file=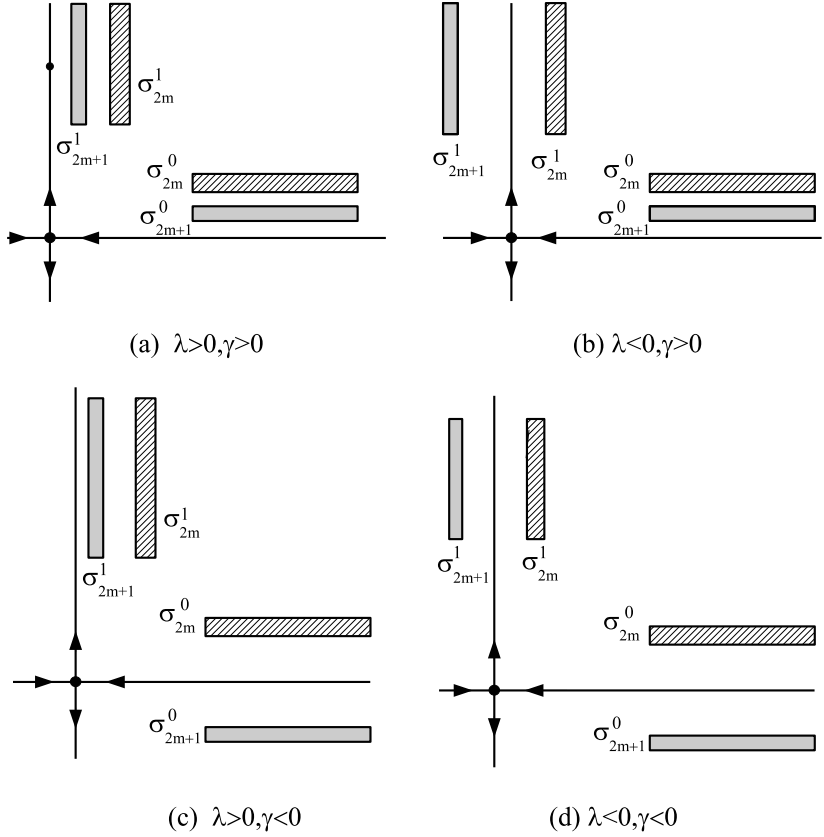,
height=12cm }} \caption{{\footnotesize The strips $\sigma_{k}^{0}$ and $\sigma_{k}^{1}$ for
$\lambda$ and $\gamma$ of various signs.}} \label{fig:strips}
\end{figure}

{We divide the APMs under consideration into three groups:
\begin{itemize}
 \item[(i)]  {the \emph{symplectic maps}}, when $T_0$ and $T_1$ are both orientable ($\lambda \gamma =1$ and $bc =-1$), in this case
condition \textbf{A}.1 holds;
\item[(ii)]  {the \emph{globally non-orientable maps}}, when $T_0$ is orientable and $T_1$ is non-orientable ($\lambda \gamma =1$ and $bc =1$),
i.e., the condition \textbf{A}.1 holds;
\item[(iii)]  {the \emph{locally non-orientable maps}}, when $T_0$ is non-orientable ($\lambda \gamma = -1$), i.e., the
condition \textbf{A}.2 holds.
\end{itemize}
}

Note that in the
case  $\lambda\gamma = -1$, the global map $T_1$ can be orientable ($bc=-1$) or non-orientable ({$bc=1$}) depending on the choice
of pairs of
homoclinic points $M^+$ and $M^-$. If $T_1$ is orientable for a given pair $(M^+,M^-)$, then it is non-orientable for the pairs $(T_0(M^+),M^-)$ or
$(M^+,T_0^{-1}(M^-)$ and again orientable for the pairs $(T_0^2(M^+),M^-)$, $(M^+,T_0^{-2}(M^-)$ or $(T_0(M^+),T_0^{-1}(M^-)$ etc. We will call any
such a pair  of the homoclinic points, when the
corresponding global map $T_1$ is orientable, of the {\em needed type}. For more definiteness,  we will assume that the following condition holds.

\begin{itemize}
\item[{\bf D.}] In the locally non-orientable case, we take always a pair of points $M^+\in W^s_{loc}$ and
$M^-\in W^u_{loc}$ of the homoclinic orbit $\Gamma_0$ which is of the needed type.
\end{itemize}

\subsection{Strips, horseshoes and return maps.}

We assume that the neighbourhoods $\Pi^{+}$ and $\Pi^{-}$ are sufficiently small and fixed, so that
$T_{0}(\varepsilon)(\Pi^{+})\cap\Pi^{+}=\emptyset$ and
$T_{0}^{-1}(\varepsilon)(\Pi^{-})\cap\Pi^{-}=\emptyset$ for all small $\varepsilon$. Then the
domain of definition of the
map from $\Pi^{+}$ to $\Pi^-$ under iterations of
$T_0(\varepsilon)$ consists of infinitely many nonintersecting strips $\sigma_{k}^{0}$ belonging to
$\Pi^+$ and accumulating to $W^s_{loc}\cap\Pi^+$ as $k\to\infty$. Analogously, the range of this
map consists of infinitely many (nonintersecting) strips $\sigma_{k}^{1}= T_0^k(\sigma_k^0)$
belonging to $\Pi^-$ and accumulating to $W^u_{loc}\cap\Pi^-$ as $k\to\infty$. See
Figure~\ref{fig:strips} where a location of the strips is shown for various cases of the signs of $\lambda$ and $\gamma$.

According to~(\ref{eq:t1}) and~(\ref{eq:t1ext}), the images $T_{1}(\sigma_j^{1})$ of the strips
$\sigma_j^{1}$ have a horse-shoe form and accumulate to the curve $l_u=T_1(W^u_{loc})$ as
$j\to\infty$. Note that any orbit staying entirely in $U$ must intersect both the neighbourhoods
$\Pi^-$ and $\Pi^+$ (otherwise, it would not be close enough to $\overline{\Gamma}_0$). Thus, such orbits
must have points belonging to the intersections of the horseshoes $T_{1} (\sigma_j^{1})$ and the strips
$\sigma_{i}^{0}$ for all possible integer $i$ and $j$.

When $\mu$ varies, the location of the horseshoes $T_{1}(\sigma_j^{1})$ changes: they move
together with $T_1(W^u_{loc})$. It implies that the character of mutual intersections of the strips
and horseshoes can change {drastically}. This concerns, in particular, the strips $\sigma_{i}^{0}$
and horseshoes  $T_{1} (\sigma_i^{1})$ with the same numbers $i$. Thus, {when $\mu$ changes},
bifurcations of Smale horseshoes creation/destruction  will occur.
In order to understand these bifurcations
we need to study, first of all, the dynamics of
the map $f_0$, i.e., at $\mu =0$.

For this goal, we study in the next section the semi-local dynamics of the  {APMs} 
with the homoclinic tangencies under conditions \textbf{A} and \textbf{B}.

\section{On a semi-local dynamics of APMs with homoclinic tangencies.}\label{sec3}

In this section we consider  APMs $f_0$ satisfying the conditions \textbf{A} and \textbf{B}. The main goal is to
understand the semi-local dynamics of $f_0$, i.e., the structure of the set $N$ of orbits of the map
$f_0$ entirely lying in a small fixed neighbourhood $U$ of the contour $O\cup\Gamma_0$.  Since $U$
is actually small and contains the neighbourhoods $\Pi^+$ and $\Pi^-$ of the homoclinic points
$M^+$ and $M^-$, we can assume that,  {apart from the orbit $O$}, the set $N$ contains only such
orbits that have intersection points with both $\Pi^+$ and $\Pi^-$. {Equivalently}, for a given
sufficiently large integer $\bar k>0$, we can assume that the neighbourhoods $\Pi^+$ and $\Pi^-$
contain the strips $\sigma_k^0$ and $\sigma_k^1$, respectively, only {for} $k\geq \bar k$.
In other words, we will consider only such {orbits} entirely lying in $U$  whose points from $\Pi^+$
can reach $\Pi^-$ {after} a number of iterations (under $f_0$) that is not less than $\bar k$. We denote
the set of such orbits by $N_{\bar k}\equiv N_{\bar k}(f_0)$.

We study properties of the orbits {in} $N_{\bar k}(f_0)$
using the main analytical result, Lemma~\ref{lm:th01}, (proved in~\cite{GS87,GS95}) as
a tool for detecting the type of intersection between  the horseshoes $T_{1}(\sigma_j^{1})$ and strips
$\sigma_i^{1}$ for various $i,j\geq\bar k$. We assume that
the initial tangency (under the conditions \textbf{A} and \textbf{B}) does not split: this corresponds to $\mu=0$ in
(\ref{eq:t1ext}). We will show that, in this case, the set $N_{\bar k}$ can be completely described
(in terms of symbolic dynamics) {for} an open and dense set\footnote{This is
not the case if  $|\lambda\gamma|\neq 1 $: as it is shown in
\cite{GST93a,GST99}),  systems with infinitely degenerate periodic orbits are
dense among those
with quadratic homoclinic tangencies.} of maps from $\cal H$.
The density should be regarded in the following sense: for a given $\bar k$, in ${\cal H}$ there exists an
open  set of maps whose set $N_{\bar k}$ is completely {determined} and this set becomes dense as $\bar
k\to\infty$:
in fact, one has to exclude only maps satisfying certain conditions {like ``the
invariant $\tau$, given in~(\ref{eq:tau}), is integer number''.}

\subsection{Conditions for the  intersection of horseshoes and strips.} \label{sec:inth+s}

Evidently,
any orbit of $N_{\bar k}$ (except for the orbits  $O$ and
$\Gamma_0$) must have
points belonging to the intersection of the horseshoes $T_{1} (\sigma_j^{1})$ and strips
$\sigma_{i}^{0}$ for some $i,j\geq\bar k$. Thus, the structure of $N_{\bar k}$ depends essentially
on the character of this  intersection.

\begin{df} {\em We say that the horseshoe $T_{1}(\sigma_{j}^{1})$ has
a {\sf regular intersection} with the strip $\sigma_{i}^{0}$ if}
\begin{itemize}
\item[(i)] {\em the set $T_{1}(\sigma_j^{1}) \cap\sigma_i^{0}$ consists of
two connected components $\Delta_{ij}^{1}$ and $\Delta_{ij}^{2}$~;}
\item[(ii)] {\em the map $T_{1}T_{0}^{j}$ restricted to the preimage
$(T_{1}T_{0}^j)^{-1}\Delta_{ij}^\alpha \subset \sigma_j^{0}$ of the component $\Delta_{ij}^\alpha$,
where $\alpha=1,2,$ is a saddle map
(i.e., it is exponentially
contracting along one of coordinates,
$x$,  and expanding along the other one,
$y$), see Figure~\ref{saddlint}.}
\end{itemize}
\label{def:regularinters}
\end{df}

\begin{figure}[htb]
\centerline{\epsfig{file=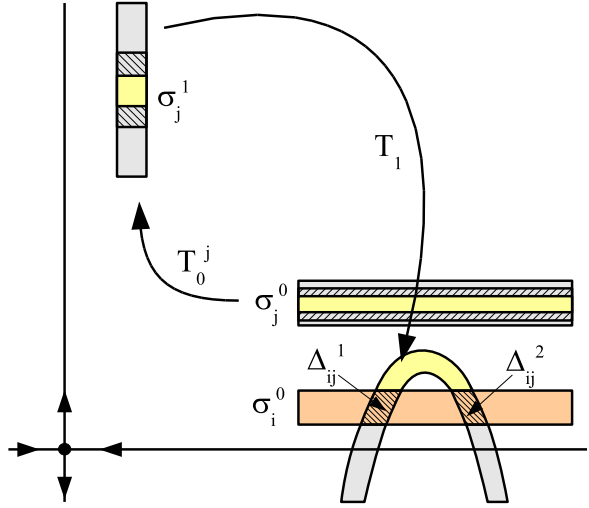,
    height=8cm}} 
\caption{{\footnotesize Regular intersection of the horseshoe $T_1(\sigma_j^1)$
and the strip $\sigma_i^0$ (Definition 2).}} 
\label{saddlint}
\end{figure}

The following lemma provides sufficient conditions characterizing intersections of the
strips and horseshoes.

\begin{lm} {\rm~\cite{GS87,GS95}}
Given {$f_0$ satisfying conditions \textbf{A} and \textbf{B}, with $c$ and $d$
given in~(\ref{eq:t1ext}) at $\mu=0$}, there exist a constant $S_1>0$ and a sufficiently large
integer $\bar{k}$ such that, for any $i,j \geq \bar{k}$, the following assertions hold:

{\rm (i)} If
\begin{equation}
d ( \gamma^{-i} y^{-} - c \lambda^{j} x^{+} ) > S_{1} ( |\lambda|^{i} + |\lambda|^{j} )
|\lambda|^{\bar{k}/2} \;, \label{eq:regular}
\end{equation}
then the horseshoe $T_{1}(\sigma_{j}^{1})$ and strip $\sigma_{i}^{0}$ {intersect}
regularly.

{\rm (ii)} If
\begin{equation}
d ( \gamma^{-i} y^{-} - c \lambda^{j} x^{+} ) < - S_{1} ( |\lambda|^{i} + |\lambda|^{j} )
|\lambda|^{\bar{k}/2}\;, \label{eq:empty}
\end{equation}
then $T_{1}(\sigma_{j}^{1}) \cap \sigma_{i}^{0} = \emptyset$.
\label{lm:th01}
\end{lm}

It is convenient to reformulate this lemma as follows:
\begin{itemize}
\item[(i')] if the horseshoe $T_{1}(\sigma_{j}^{1})$ has an irregular intersection with the strip
$\sigma_{i}^{0}$ (i.e., the intersection $T_{1}(\sigma_{j}^{1}) \cap \sigma_{i}^{0}$ consists of
one connected component or the corresponding maps from Definition~\ref{def:regularinters} are
not saddle), then the following inequality holds
\begin{equation}
|d| |\gamma^{-i} y^{-} - c \lambda^{j} x^{+}|\leq  S_{1} ( |\lambda|^{i} + |\lambda|^{j} )
|\lambda|^{\bar{k}/2}\;, \label{eq:irr1}
\end{equation}

\item[(ii')] if $T_{1}(\sigma_{j}^{1}) \cap \sigma_{i}^{0} \neq \emptyset$, then the following inequality holds
\begin{equation} d ( \gamma^{-i} y^{-} - c \lambda^{j}
x^{+} ) \geq -  S_{1} ( |\lambda|^{i} + |\lambda|^{j} ) |\lambda|^{\bar{k}/2}\;.
\label{eq:irr2}
\end{equation}
\end{itemize}

The inequalities~(\ref{eq:regular})--(\ref{eq:irr2}) have a rather simple geometrical sense. The
strip $\sigma_i^0$ is a narrow horizontal rectangle in $\Pi^+$ having a central line
$y=\gamma^{-i}y^-$, while, the strip $\sigma_j^1$ is a narrow vertical rectangle in $\Pi^-$ having
a central line $x=\lambda^{j}x^+$. By~(\ref{eq:t1}) and~(\ref{eq:t1ext}), the strip $\sigma_j^1$ is
mapped under $T_1$ into a horseshoe which contains a parabola $y= c\lambda^j x^+ + d(x-x^+)^2/b^2$.
The inequality $d ( \gamma^{-i} y^{-} - c \lambda^{j} x^{+} )>0$ means that the straight line
$y=\gamma^{-i}y^-$ and the parabola are crossed in two points, whereas, the inequality $d (
\gamma^{-i} y^{-} - c \lambda^{j} x^{+} )<0$ implies that these curves do not intersect.
 {The small coefficient in the right side of the inequalities appears in order to
take into account a non-zero thickness of the strips and
horseshoes.}

Note that when the regular intersection exists one can establish certain hyperbolic properties of
(area-preserving) maps $f_0$. The following simplest result of  {such} kind relates to the existence
of Smale horseshoes in the first return maps $T_i$.

\begin{prop} {\rm~\cite{GS87}} Given $f_0$ satisfying conditions \textbf{A} and \textbf{B},
{assume that} the strip $\sigma_{i}^{0}$ and the horseshoe
$T_{1}(\sigma_{i}^{1})$  (with the same number) have a regular intersection for which condition
(\ref{eq:regular}) with  $i=j$ holds. Then the first return map $T_i\equiv
T_1T_0^i:\sigma_{i}^{0}\mapsto\sigma_{i}^{0}$ is a Smale horseshoe map, i.e., the map $T_i$ has a
non-wandering set $\Omega_i$ which is the closed invariant uniformly hyperbolic set such that the
system $T_i|_{\Omega_i}$ is conjugate to the topological Bernoulli shift with two symbols.
\label{Sm-hsi}
\end{prop}

\subsection{Three classes of APMs with homoclinic tangencies.}
\label{3classes}

{Clearly}, the structure
of integer solutions of the inequalities
(\ref{eq:regular})--(\ref{eq:irr2}) depends, first of all, on the signs of the parameters $\lambda,
\gamma, c$ and $d$.
This means that the structure of the set
$N_{\bar k}$ depends essentially on the type of the homoclinic tangency.
By this principle, the same as for the case of general diffeomorphisms~\cite{GaS73,GS87,GS86}, we
can divide  quadratic homoclinic tangencies in the area-preserving case
into three classes in the following way:
\begin{itemize}
\item {\em The first class} {is connected} to the tangencies
with
$\lambda>0,\gamma=\lambda^{-1}$, $c<0$ and $d<0$.

\item {\em The second class} {has to do with} the tangencies
with
$\lambda>0,\gamma=\lambda^{-1}$, $c < 0$ and $d > 0$.

\item The tangencies of all other types (with all other combinations of the signs of
$\lambda,\gamma,c$ and $d$) {belong to}  {\em the third
class}.
\end{itemize}

We will say also that a given APM 
is of {the first, second or third class}, if it
has a homoclinic tangency of the first, second or third class, respectively.

Concerning maps of the third class, one can obtain formally 14 different combinations of the signs of coefficients $\lambda$,
$c$ and $d$.
However, some of them can be transformed to the others.

For example, if we choose the pair $M^{+\prime}=T_0(M^+)$ and $M^-$ of homoclinic points instead of $M^+$ and $M^-$,  the new global map
$T_1^\prime = T_1T_0: \Pi^- \to T_0(\Pi^+)$.
Then, by~(\ref{eq:nf1}),~(\ref{eq:t1}) and~(\ref{eq:t1ext}), {it} can be written as follows
$$
\bar x = \lambda x^+ + \lambda a x + \lambda b (y-y^-) + \cdots\;,\;\; \bar y = \gamma c x + \gamma d (y-y^-)^2 + \cdots.
$$
If $\lambda$ is positive, we have that ${x^+}^\prime =  \lambda x^+, c^\prime = \gamma c $ and
$d^\prime = \gamma d$.  If $\lambda$ is negative, first we make  the change $x \mapsto -x$ and then
obtain that ${x^+}^\prime = - \lambda x^+ > 0, c^\prime = - \gamma c $ and $d^\prime = \gamma d$.
Thus, in both cases we can write  that
\begin{equation}
\mbox{sign}\;c^\prime = \mbox{sign}\;(c\lambda\gamma),\;
\mbox{sign}\;d^\prime = \mbox{sign}\;(d\gamma)
\label{cdnew}
\end{equation}
and, hence, in the case $\lambda=\gamma^{-1} <0$, by~(\ref{cdnew}), we  can always assume $d>0$.

Besides, in the area-preserving case, there is no necessity to distinguish $f_0$ and $f_0^{-1}$. Moreover, the following relations
take place for the local and global maps $\tilde T_0=T_0^{-1}$ and $\tilde T_1 = T_1^{-1}$ of  $f_0^{-1}$
\begin{equation}
\tilde\lambda = \gamma^{-1}, \; \tilde\gamma = \lambda^{-1}, \;
\tilde c = \frac{1}{c},\; \tilde d = - \frac{d}{cb^2}.
\label{eq:relf-1}
\end{equation}
and, indeed, by~(\ref{eq:t1}) and~(\ref{eq:t1ext}), the map  $\tilde T_1 = T_1^{-1}$ can be written as
$$
y-y^- = \frac{1}{b}(\bar x - x^+) + \cdots,\;\; x = \frac{1}{c}\bar y - \frac{d}{b^2 c}(\bar x - x^+)^2 + \cdots,
$$
which takes the standard form~(\ref{eq:t1}) {if we interchange the variables $x$ and $y$ as well as the constants $x^+$ and $y^-$.}
Thus, we {we will not} distinguish the combinations $\lambda > 0$, $\gamma>0$, $c > 0$, $d > 0$
and $\lambda > 0$, $\gamma>0$, $c > 0$, $d < 0$ (see Figure~\ref{class3>0}). Also, in the case $\lambda \gamma = -1$, we can set $\lambda<0$,
$\gamma>0$.
\begin{figure}[htb]
\centerline{\epsfig{file=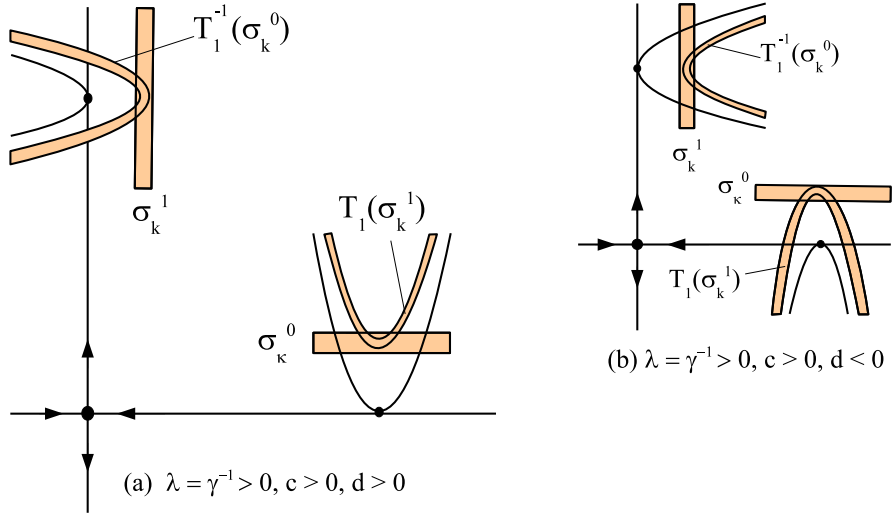, width=13cm
}} \caption{{\footnotesize  {APMs} with a homoclinic tangency of the
third class for $\lambda=\gamma^{-1}>0$: (a) the case $c>0, d>0$; (b) {the case $c>0, d<0$.}
{One can see} the complete analogy
in the geometric structure of the strips and horseshoes in these cases,
especially if we consider the map $f_0^{-1}$ in case (b).
}} \label{class3>0}
\end{figure}

{Therefore, we can reduce the number of different types of homoclinic tangencies of
the third class to the 5 {main} different ones represented in Figure~\ref{fig:table}. We denote by $H_3^i$ , $i = 1, ..., 5$, the
corresponding locally
connected codimension 1 bifurcation surfaces of APMs with homoclinic tangencies. Note that in the locally non-orientable case,
$\lambda \gamma = -1$, we have to consider always, by condition \textbf{D} of Section~\ref{sec:globmap}, pairs $(M^+, M^-)$ of the
homoclinic points of \emph{the needed type} (the corresponding global map $T_1$ is orientable, i.e., bc = -1). This means that the sign of $c$
plays an important r\^{o}le
and, thus, the surfaces $H_3^2$ and $H_3^3$ split into the parts $H_3^{2,1}$ , $H_3^{2,2}$ and $H_3^{3,1}$, $H_3^{3,2}$, respectively, see the
table of Figure~\ref{fig:table}.  {However, we note that the semi-local dynamics of the maps in
$H_3^{2,1}$ and $H_3^{2,2}$ or in $H_3^{3,1}$ and $H_3^{3,2}$ are ``of the same type'' and only
the corresponding invariant sets (e.g. Smale horseshoes $\Omega_i$), when they exist, will be
orientable and non-orientable, respectively, see Figure~\ref{fig:diforient} below.}
}

\begin{figure}[htb]
\centerline{\epsfig{file=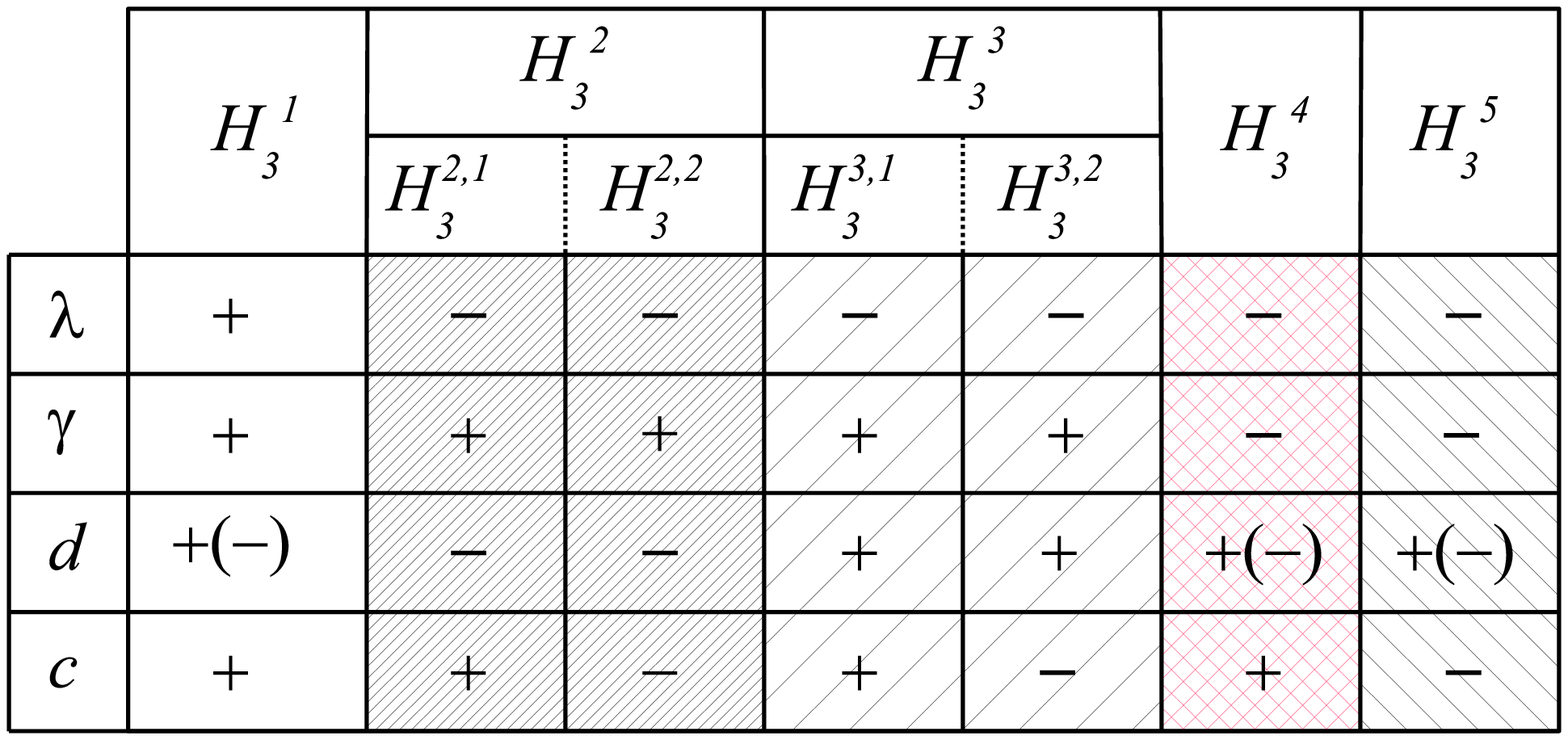, width=11cm
}} \caption{{\footnotesize Five types of APMs 
of the third class.} }
\label{fig:table}
\end{figure}

Note that  {symplectic maps can {arise} in} $H_3^1,H_3^4$ and $H_3^5$,
while maps
in $H_3^2$ and $H_3^3$ are always orientation-reversing. {Besides, non-orientable APMs can have homoclinic tangencies of
all types (in this case, the Jacobian of $T_1$ is negative for homoclinic tangencies of the first and second classes as well as for maps {inside}
$H_3^1,H_3^4$ and $H_3^5$).}

\subsection{Dynamical properties of APMs of the first and second classes.} \label{sec:1-2cl}

For maps of the first class, the inequality~(\ref{eq:empty}) holds for all $i,j\geq\bar k$. It
follows that $T_1(\sigma_j^1)\cap \sigma_i^0 = \emptyset$ for all sufficiently large $i$ and $j$
(see Figure~\ref{class}(a)) which implies the following result:
\begin{prop} {\rm~\cite{GS87}} Let $f_0$ be an 
APM of the first class. Then
there exists such sufficiently large $\bar k$ that  the set $N_{\bar{k}}$  has the trivial
structure: $N_{\bar{k}} = \{O,\Gamma_0\}$.
\label{1cl}
\end{prop}

\begin{figure}[htb]
\centerline{\epsfig{file=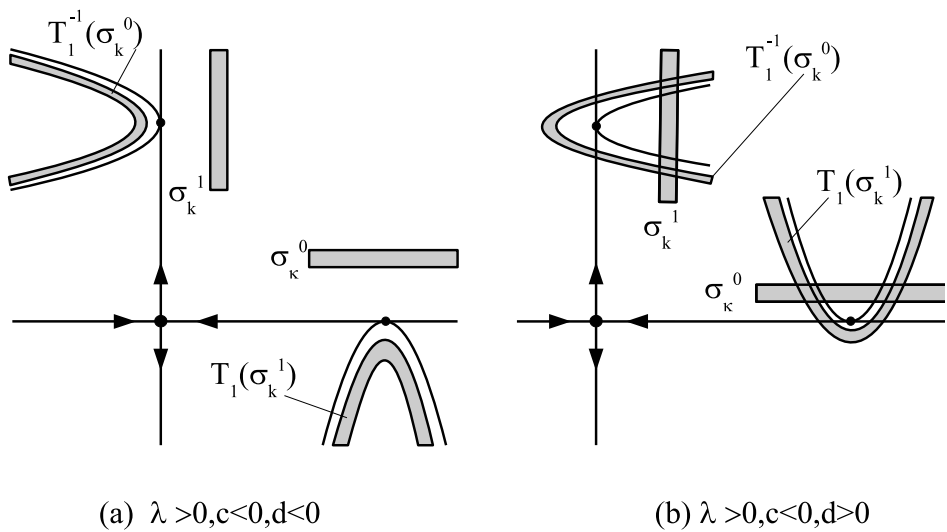, width=13cm
}} \caption{{\footnotesize 
APMs with a homoclinic tangency a) of the first class;
b) of the second class.
}} \label{class}
\end{figure}

For maps of the second class, inequality~(\ref{eq:regular}) holds for all $i,j\geq\bar k$. This means
that all the horseshoes $T_1(\sigma_j^1)$ and strips $\sigma_i^0$ have regular intersection.
Therefore, the set $N_{\bar{k}}$ possesses a non-uniformly hyperbolic structure and all orbits in
$N_{\bar{k}}$, except $\Gamma_0$, are saddle (see also~\cite{GS87}). Moreover, we can give the
exact description of the set $N_{\bar{k}}$. Namely, let $B^3_{\bar{k}+q}$ be a subsystem of the
topological Bernoulli scheme (shift) with three symbols $(0,1,2)$ consisting only of (bi-infinite)
sequences of the form
\begin{equation}
(\ldots,0,\alpha_{s-1},\overbrace {0,\ldots,0,\alpha_s }^{k_s + q}, \overbrace
{0,\ldots,0,\alpha_{s+1}}^{k_{s+1} + q},0,\ldots ),
\label{c8}
\end{equation}
where $\alpha_s\in\{1,2\}$, $k_s\geq\bar k$ for any $s$ and any sequence~(\ref{c8}) does not
contain two neighboring nonzero symbols. We assume also {that} the set of sequences~(\ref{c8}) includes
sequences having strings of infinite length composed by zeros. Let $\tilde
B^3_{\bar{k}+q}$ be a factor-system
resulted from $B^3_{\bar{k}+q}$ {by} {identifying}
homoclinic orbits $(\ldots,0,\ldots,0,1,0,\ldots,0,\ldots)$ and $(\ldots,0,\ldots,0,2,0,\ldots,0,\ldots)$.
 {We denote this orbit by} $\tilde\omega$  {as well as the orbit $(\ldots,0,\ldots,0,\ldots)$ by $\hat O$}.
\begin{prop} {\rm~\cite{GS87,GS01}}
Let $f_0$ be a map of the second class. Then, for any sufficiently large $\bar k$, the system
${f_0}|_{N_{\bar{k}}}$ is topologically conjugate to $\tilde B^3_{\bar{k}+q}$. Moreover, the
conjugating homeomorphism $\hbar$ is such that $\hbar(\Gamma_0) = \tilde\omega$ and $\hbar(O) =
\hat O$.
\label{th2cl}
\end{prop}

\subsection{Dynamical properties of APMs of the third class.} \label{sec:3cl}

Consider the following {number}
\begin{equation}
\displaystyle \tau  =
\frac{1}{\ln|\lambda|}\ln\left|\frac{cx^{+}}{y^{-}}\right|.
\label{eq:tau}\end{equation}
It was shown in~\cite{GS87} that $\tau$ is invariant on
two-dimensional diffeomorphisms with a homoclinic tangency to a
neutral saddle (i.e., condition $|\lambda\gamma|  =1$ holds but the
diffeomorphism itself is not necessarily area-preserving). We will
show (see also~\cite{GS87,GS01,GS03}) that $\tau$ can be
effectively used for the description of dynamics of orbits  from a
small neighbourhood of the homoclinic orbit in the case of APMs of
the third class.

We will use below the notations $[\tau]$ and $\{\tau\}$ for the integer and fractional part of
$\tau$, respectively.

\subsubsection{APMs in $H_3^1$.}

First, we consider  maps in $H_3^1$, i.e., APMs with $\lambda=\gamma^{-1}>0$ and $c>0$ corresponding to the first column of the table of
Figure~\ref{fig:table}. For more definiteness, we assume that $d$ is positive (the case $d<0$ is reduced to this for $f_0^{-1}$,
see Section~\ref{3classes}).

\begin{prop}
Let $f_0\in H_3^1$
and  $\tau \neq 0$. Then there exists an integer number $\bar{k}=\bar{k}(\tau)$  such that $\bar{k}(\tau)\to\infty$ as
$\tau\to 0$ and the following holds. \\
{\rm 1)} If $\tau<0$, then
the set $N_{\bar{k}}$ has a trivial structure:
$N_{\bar{k}}=\{O,\Gamma_0 \}$. \\
{\rm 2)} If $\tau>0$, the set $N_{\bar k}$
contains a nontrivial hyperbolic subset including infinitely many horseshoes $\Omega_i$
for all $i\geq\bar k$.
\label{prH31}
\end{prop}

\begin{proof}
1) Consider the inequality~(\ref{eq:empty}) which  can be rewritten as
$$
\lambda^{i} \left(y^{-} + \frac{S_1}{d}\lambda^{\bar{k}/2}\right) < \lambda^j\left(cx^+ -
\frac{S_1}{d}\lambda^{\bar{k}/2}\right),
$$
since $\lambda=\gamma^{-1}>0$,
$c>0$ and $d>0$. Taking logarithms of both sides, we obtain the inequality
\begin{equation}
j - i + \tau < -  S|\lambda|^{\bar k/2},
\label{l92}
\end{equation}
where $S$ is a positive constant (independent of $i,j$ and $\bar k$). By Lemma~\ref{lm:th01}, if
$i\geq\bar k$ and $j\geq\bar k$ satisfy~(\ref{l92}), then  $T_1(\sigma_j^1)\cap \sigma_i^0 =
\emptyset$. Note that the  inequality~(\ref{l92}) with $\tau <0$ has all the solutions of the form
$i\geq j$.  Since $d>0$, this means that, for all $k\ge\bar{k}$, the horseshoes
$T_{1}\sigma^{1}_{k}$ are  {located} 
above the \emph{own} strips $\sigma^{0}_{k}$, see
Figure~\ref{fig:fig6}(a).
Therefore, all orbits, except for
$O$ and $\Gamma_0$, leave the neighbourhood $U$ under forward
iterations of $f_0$.

\begin{figure}[htb]
\centerline{\epsfig{file=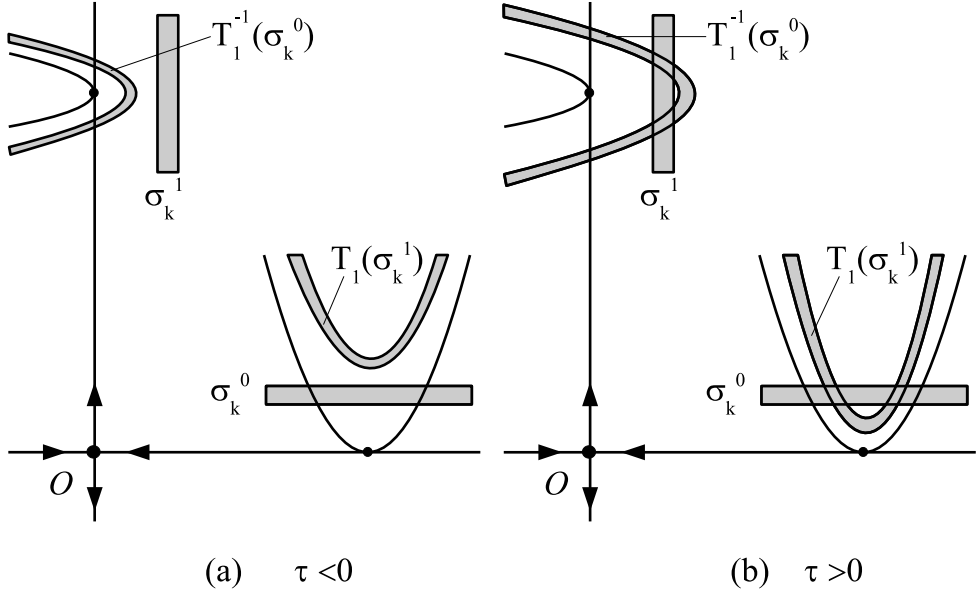, width=14cm
}} \caption{{\footnotesize A creation of the Smale horseshoes $\Omega_k$ at transition from $\tau<0$ to $\tau>0$ in the case
of APMs in $H_3^1$. }} \label{fig:fig6}
\end{figure}

2) Consider the  inequality~(\ref{eq:regular}) which can be written now in the form
\begin{equation}
j - i + \tau >  S|\lambda|^{\bar k/2}. \label{l93}
\end{equation}
When $\tau$ is positive  this inequality, for sufficiently large $\bar k$, has always infinitely many integer
solutions of the form $j\leq i$ including
the solutions  $j = i$.
By Lemma~\ref{lm:th01}, this means that
the horseshoes
$T_{1}\sigma^{1}_{i}$ have regular intersection with the strips $\sigma^{0}_{i}$,
see Figure~\ref{fig:fig6}(b).
By Proposition~\ref{Sm-hsi}, this implies that if $\tau>0$,
the map $f_0\in H_3^1$
has infinitely many Smale horseshoes $\Omega_i$.
\end{proof}

\begin{prop}
Let $f_0\in H_3^1$ and $\tau >0$. If $\tau$ is not integer,
then for some $\bar{k}=\bar{k}(\tau)$, where $\bar{k}(\tau)\to\infty$ as
$\{\tau\}\to 0$, the set $N_{\bar k}$ is completely described in terms of symbolic dynamics.
\label{pr:H31comp}
\end{prop}

\begin{proof}
Taking logarithm of both sides, the inequality~(\ref{eq:irr1})  is rewritten as follows
\begin{equation}
|j - i + \tau| \leq S |\lambda|^{\bar k/2}. \label{l93r}
\end{equation}
If $\tau$ is not an integer,  this inequality has no integer solutions for sufficiently large  $\bar
k=\bar k(\tau)$. Thus, in this case all the strips and horseshoes have either regular or empty
intersections. Taking into account only the regular intersections we obtain the complete
description for $N_{\bar k}$.
\end{proof}

The description of the set $N_{\bar k}$ from Proposition~\ref{pr:H31comp} can be obtained as follows.
If $\tau>0$ and $\{\tau\}\neq 0$,
we can write the inequality~(\ref{l93}) in the equivalent form
\begin{equation}
j - i + [\tau] + \frac{1}{2} > 0, \;\; i,j \geq \bar k,
\label{l93j}
\end{equation}
where $\bar k = \bar k(\tau)$ is sufficiently large (in any case, $\bar{k}(\tau)\to\infty$ as
$\{\tau\}\to 0$). Let ${\cal B}_{[\tau]}^3$ be a subsystem of $\tilde B^3_{\bar k+q}$ containing
the orbits $\hat O$, $\;\tilde\omega$ and all orbits corresponding to the sequences~(\ref{c8}) in
which
\begin{itemize}
\item {\em all the {integer} numbers $k_s$ and $k_{s+1}$
satisfy the inequality}~(\ref{l93j}) with $k_s=j, k_{s+1} = i$ including also  all the pairs
$k_0,k_1$ where
$k_0=\infty, \bar k \leq k_1 < \infty$.
\end{itemize}
Then, using methods of ~\cite{GS87,GS95}, we prove the following result

\begin{prop}
Assume that the hypotheses of Proposition~{\rm\ref{pr:H31comp}} hold. Then the system ${f_0}|_{N_{\bar k}}$ is
topologically conjugate to ${\cal B}_{[\tau]}^3$.
\label{pr:H31descrip}
\end{prop}

\begin{figure}[htb]
\centerline{\epsfig{file=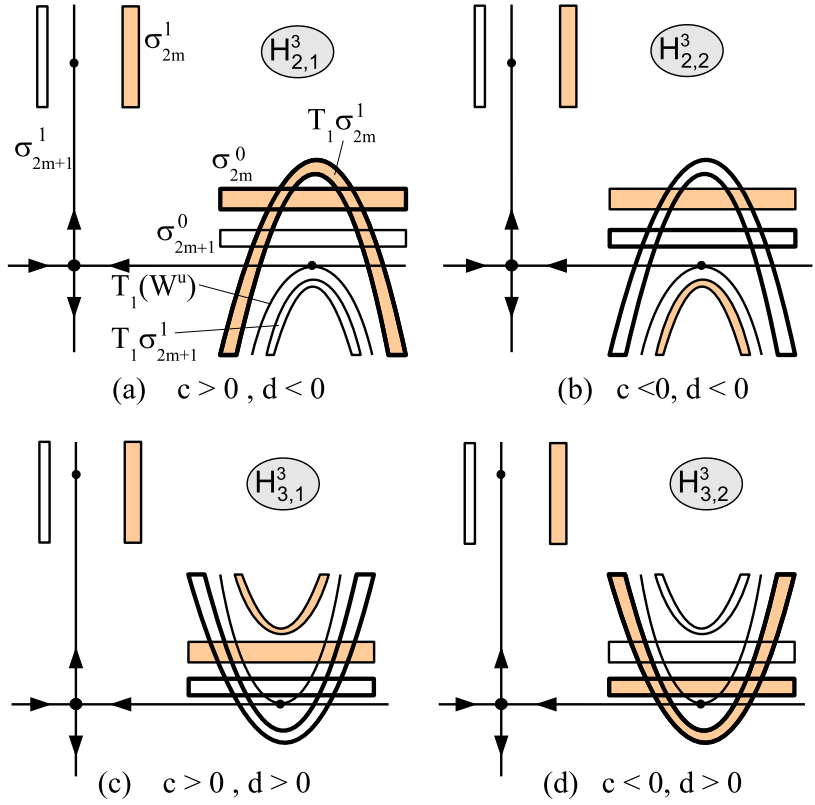, width=12cm
}} \caption{{\footnotesize Geometry of the strips and horseshoes  for the locally non-orientable
case: for APMs maps in $H_{3}^2$ (above) and $H_{3}^3$ (below). It is illustrated that
the Smale horseshoes $\Omega_k$ for the maps in $H_{3,1}^2$ (a) and $H_{3,1}^2$ (b) as well as in
$H_{3,1}^3$ (c) and {$H_{3,2}^3$} (d) {have different orientation}.
} }
\label{fig:diforient}
\end{figure}

\subsubsection{APMs of the third class with negative $\lambda$.}  \label{sec:3clneq}

For maps of the third class with negative $\lambda$,
a {quick glance at the}
Figures~\ref{fig:diforient} and~\ref{fig:f2345} suggests us that the description of $N_{\bar k}$ has {several} peculiarities in
{each of the six cases} under consideration.
Moreover, since $\lambda$
is negative,
it is clear that this description must
include {some} conditions on {the parity of the numbers $i$ and $j$ of the strips
$\sigma_i^0$
and the horseshoes
$T_1(\sigma_j^1)$}.

\begin{figure}
\centerline{\epsfig{file=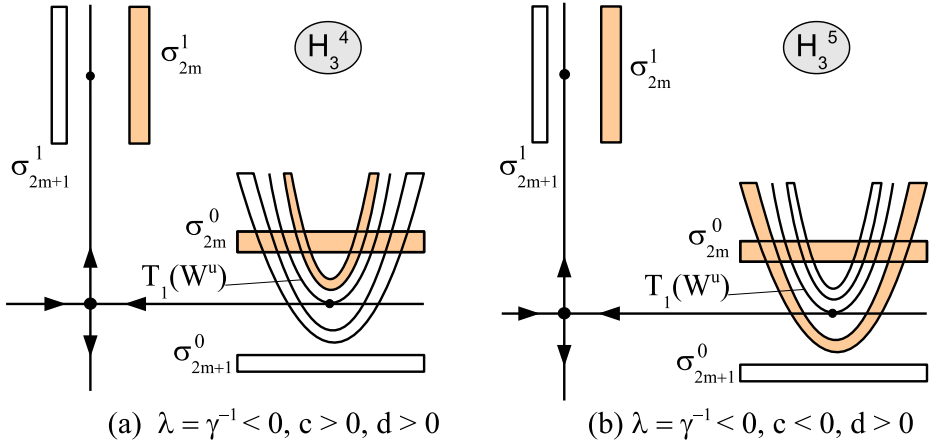, width=14cm
}} \caption{{\footnotesize Geometry of the strips and horseshoes for APMs in (a)  $H_3^4$ and (b) $H_3^5$.
} }
\label{fig:f2345}
\end{figure}

Note that for the maps {inside} $H_3^3$ and $H_3^5$ the set $N_{\bar k}$ has always a nontrivial
structure. In particular,
the following result (which is
quite analogous to the Proposition~\ref{th2cl}) holds.

\begin{prop}{
Let $f_0\in H_3^3, H_3^5$. Then there exists $\bar k$ such that the set $N_{\bar k}$ contains a non-uniformly hyperbolic subset
$\hat N_{\bar k}$ and the following holds
\begin{enumerate}
\item if $f_0\in H_3^{3,1}$, then  $f_0|_{\hat N_{\bar k}}$ is conjugate to $B^{3odd}_{\bar{k}+q}$,
\item if $f_0\in H_3^{3,2}\cup H_3^5$, then $f_0|_{\hat N_{\bar k}}$ is conjugate to $B^{3ev}_{\bar{k}+q}$,
\end{enumerate}
where $B^{3odd}_{\bar{k}+q}\;\;$ (respectively, $B^{3ev}_{\bar{k}+q}$) is a subsystem of the system $B^3_{\bar{k}+q}$ consisting only of
bi-infinite sequences of the form {\rm (\ref{c8})} where all numbers $k_s$ are odd  (respectively, even).}
\label{H25}
\end{prop}

In particular, it follows from Proposition~\ref{H25} that any map $f_0\in H_3^3, H_3^5$ has always infinitely many Smale
horseshoes $\Omega_i$ with numbers $i$ running all sufficiently large integers of a certain
parity:\footnote{ The horseshoes are orientable for $f_0\in H_3^{3,2}$ and and non-orientable for
$f_0\in H_3^{3,1}$, when $f_0\in H_3^5$ the horseshoes are orientable in the symplectic case and
non-orientable in the globally non-orientable case. }
all the numbers are odd for $f_0\in H_{3,1}^3$ and even for $f_0\in H_{3,2}^3\cup H_3^5$,  see Figure~\ref{fig:diforient}(c)--(d) and
Figure~\ref{fig:f2345}(b).

However, {as for the maps in} $H_3^1$, the horseshoe geometry of APMs in $H_3^i, i=2,3,4$
depends essentially on $\tau$ and,
above all, on the sign of $\tau$ that one can see in Figure~\ref{fig:tau(2-4)}. Below we consider
some results analogous to those for maps in $H_3^1$.

\begin{figure}[htb]
\centerline{\epsfig{file=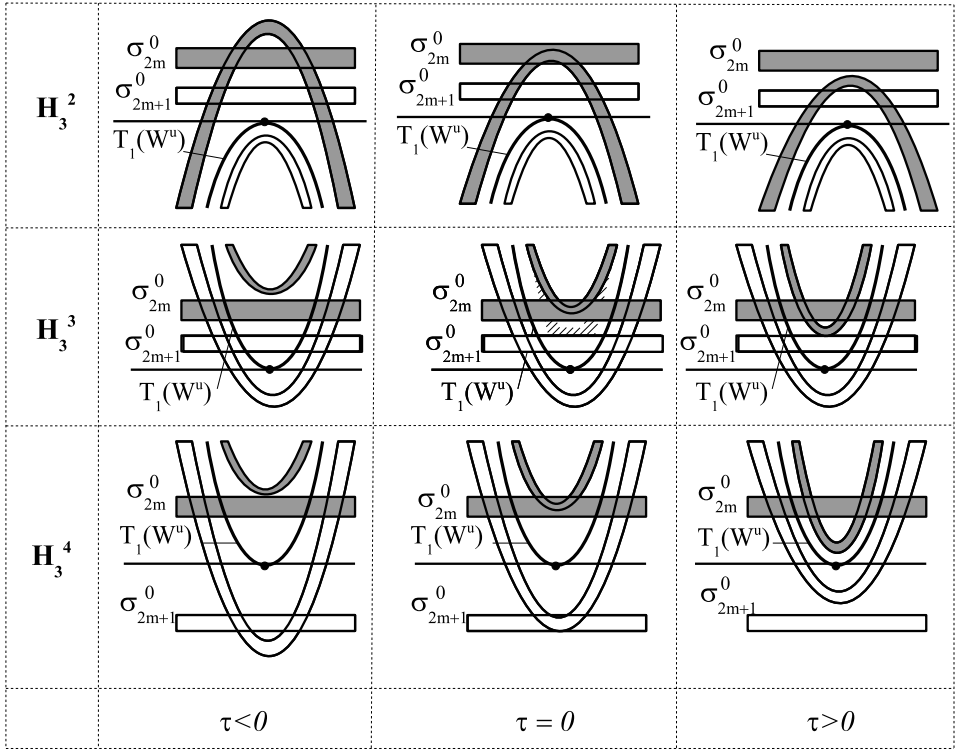, width=14cm
}} \vspace{-1cm}\caption{{\footnotesize A creation/destruction of the new Smale horseshoes $\Omega_{2m}$ and/or $\Omega_{2m+1}$ at transition
from $\tau<0$ to $\tau>0$ in the case of APMs in $H_3^{2,1}$, $H_3^{2,1}$ and {$H_3^{4}$.} }}
\label{fig:tau(2-4)}
\end{figure}

\begin{prop}
Let $f_0\in H_3^2$.
{\rm 1)} If $\tau>0$, then there exists $\bar{k}=\bar{k}(\tau)$, $\bar{k}(\tau)\to\infty$ as
$\tau\to +0$, such that the set $N_{\bar{k}}$ has a trivial structure:
$N_{\bar{k}(\tau)}=\{O,\Gamma_0 \}$.
{\rm 2)}
If $\tau<0$,
the set $N_{\bar k}$ contains infinitely many horseshoes $\Omega_i$, where the numbers $i$ are even for
$f_0\in H_3^{2,1}$ (see Figure~\ref{fig:tau(2-4)}) or odd for $f_0\in H_3^{2,2}$.
\label{H32hyp}
\end{prop}

\begin{proof} Consider the case of $f_0\in H_3^{2,1}$ (the proof for the case of $f_0\in H_3^{2,2}$ is analogous).

1)
As one can see in Figure~\ref{fig:diforient}(a), all horseshoes $T_1(\sigma_j^1)$ with odd $j$ do
not intersect any strip $\sigma_i^0 \subset\Pi^+$: here the inequality~(\ref{eq:empty}) holds for
any $i$ when $j$ is odd (since $\lambda<0,\gamma>0,c>0,d<0$). This means that only the strips
$\sigma_i^{0}$ and the horseshoes $T_1(\sigma_j^1)$ with both even $i$ and $j$ can be responsible
for a nontrivial structure of $N_{\bar k}(f_0)$. Now we assume that $i$ and $j$ are even and
consider the inequality~(\ref{eq:irr2}) that can be rewritten as
\begin{equation}
|\lambda|^{i} \left(y^{-} - |d|^{-1} S_{1} |\lambda|^{\bar{k}/2}\right) \leq |\lambda|^{j}\left(|c|
x^{+} + |d|^{-1} S_{1} |\lambda|^{\bar{k}/2}\right).
\label{l10nemp2}
\end{equation}
Taking logarithm of both sides of~(\ref{l10nemp2}), we
obtain the inequality
\begin{equation}
j - i + \tau \leq - S  |\lambda|^{\bar{k}/2}, \;\; i,j=0(\mbox{mod}2).
\label{l10n2-1}
\end{equation}
If $\tau$ is positive, ~(\ref{l10n2-1}) has {only} (integer) solutions $(i,j)$ such that $i>j$. This
means that any horseshoe $T_1(\sigma_j^1)$ can {only} intersect the strips $\sigma_i^0$ whose number
$i$ is greater than $j$. Since $d<0$, this means that the backward semi-orbit of any point from $\Pi^+$
(except for $M^+$) leaves $U$.

2) Let $\tau$ be negative. Consider now the inequality~(\ref{eq:regular}) for even $i$ and $j$
which
can be rewritten as follows
\begin{equation}
j - i + \tau <  S |\lambda|^{\bar{k}/2},\;\; i,j=0(\mbox{mod}2).
\label{l10n2-12}
\end{equation}

Since $\tau<0$, this inequality  has infinitely many integer solutions of the form $i\leq j$, in
particular, it has
the solutions $i=j$ with even $i$. This implies, by Proposition~\ref{Sm-hsi}, the existence of infinitely
many horseshoes $\Omega_i$, where $i$ runs all sufficiently large  even integer numbers.
\end{proof}

{In the same way as} was done in Proposition~\ref{pr:H31comp}, one can give a complete description of the set
$N_{\bar k(\tau)}$ in the case $f_0\in H_3^2$ when $\tau$ is negative and not integer.
Namely, let ${\cal B}_{[\tau],2}^3$ be a subsystem of $\tilde
B^3_{\bar k+q}$ containing the orbits $\hat O$ and
$\tilde\omega$ and such that, in any sequence~(\ref{c8}),
\begin{itemize}
\item {\em every $k_s$ is even}
\item
{\em the numbers $k_s$ and $k_{s+1}$
 satisfy the inequality}
$ k_s - k_{s+1} + [\tau] + \frac{1}{2} <0, $  including such codings with $k_{-1}<\infty, k_0 =
\infty$.
\end{itemize}

\begin{prop}
Let $f_0\in H_3^{2}$ and $\tau<0$ be not integer. Then, for some $\bar{k}=\bar{k}(\tau)\to\infty$
as $\{\tau\}\to 0$, the system ${f_0}|_{N_{\bar k}}$  is topologically conjugate to ${\cal
B}_{[\tau],2}^3$.
\label{prop:H31-comp23}
\end{prop}

Note that Proposition~\ref{H25} deals with horseshoes $\Omega_i$ which exist always for maps
in $H_3^3$ and $H_3^5$. However, other horseshoes can appear here when varying $\tau$ that the
following result shows for  maps in $H_3^3$.

\begin{prop}
{
If $\tau<0$, then any map $f_0\in H_{3,1}^3$ (respectively, $f_0\in H_{3,2}^3$) has no horseshoes $\Omega_i$ with sufficiently
large even $i$ (respectively, with odd $i$), whereas for any $\tau>0$ infinitely many such horseshoes exists.}
\label{H33}
\end{prop}

The proof is quite analogous to the proof of Proposition~\ref{H32hyp} and we omit it.

Concerning maps in $H_3^4$, we note that they have a specific peculiarity related to the fact that
the value $\tau=0$ is here a ``distinctive switch''
{between} the strips $\sigma_0^i$ and
horseshoes $T_1(\sigma_j^1)$ involved {in the} dynamics, {since they only occur} for even $i$ and $j$ {when}
$\tau>0$ and for odd $i$ and $j$ {when}  $\tau<0$.
Moreover, in this case the dynamics can be trivial only in the case $\tau=0$.
The corresponding result can be formulated as follows.

\begin{prop}
Let $f_0\in H_3^4$ and $\tau \neq 0$. Then there exists an integer $\bar k = \bar k(\tau) \to \infty$ as $\tau \to 0$ {such} that
the following holds.
If $\tau>0$ (respectively, $\tau<0$), the set $N_{\bar k}$ contains infinitely many horseshoes
$\Omega_i$, where $i\geq \bar k$ runs for all sufficiently large even integers (respectively, odd integers). Moreover,
$N_{\bar k}$ does not contain any orbit having intersection points with some strip $\sigma_j^0$ for
odd $j$ (respectively, for even $j$).
\label{H34}
\end{prop}

\begin{proof} Consider inequality~(\ref{eq:regular}) for even $i$ and $j$.
Since  $c>0,d>0,\lambda = \gamma^{-1}<0$, it  can be rewritten as $|d|(|\lambda|^i y^- - |c| |\lambda|^j x^+) > S_{ij}(\bar k)$ or, accordingly, as
\begin{equation}
j - i + \tau > S |\lambda|^{\bar{k}/2}, \;\; i,j=0(\mbox{mod}\;2).
\label{l10n2-3}
\end{equation}
{Clearly}, if $\tau>0$  and $\bar k$ is sufficiently large, this inequality has infinitely
many integer solutions with $j\leq i$ and, in particular, with $i=j$. This implies, by
Proposition~\ref{Sm-hsi}, that  infinitely many horseshoes $\Omega_i$ with even $i$ exist in $N_{\bar k}$.

The inequality~(\ref{eq:regular}) for odd $i$ and $j$ can be
written as $|d|(- |\lambda|^i y^- + |c| |\lambda|^j x^+) > S_{ij}(\bar k)$ or as
$$
j - i + \tau < - S |\lambda|^{\bar{k}/2}, \;\; i,j=1(\mbox{mod}2).
$$
If $\tau<0$, this inequality has infinitely many integer solutions of the form $j\geq i$ including
$j=i$.
This implies that (when $\tau<0$) any map $f_0\in H_3^4$ has infinitely many horseshoes $\Omega_i$
with odd $i$.

For $f_0\in H_3^4$ we have always that  $T_1(\sigma_j^1)\cap\sigma_i^0=\emptyset$ for even $j$ and
odd $i$, since the inequality~(\ref{eq:empty}) holds here (see also Figure~\ref{fig:f2345}(a)).
Let $\tau >0$. We consider some  strip
$\sigma_j^0$ with odd $j$.
Then the horseshoe $T_1(\sigma_j^1)$ can
intersect only those strips  $\sigma_i^0$ with odd numbers $i$ satisfying the inequality~(\ref{eq:irr2}) which is equivalent to
the inequality $j - i + \tau \leq S |\lambda|^{\bar{k}/2}$
for odd $i$ and $j$. Evidently, the last inequality has integer solutions {only} of the form $i>j$,
since $\tau>0$. Thus, there are no points on $\sigma_j^0$ which can return back to $\sigma_j^0$
after forward iterations {by} $f_0$. Moreover, under backward iterations all points of $\sigma_j^0$
leave $U$, since if $T_1^{-1}(\sigma_j^0)\cap \sigma_l^0 \neq\emptyset$, then $l$ is odd and $l<i$.
Thus, if $\tau>0$, only the strips $\sigma_k^0$ with even $k$ can contain points of orbits from
$N_{\bar k}$. The case $\tau<0$ is proved similarly.
\end{proof}

Concerning the maps in $H_3^5$, we note that they are not sensitive to the resonance $\tau=0$ and, moreover, like APMs of the second class,
the set $N_{\bar k}$ admits here a complete description when $\tau\in (-1, +1)$.
Indeed, consider a subsystem ${\cal B}^3_{0,5}$ of $B^3_{\bar k +q}$ containing the orbits $\hat O$ and
$\tilde\omega$ and such that, in any sequence~(\ref{c8}),

\begin{itemize}
\item if $k_s$ is odd, then $k_{s+1}$ is even and such that $k_{s+1} < k_s$;
\item if $k_s$ is even, then either  $k_{s+1}$ is any even {integer} ($\geq \bar k$) or $k_{s+1}$ is odd and such that $k_{s+1} > k_s$.
\end{itemize}

\begin{prop}
Let $f_0\in H_3^{5}$ and $\tau \in (-1,+1)$. Then, for some $\bar{k}=\bar{k}(\tau)\to\infty$
as $|\tau|\to 1$, the system ${f_0}|_{N_{\bar k}}$  is topologically conjugate to ${\cal
B}^3_{0,5}$.
\label{prop:H31-comp5}
\end{prop}

The proof follows immediately from the simple fact that inequality~(\ref{eq:regular}) automatically  holds for the pointed out integers
$j = k_s$ and $i=k_{s+1}$ when $\lambda = \gamma^{-1}, c<0,d>0$ and $|\tau|<1$.

We see that the value $\tau=0$ has a special meaning for the dynamics of  APMs of the third class,
{only maps in $H_3^5$ are not sensitive to the global resonance.} This feature will
effectively come to light {when we study} bifurcations of single-round periodic orbits.

\section{General unfoldings and bifurcations.} \label{sec:bifsingleround}

The main goal of the rest of the paper
is the study of bifurcations of {\it single-round periodic orbits} in one and two parameter {unfolding}
families $f_\varepsilon$ of APMs
(under condition \textbf{C})  {with} the
initial quadratic homoclinic tangency of the map $f_0$ satisfying conditions \textbf{A} and \textbf{B}. Recall that,
by  Definition~\ref{definition:p-round}, every such an orbit has only one
intersection point with $\Pi^+$ (or with $\Pi^-$). Thus, such a  point can be considered as a fixed
point of the corresponding {\it first return map} $T_k\equiv
T_1T_0^k:\sigma_{k}^{0}\mapsto\sigma_{k}^{0}$ with an appropriate integer $k\geq \bar k$. Note that
the integers $k$ can run {among} all values in the set $\{\bar k,\bar k+1,\dots\}$.

Concerning the parameter families we will consider either
one parameter families with the parameter $\varepsilon = \mu$ (general case) or two parameter
families with $\varepsilon = (\mu,\tau)$ (the global resonance case). The latter (two parameter)
family will be used only {to study the} bifurcation of APMs $f_0\in H_3^i$ with $i=1,2,3,4$, which are
{extremely} sensitive to the resonance $\tau=0$, see Section~\ref{sec:3cl}. Recall that $\mu$ is the
parameter of splitting manifolds $W^u(O)$ and $W^s(O)$ with respect to the homoclinic point $M^+$
(see maps~(\ref{eq:t1ext})), and $\tau$ is the invariant {quantity}
given in~(\ref{eq:tau}).

\subsection{The main rescaling lemma.}  \label{sec:mainresclemma}

In principle, one can  study bifurcations of the first return maps $T_k$ written in the initial
coordinates and with the initial parameters $\varepsilon$, using the corresponding formulae for the
local map $T_0$, its iterations $T_0^k$ and the  global map $T_1$  from
Section~\ref{sec:stofproblem}. However, there is a more effective way for studying homoclinic
bifurcations. Namely, we can bring maps $T_k$ to some unified form for all large $k$ using the
so-called rescaling method as it has been done in many papers, see e.g. the papers
\cite{BSh89,MR97,GS97,LS04,DGGTS13} where the rescaling method was applied to the conservative and
reversible cases.  After this, we can study (once) bifurcations in the unified map and
``project'' the obtained results onto the first return maps $T_k$ for various $k$.

The main technical result of this section
is the following

\begin{lm} {\sf [The main rescaling lemma]}\\
For every sufficiently large $k$ the first return map $T_k : \sigma_k^0 \rightarrow \sigma_k^0$ can
be brought, by a linear transformation of coordinates and parameters, to the following form
\begin{equation}
\begin{array}{l}
\bar{X} \; = \; Y  +
k\lambda^{2k}\varepsilon_k^1, \\
\displaystyle \bar{Y}  \;=\; M - \nu_1\; X - Y^2  + \nu_2\;\frac{f_{03}}{d^2}\lambda^{k} Y^3 +
k\lambda^{2k}\varepsilon_k^2 \;,
\end{array}
\label{henon0}
\end{equation}
where $\nu_1 =\; \mbox{{\rm sign}}\;\left( - bc \lambda^k\gamma^k\right)$, $\nu_2 =\; \mbox{{\rm
sign}}\;\left(\lambda^k\gamma^k\right)$;
the functions $\varepsilon_k^{1,2}(X,Y,M)$
are defined on a ball $\|(X,Y,M)\|\leq R$ with arbitrary large $R$ (when $k$ are big)
and are uniformly
bounded in $k$ along with all {their} derivatives up to order $(r-4)$.
Moreover, the following formulae take place for $M$ depending on the saddle $O$ {being} orientable or
not:
if $\lambda\gamma=+1$,
then
\begin{equation}
\begin{array}{l}
\displaystyle M = -d(1+\rho_k^1)\lambda^{-2k}\left(\mu + \lambda^k(cx^+ - y^-)(1 + k\beta_1
\lambda^k x^+y^-)\right) - s_0 + \rho_k^2\;,
\end{array}
\label{mui}
\end{equation}
if $\lambda\gamma=-1$,
then
\begin{equation}
\begin{array}{l}
\displaystyle M = -d(1+\rho_k^3)\lambda^{-2k}\left(\mu + c\lambda^k x^+ - \gamma^{-k}y^-)\right) -
s_0 + \rho_k^4\;,
\end{array}
\label{muinor}
\end{equation}
where
\begin{equation}
s_0   = dx^+(ac + f_{20}x^+) + \frac{1}{2}f_{11}x^+\left(1 + \nu_1 -
\frac{1}{2}f_{11}x^+\right)\; \label{s0}
\end{equation}
and
$\rho_k^{i} = O(k\lambda^k)$ are some small coefficients.
\label{henmain}
\end{lm}

\begin{proof}
We will use the representation of the map $T_0$ in the ``second normal form", i.e., as in~(\ref{eq:nfn2}).\footnote{Of course we
lose a little in the smoothness, since the second order
normal form is $C^{r-2}$ only, see Lemma~\ref{lemma:FSNF}. However, we {gain}
important information on form of the first return maps. On the other hand, our considerations cover
also the $C^\infty$ and real analytical cases.}  Then the map $T_0^k:\;\sigma_k^0 \rightarrow
\sigma_k^1$, for all sufficiently large $k$, can be written in form~(\ref{eq:03bnf1}).

First we consider the case where the saddle $O$ is {\em orientable}, $\gamma^{-1} = \lambda$.
Then, using~(\ref{eq:t1}),~(\ref{eq:t1ext}) and~(\ref{eq:03bnf1}),
we can write the first return  map $T_k : \sigma_k^0 \rightarrow \sigma_k^0$
in the following form
\begin{equation}
\begin{array}{l}
\bar{x}-x^+  =  a\lambda^k x + b(y-y^-) + e_{02}(y-y^-)^2 + \\
\qquad\qquad +O(k|\lambda|^{2k}|x| + |y-y^-|^3 +
|\lambda|^k|x||y-y^-|), \;\; \\ \\
\lambda^{k}\bar{y}\left(1+k\lambda^k\beta_1 \bar{x}\bar{y}\right)
+
k\lambda^{3k} O(|\bar{x}| + |\bar{y}|)\;= \; \\
\qquad \;=\; \mu + c\lambda^k x\left(1+k\lambda^k\beta_1
xy\right) + d(y-y^-)^2 +
\lambda^{2k}f_{20}\;x^2 + \\
\qquad\qquad + \lambda^{k}f_{11}(1+k\lambda^k\beta_1 xy)\;x(y-y^-) +
\lambda^{k}f_{12}\; x (y-y^-)^2 +
f_{03}(y-y^-)^3 + \\
\qquad\qquad + O\left((y-y^-)^4 + \lambda^{2k}|x||y-y^-| +
k|\lambda|^{3k}|x| + k\lambda^{2k}|x||y-y^-|^2\right),
\end{array}
\label{16}
\end{equation}
where we use the Shilnikov cross-coordinates  $x=x_0, y=y_k$, which are very convenient for the construction of return maps near saddles
(see e.g.~\cite{GaS73}).

Below, we will denote by $\alpha_{ki} = O(k\lambda^k)$, $i=0,1,2,...,$ some asymptotically small
coefficients. Now we shift the coordinates
$$
\eta = y - y^-,\; \xi = x - x^+ -  \lambda^k x^+(a+\alpha_k^0),
$$
in order to {vanish} the constant term (independent of coordinates) in the first equation of
(\ref{16}). Thus,~(\ref{16}) is recast as follows
\begin{equation}
\begin{array}{l}
\bar{\xi}  = a\lambda^k\xi + b(1+ {\alpha_{k1}})\eta + e_{02}\eta^2 +
O\left(k\lambda^{2k}|\xi| + |\eta|^3 +
|\lambda|^k|\xi||\eta|\right), \\ \\
\lambda^{k}\bar{\eta}(1 + \alpha_{k2})
+ k\lambda^{2k} O(|\bar{\xi}| + \bar{\eta}^2) +
k\lambda^{3k} O(|\bar{\eta}|)
\;=\; M_1 + c\lambda^k\xi(1  + \alpha_{k3}) + \\
\qquad\qquad + \eta^2(d + \lambda^k f_{12}x^+) + \lambda^{k}\eta(f_{11}x^+ +
\alpha_{k4}) + \lambda^k f_{11}\xi\eta +
f_{03}\eta^3  + \\
\qquad\qquad + O\left(\eta^4 + k|\lambda|^{3k}|\xi| +
k\lambda^{2k}(\xi^2 + \eta^2) + \lambda^{k}|\xi|\eta^2\right),
\end{array}
\label{sys11}
\end{equation}
where
\begin{equation}
M_1 = \mu + \lambda^k(cx^+ - y^-)(1 + k\lambda^k\beta_1 x^+y^-) + \lambda^{2k}x^+(ac + f_{20}x^+) +
O(k\lambda^{3k}).
\label{m1or}
\end{equation}

Now we rescale the variables:
\begin{equation}
\displaystyle \;\xi = -\frac{b(1+\alpha_{k1})(1+\alpha_{k2})}{d +
\lambda^kf_{12}x^+}\lambda^{k} u \;,\; \eta =
-\frac{1+\alpha_{k2}}{d + \lambda^kf_{12}x^+}\lambda^{k} v .
\label{reskk}
\end{equation}
System~(\ref{sys11}) in the coordinates $(u,v)$ is rewritten in the following form
\begin{equation}
\begin{array}{l}
\displaystyle \bar{u}  = v + a\lambda^k u -
\frac{e_{02}}{bd}\lambda^{k} v^2 +
O(k\lambda^{2k}), \\ 
\bar{v} \; = \;
M_2 - \nu_1 u(1 + \alpha_{k5}) - v^2 + \\
\displaystyle \qquad + v(f_{11}x^+ + \alpha_{k6})  -
\frac{f_{11}b}{d}\lambda^{k} uv + \frac{f_{03}}{d^2}\lambda^{k}
v^3 + O(k\lambda^{2k}) \;,
\end{array}
\label{0017}
\end{equation}
where $\nu_1 = -bc$ since the saddle $O$ is orientable for the case under consideration, i.e.,
$\nu_1=1$, if $T_1$ (and also $T_k$) is orientable map, and {$\nu_1 = -1$} if $T_1$ is non-orientable
(the globally non-orientable case), and
$$
\displaystyle M_2 = -\frac{d +
\lambda^kf_{12}x^+}{1+\alpha_{k2}}\lambda^{-2k} M_1.
$$
By the following shift of the coordinates
\begin{equation}
u_{new} = u - \frac{1}{2}(f_{11}x^+ + \alpha_{k6}),\; v_{new} = v
- \frac{1}{2}(f_{11}x^+ + \alpha_{k7})
\label{00170-z}
\end{equation}
with $\alpha_{k6},\alpha_{k7} = O(\lambda^k)$, we  bring the map~(\ref{0017}) {into} the following form
\begin{equation}
\begin{array}{l}
\displaystyle \bar{u}  = v + a\lambda^k u -
\frac{e_{02}}{bd}\lambda^{k} v^2 +
O(k\lambda^{2k}), \\ 
\displaystyle \bar{v}  = M_3 - \nu_1\; u - v^2  - \frac{f_{11}b}{d}\lambda^{k} uv +
\frac{f_{03}}{d^2}\lambda^{k} v^3 + O(k\lambda^{2k}) \;,
\end{array}
\label{00170}
\end{equation}
where
$$
M_3 = M_2 - \frac{f_{11}x^+}{2}(1 + \nu_1) + \frac{(f_{11}x^+)^2}{4}.
$$

Now we make the following linear change of coordinates
\begin{equation}
x \;=\; u + \tilde\nu_k^1\; v \;\;,\;\; y \;=\; v + \tilde\nu_k^2\; u
\;, \label{110c}
\end{equation}
where
\begin{equation}
\displaystyle \tilde\nu_k^1 = \frac{e_{02}}{bd}\lambda^k, \; \tilde\nu_k^2 = a\lambda^k -\nu_1\;
\frac{e_{02}}{bd}\lambda^k. \label{110d}
\end{equation}
Then system~(\ref{00170}) is rewritten as
\begin{equation}
\begin{array}{l}
\bar{x}  =  y + M_3\tilde\nu_k^1 +
O(k\lambda^{2k}), \\ 
\displaystyle \bar{y}  = M_3 - \nu_1\; x - y^2  + a\lambda^k y - \tilde R\lambda^{k} xy +
\frac{f_{03}}{d^2}\lambda^{k} y^3 + O(k\lambda^{2k}) \;,
\end{array}
\label{0018}
\end{equation}
where $ \tilde R =  \left(2a - 2e_{02}\nu_1/bd - bf_{11}/d\right)$. Since $\nu_1 = - bc$, we obtain, by~(\ref{bcR}), that
$\displaystyle \tilde R =\frac{1}{d}(2ad - 2ce_{02} - bf_{11})  \equiv 0 $.
Thus, the map~(\ref{0018}) takes the following form
\begin{equation}
\begin{array}{l}
\bar{x}  =  y + M_3 \tilde\nu_k^1 +
O(k\lambda^{2k}), \\ 
\displaystyle \bar{y}  = M_3 - \nu_1\; x - y^2  + a\lambda^k y + \frac{f_{03}}{d^2}\lambda^{k} y^3
+ O(k\lambda^{2k}).
\end{array}
\label{0019}
\end{equation}

Finally, we make the last shift of the coordinates:
\begin{equation}
X = x - \frac{1}{2}a\lambda^k - \tilde\nu_k^1 M_3,\;\; Y = y -
\frac{1}{2}a\lambda^k,
\label{eq:fshift}
\end{equation}
in order to {cancel} the constant term  in the first equation   and the linear {term} in $y$
in the second equation of~(\ref{0019}). After this, we obtain the final form~(\ref{henon0}) of the map $T_k$
in the rescaled coordinates, where formula~(\ref{mui}) takes place for the parameter $M$.

The proof in the case of the {\em non-orientable saddle} $O$, i.e., when $\gamma^{-1} = - \lambda$,
is quite similar and it is completely the same when $k$ is even, taking into account that
$\beta_1=0$. If $k$ is odd, then $\gamma^{-k} = - \lambda^k$ and the corresponding formulae change. Therefore, we consider the case of odd $k$.
By~(\ref{03bnf1-1}),
 system~(\ref{sys11}) can be written as follows
\begin{equation}
\begin{array}{l}
\bar{\xi}  = a\lambda^k\xi + (b + \hat\alpha_{k1}) \eta + e_{02}\eta^2 +
O\left(\lambda^{2k}(|\xi|  +
|\eta|) + |\eta|^3)\right), \\ 
\bar{\eta}
\;=\;
- M_1\lambda^{-k} - (c + \hat\alpha_{k2}) \xi
-  (d + \lambda^k f_{12}x^+)\lambda^{-k} \eta^2 - (f_{11}x^+ +  \hat\alpha_{k3}) \eta  - f_{11}\xi\eta -
f_{03}\lambda^{-k}\eta^3  + \\ 
\qquad +  O\left(|\lambda|^{-k}\eta^4 + |\eta^3| + |\lambda|^{k}( \xi^2  +
\eta^2)\right),
\end{array}
\label{sys11-1}
\end{equation}
where (recall that $\beta_1=0$ now)
\begin{equation}
M_1 = \mu + c\lambda^kx^+ +  \lambda^{k}y^-  + \lambda^{2k}x^+(ac + f_{02}x^+) + O(\lambda^{3k})
\label{m1or-1}
\end{equation}
and $\hat\alpha_{ki} = O(\lambda^k)$, $i=1,2,...$, are some small coefficients.

After the rescaling
\begin{equation}
\displaystyle \;\xi = \frac{(b+\hat\alpha_{k1}) \lambda^{k}}{d + \lambda^kf_{12}x^+}\; u \;,\; \eta
= \frac{\lambda^{k}}{d + \lambda^kf_{12}x^+} \;v \label{reskk-1}
\end{equation}
system~(\ref{sys11-1}) is rewritten in the
following form
\begin{equation}
\begin{array}{l}
\displaystyle \bar{u}  = v + a\lambda^k u +
\frac{e_{02}}{bd}\lambda^{k} v^2 +
O(\lambda^{2k}), \\ 
\displaystyle \bar{v} \; = \;
M_2 + u(1+\hat\alpha_{k4}) - v^2
- (f_{11}x^+ + \hat\alpha_{k3}) v   -
\frac{f_{11}b}{d}\lambda^{k} uv - \frac{f_{03}}{d^2}\lambda^{k}
v^3 + O(\lambda^{2k}) \;,
\end{array}
\label{0017-1}
\end{equation}
where $\displaystyle M_2 = -(d + \lambda^kf_{12}x^+) \lambda^{-2k} M_1$.

Recall that in this (locally non-orientable) case, by condition \textbf{D}, the homoclinic points $M^+$ and $M^-$ are of the needed type, i.e.,
 the global
map $T_1$ is orientable: $bc = -1$. Then the first return map $T_k$ for
odd $k$ will be non-orientable, i.e., the Jacobian of the map~(\ref{0017-1}) is equal identically to~$-1$.

After the coordinate shift~(\ref{00170-z})
with appropriate  $\alpha_{k6},\alpha_{k7} = O(\lambda^k)$,
the  map~(\ref{0017-1}) is recast as
\begin{equation}
\begin{array}{l}
\displaystyle \bar{u}  = v + a\lambda^k u +
\frac{e_{02}}{bd}\lambda^{k} v^2 +
O(\lambda^{2k}), \\ 
\qquad \displaystyle \bar{v}  = M_3 + u - v^2  -
\frac{f_{11}b}{d}\lambda^{k} uv - \frac{f_{03}}{d^2}\lambda^{k}
v^3 + O(\lambda^{2k}) \;,
\end{array}
\label{00170-v}
\end{equation}
where
$
\displaystyle M_3 = M_2 + \frac{1}{4}(f_{11}x^+)^2 + O(\lambda^k).
$
After the linear change of coordinates~(\ref{110c})--(\ref{110d})
with $\nu_1 =-1$, the system
(\ref{00170-v}) takes the form
\begin{equation}
\begin{array}{l}
\bar{x}  =  y + M_3\tilde\nu_k^1 +
O(\lambda^{2k}),\;\; 
\displaystyle \bar{y}  = M_3 +  x - y^2  + a\lambda^k y  - \frac{f_{03}}{d^2}\lambda^{k} y^3 +
O(\lambda^{2k}) \;.
\end{array}
\label{0018-1}
\end{equation}
Note that in the second equation of~(\ref{0018-1}), the same as in ~(\ref{0019}), the term with $xy$ vanishes.
Note also that the sign before $y^3$ is  opposite to that in~(\ref{0019}).

Finally, by means of the coordinate shift~(\ref{eq:fshift}), we bring the map~(\ref{0018-1}) into the form~(\ref{henon0}).
This completes the proof.
\end{proof}

\subsection{On bifurcations of fixed points in the conservative H\'enon maps.}

The Rescaling Lemma~\ref{henmain} shows that the
limit rescaled form of the first return maps $T_k$ is the conservative
H\'enon map
which is orientable if $\nu_1=1$, and non-orientable if {$\nu_1=-1$} (recall that $\nu_1$ is the Jacobian of $T_k$, i.e.,
$\nu_1=1$ if $T_k$
is orientable and $\nu_1=-1$ if $T_k$ is non-orientable).
Bifurcations of fixed points in these conservative maps are well-known.

\begin{figure}[ht]
\begin{tabular}{l}
\hspace*{35mm}\psfig{file=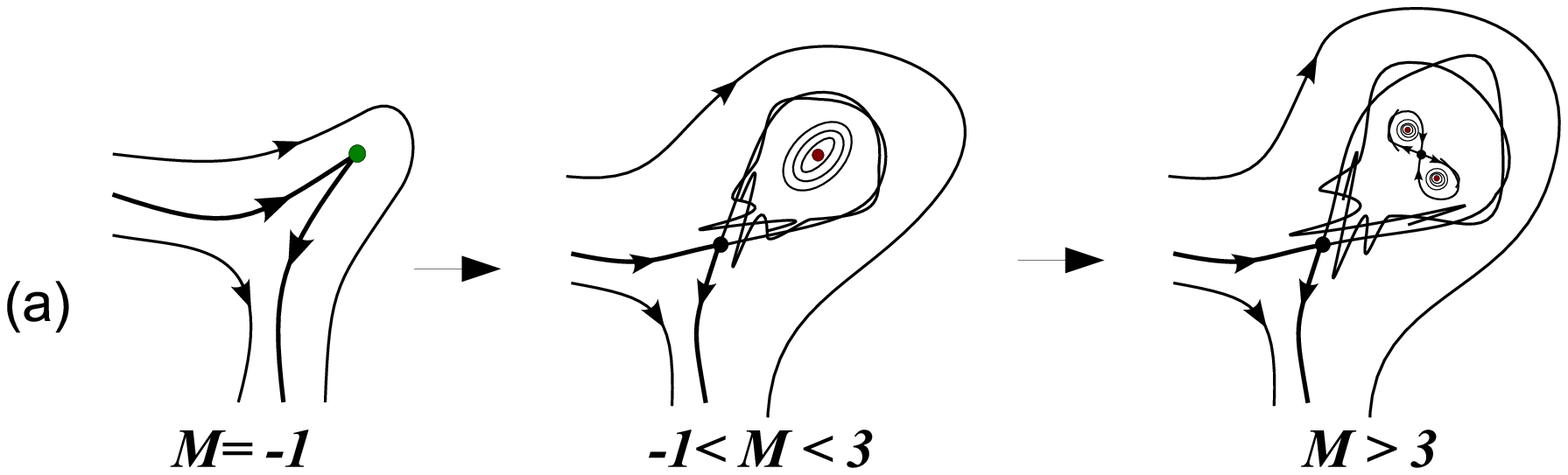, width=10cm}\\
\hspace*{8mm}\psfig{file=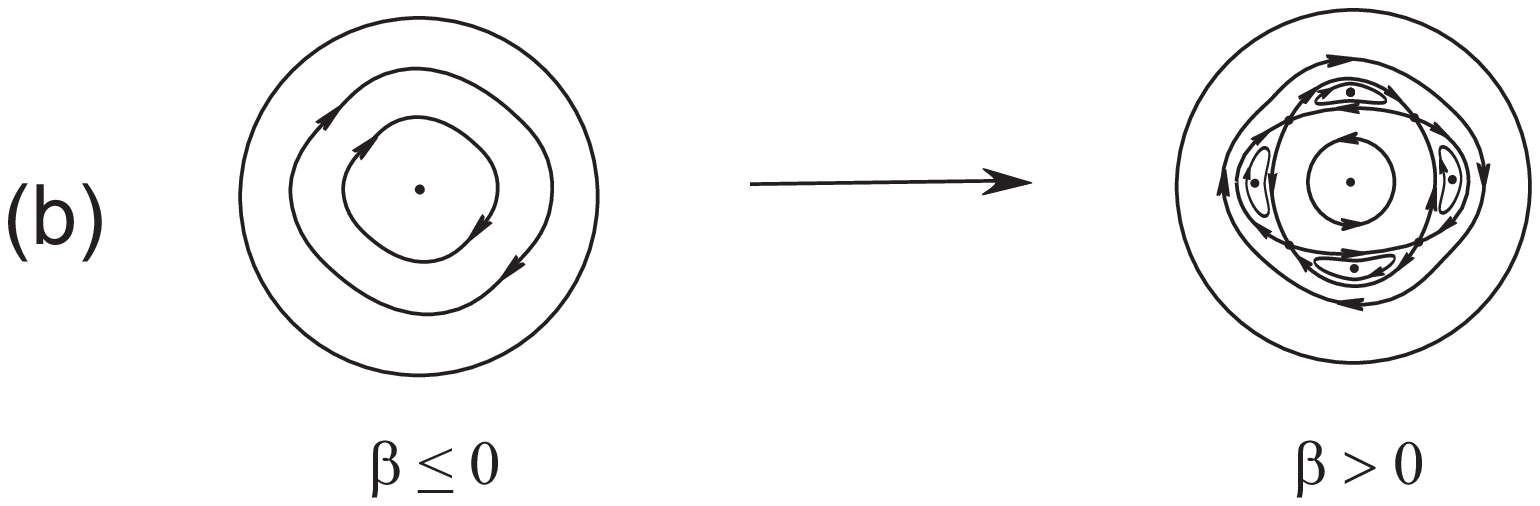,width=60mm,angle=0}
\hspace*{20mm}\psfig{file=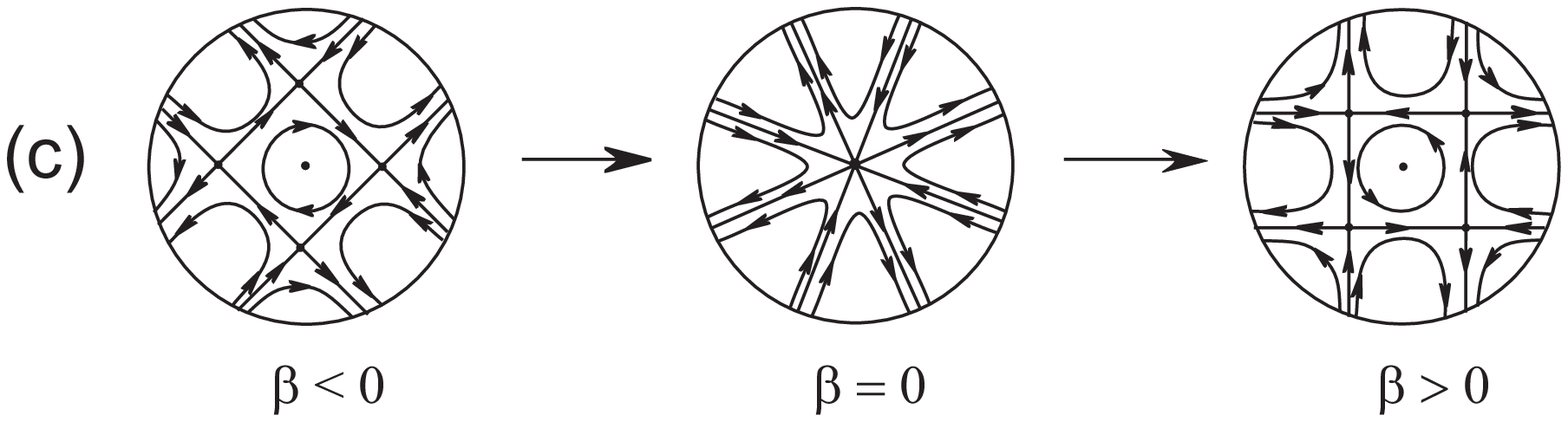,width=75mm,angle=0}
\end{tabular}
\caption{{\footnotesize Bifurcations of fixed points in the orientable conservative H\'enon map:
(a) the main scenario: $M<-1$, there are no fixed points; $M=-1$, a fixed parabolic point appears;
$-1<M<3$, two saddle and elliptic fixed points exist; $M=-3$, the period doubling bifurcation
with the elliptic fixed point; (b)--(c) bifurcations near resonance $1:4$ in the rescaled first return
map~(\ref{henon0}) for the cases (b) $\hat\nu_k= \nu_2 f_{03}d^{-2}\lambda^k>0$ (here the fixed point is always elliptic with multipliers
$e^{\pm i \psi}$)
and (c) $\hat\nu_k<0$ (at $\beta=0$ the fixed point is a saddle with eight separatrices) -- here
$\beta$ is a parameter characterizing a deviation of $\psi$ from $\pi/2$.
}} \label{fig:bifhen}
\end{figure}

\subsubsection{The orientable case.} \label{sec:421}

In the orientable case $\nu_1= 1$, the conservative H\'enon  map
\begin{equation}
\bar{x} \; = \; y, \;\; \bar{y}  \;=\; M - x - y^2 \;
\label{hen0or1}
\end{equation}
has an elliptic fixed point  {with multipliers $e^{\pm i\psi}$, $\psi = \arccos \left( 1 - \sqrt{1+M} \right)$,} for $M\in (-1,3)$.
This point
is generic (KAM-stable)
for all $M\in (-1,3)$ except for two values: $M=0$ for $\psi = \pi/2$ and $M = 5/4$ for $\psi = 2\pi/3$. The {fixed} point for $\psi = 2\pi/3$
is non-degenerate and {always} unstable: it is a saddle with six separatrices. {On the other hand}, the {fixed}
point for $\psi = \pi/2$ is degenerate:
the so-called case ``$A=1$'',
\cite{Arn-Geom,AAIS86}, takes place in the H\'enon map. However, the
 map $T_k$, in reality, has the rescaled form
(\ref{henon0}).
{Therefore}, if the coefficient $f_{03}$ is non-zero, the conservative resonance $1:4$ becomes
non-degenerate~\cite{B87,MGon05}. Namely, the corresponding fixed point will be KAM-stable (of elliptic type) if $f_{03}\lambda^k >0$
and unstable (a saddle with 8 separatrices) if $f_{03}\lambda^k <0$, see Figure~\ref{fig:bifhen}(b),(c).

The conservative H\'enon map has also  fixed parabolic points, for
$M=-1$ with double multiplier $+1$, and for $M=3$ with double multiplier $-1$. The
corresponding conservative bifurcations are non-degenerate. See Figure~\ref{fig:bifhen} for an illustration.

\subsubsection{The non-orientable case.} \label{sec:422}

In the non-orientable case $\nu_1=-1$, the conservative non-orientable H\'enon  map
\begin{equation}
\bar{x} \; = \; y, \;\; \bar{y}  \;=\; M + x - y^2 \;
\label{hen0nor1}
\end{equation}
does not have elliptic fixed
points. However, elliptic {2-periodic orbits} exist here for $M\in (0,1)$.

The non-orientable H\'enon map  
 {(\ref{hen0nor1})} has no fixed points
for $M<0$, it has one fixed point $\bar O(0,0)$ with multipliers $\nu_1=+1,\nu_2=-1$ for $M=0$ and
two saddle fixed points, {$\bar O_1(-\sqrt{M},-\sqrt{M})$ and $\bar O_2(\sqrt{M},\sqrt{M})$,} for
$M>0$. Besides, an elliptic {2-periodic} orbit exists for $0<M<1$, {consisting} of two points
$p_1(-\sqrt{M},\sqrt{M})$ and $p_2(\sqrt{M},-\sqrt{M})$ and has multipliers $e^{\pm i\psi}$, where $\psi = \arccos (1-2M)$.
The value $M=+1$ corresponds to the
period doubling bifurcation of this elliptic orbit. See Figure~\ref{fig:bifhen-1} for an illustration.

\begin{figure}[htb]
\centerline{\epsfig{file=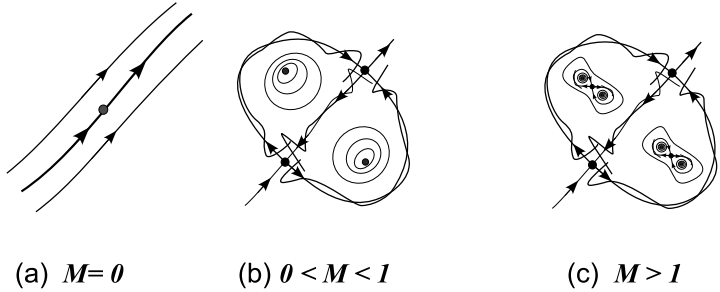,
height=5cm }} \caption{{\footnotesize The main bifurcation scenario in the non-orientable
conservative H\'enon map.}} \label{fig:bifhen-1}
\end{figure}

Note that the elliptic {2-periodic} orbit is generic for all $M\in (0,1)$, except for $M=1/2$ and $M= 3/4$ which correspond
to the strong resonances
$1:4$ and $1:3$,
respectively, and $M = {5}/{8}$ which corresponds to {the cancellation of} the first Birkhoff coefficient
at the cycle $\{p_1,p_2\}$, see~\cite{DGGTS13}.\footnote{However, we do not study here the question of KAM-stability of the corresponding
{2-periodic orbits} with $\psi = \pi/2, 2\pi/3, \arccos {-1/4}$. Note that in the case of the conservative (and orientable) H\'enon map the
corresponding problem was solved in~\cite{B87}, where it was established that the fixed elliptic point of the H\'enon map
at $M=9/16$, with $\psi = \arccos(-1/4)$, has zero first Birkhoff coefficients and nonzero second one, i.e., it is KAM-stable.  }

\section{One parameter cascades of elliptic periodic orbits.}\label{sec:fmu}

In this section we consider the problem on {existence of} cascades of elliptic {periodic orbits} in the case of APMs with
quadratic homoclinic tangencies. {We} establish not only their existence but {we also} analyze questions on the
coexistence of elliptic periodic orbits of different periods. {This} will  allow us to construct the  main
elements of the bifurcation diagrams near homoclinic tangencies for APMs.

Let $f_0$ be an APM with a quadratic homoclinic tangency for which conditions \textbf{A}
and \textbf{B} hold. We embed $f_0$ in a one parameter family $f_\mu$ of APMs which unfolds
generally, under condition~\textbf{C}, the initial quadratic homoclinic tangency.

{\bf 1.} {We first} consider {\em the symplectic case}.

By the Rescaling Lemma~\ref{henmain}, the conservative (orientable)
H\'enon map~(\ref{hen0or1})
is the limit rescaled form for the first return maps $T_k$ with sufficiently large $k$.
In this case,
we can rewrite the relation~(\ref{mui}) as follows
\begin{equation}
\begin{array}{l}
\displaystyle \mu  = - \lambda^k y^- \alpha (1 + k\beta_1 \lambda^k x^+y^-) - \frac{1}{d}(M + s_0)\lambda^{2k} +
O(k\lambda^{3k}),
\end{array}
\label{mui1}
\end{equation}
where
\begin{equation}
s_0   = dx^+(ac + f_{20}x^+) + f_{11}x^+\left(1 - \frac{1}{4}f_{11}x^+\right)
\label{s0symp}
\end{equation}
(that gives
formula~(\ref{s0}) with $\nu_1=1$) and
\begin{equation}
\begin{array}{l}
\displaystyle \alpha = \frac{cx^+}{y^-} - 1.
\end{array}
\label{alph1}
\end{equation}

Since the parabolic fixed points exist in the map~(\ref{hen0or1}) for $M=-1$ and $M=3$, we obtain that
the first return map $T_k$, for sufficiently large $k$, has a fixed point with double multiplier $+1$ and with double multiplier $-1$,
respectively, for such values of $\mu$:
\begin{equation}
\begin{array}{l}
\displaystyle \mu  = \mu_k^+ \equiv - \lambda^k y^- \alpha(1 +
k\beta_1 \lambda^k x^+y^-) - \frac{1}{d}(s_0  -1)\lambda^{2k} +
O(k\lambda^{3k}), \\ \\
\displaystyle \mu  = \mu_k^- \equiv - \lambda^k y^- \alpha (1 + k\beta_1 \lambda^k x^+y^-) -
\frac{1}{d}(s_0 + 3)\lambda^{2k} + O(k\lambda^{3k}).
\end{array}
\label{muk+1n}
\end{equation}
Thus, the values $\mu=\mu_k^+$ and $\mu=\mu_k^-$  correspond to the border points of the interval ${\sf
e}_k$ such that the first return map $T_k$ has an  elliptic fixed point for $\mu\in {\sf
e}_k$.
These elliptic points are generic (KAM-stable) for all values of $\mu\in {\sf
e}_k$, except for {the two values}  $\mu  = \mu_k^{\pi/2}$ and $\mu  = \mu_k^{2\pi/3}$
corresponding, respectively, to $M=0$ and $M=5/4$ in~(\ref{mui1}).

{\bf 2.} {Second, we} consider  {\em the globally non-orientable case}.

In this case, $f_0$ satisfies conditions \textbf{A} and \textbf{B} with $\lambda\gamma = +1$ and also the global map $T_1$ is non-orientable,
i.e., $bc=+1$.
Then, by the Rescaling Lemma~\ref{henmain}, the conservative {\em non-orientable}
H\'enon map~(\ref{hen0nor1})
will be the limit rescaled form for the first return maps $T_k$ with sufficiently large $k$.
Note that formula~(\ref{mui}), describing the relation between the parameters $M$ and $\mu$, takes place here as before, however, since $bc=+1$,
the quantity $s_0$ will be
 {different from the one of} the symplectic case:
\begin{equation}
s_0 = s_0^{nor}  = dx^+(ac + f_{20}x^+) - \frac{1}{4}\left(f_{11}x^+\right)^2
\label{s0nsymp}
\end{equation}
(that gives formula~(\ref{s0}) with $\nu_1=-1$). Note that in this (non-orientable) case we use the notation $s_0^{nor}$ for $s_0$, only for
more definiteness.

Then in the globally non-orientable case, the relation~(\ref{mui}) is rewritten  as follows
\begin{equation}
\begin{array}{l}
\displaystyle \mu  = - \lambda^k y^- \alpha (1 + k\beta_1 \lambda^k x^+y^-) - \frac{1}{d}(M + s_0^{nor})\lambda^{2k} +
O(k\lambda^{3k}),
\end{array}
\label{mui1gl}
\end{equation}
Since the map~(\ref{hen0nor1}) has parabolic-like fixed points for $M=0$ and $M=1$, we obtain, by~(\ref{mui1gl}),   that
the first return map $T_k$, for sufficiently large $k$, has a fixed point with multipliers $+1$ and $-1$ and a {2-periodic orbit} with double
multiplier $-1$, respectively, for such values of $\mu$:
\begin{equation}
\begin{array}{l}
\displaystyle \mu  = \mu_k^{\pm1}  \equiv - \lambda^k y^- \alpha(1 +
k\beta_1 \lambda^k x^+y^-) - \frac{1}{d}s_0^{nor}\;\lambda^{2k} +
O(k\lambda^{3k}),\\ \\
\displaystyle \mu  = \mu_k^{2-} \equiv - \lambda^k y^- \alpha (1 + k\beta_1 \lambda^k x^+y^-) -
\frac{1}{d}(s_0^{nor} + 1)\lambda^{2k} + O(k\lambda^{3k}).
\end{array}
\label{muk+1-1n}
\end{equation}

Thus, the values $\mu=\mu_k^{\pm1}$ and $\mu=\mu_k^{2-}$ correspond to the border points of the interval ${\sf
e}_k^2$ such that the first return map $T_k$ has an  elliptic {2-periodic orbit} for $\mu\in{\sf
e}_k^2$.
These elliptic {2-periodic orbits} are generic (KAM-stable) for all values of $\mu\in {\sf
e}_k^2$, except for three {values} $\mu  = \mu_k^{2,\pi/2}$, $\mu  = \mu_k^{2,2\pi/3}$ and $\mu  = \mu_k^{2,\arccos{-1/4}}$
corresponding, respectively, to $M=1/2$,   $M=3/4$ and $M=5/8$ in~(\ref{mui1}).

{\bf 3.} {Finally, we} consider {\em the locally non-orientable case}.

In this case $f_0$ satisfies conditions \textbf{A} and \textbf{B} with $\lambda\gamma = -1$ and $bc=-1$ by condition~\textbf{D}.
In the case under consideration,
$f_0 \in H_3^2 \cup H_3^3$, see section~\ref{sec:3cl}. Then
the first return maps $T_k$ will be orientable for even $k$ and non-orientable for odd $k$.

First, we consider the case of {\em even} $k$, $k = 2m$.

Then, by the Rescaling Lemma~\ref{henmain}, the conservative
{orientable} H\'enon map~(\ref{hen0or1}) is the limit rescaled form for the first return maps $T_k$, where
$M$ satisfies~(\ref{muinor}) and,
the same as in the symplectic case, formula~(\ref{s0symp}) holds for $s_0$.
We can rewrite  relation~(\ref{muinor}) for even $k$ as
\begin{equation}
\begin{array}{l}
\displaystyle \mu  =  - \lambda^k y^- \alpha - \frac{1}{d}(s_0 +M)\lambda^{2k} +
O(k\lambda^{3k}),\;\; k= 2m
\end{array}
\label{mui1norev}
\end{equation}
and  obtain  that
the first return map $T_k$, for sufficiently large {\em even} $k$, has a fixed point with double multiplier $+1$ and with double multiplier $-1$,
respectively, for such values of $\mu$:
\begin{equation}
\begin{array}{l}
\displaystyle \mu  = \tilde\mu_k^{+}  \equiv - \lambda^k y^- \alpha - \frac{1}{d}(s_0 -1)\lambda^{2k} +
O(k\lambda^{3k}), \\ \\
\displaystyle \mu  = \tilde\mu_k^{-} \equiv - \lambda^k y^- \alpha -
\frac{1}{d}(s_0 + 3)\lambda^{2k} + O(k\lambda^{3k}),\;\; k= 2m.
\end{array}
\label{muk+1-1ln}
\end{equation}

Thus, the values $\mu=\tilde\mu_{2m}^{+}$ and $\mu=\tilde\mu_{2m}^{-}$
correspond to the border points of the interval $\tilde{{\sf e}}_{2m}$ such that the first return map $T_{2m}$ has an  elliptic fixed point
for $\mu\in\tilde{{\sf
e}}_{2m}$.
These elliptic fixed points are generic (KAM-stable) for all values of $\mu\in \tilde{{\sf
e}}_k$, except for {the two values} $\mu  = \tilde\mu_k^{\pi/2}$ and $\mu  = \tilde\mu_k^{2\pi/3}$
corresponding, respectively, to $M=0$ and    $M=5/4$ in~(\ref{mui1norev}).

In the case where  $k$ is {\em odd}, $k=2m+1$, by the Rescaling Lemma~\ref{henmain}, the conservative
non-orientable H\'enon map~(\ref{hen0nor1}) becomes the limit rescaled form for the first return maps $T_k$, where $M$ satisfies~(\ref{muinor}).
However, for odd $k$, as in the globally non-orientable case, formula~(\ref{s0nsymp}) holds for $s_0=s_0^{nor}$.
Thus, when $k$ is odd, since $\lambda^k = - \gamma^{-k}$ and $\lambda<0,\gamma>0$ for $f_0\in H_3^j,\;j=2,3,$ relation~(\ref{muinor}) can be
rewritten as
\begin{equation}
\begin{array}{l}
\displaystyle \mu  =  - \lambda^k y^-(\alpha+2) - \frac{1}{d}(s_0^{nor} +M)\lambda^{2k} +
O(k\lambda^{3k}),\;\; k= 2m+1
\end{array}
\label{mui1norodd}
\end{equation}
where $\displaystyle \alpha +2 = \frac{cx^+}{y^-} +1$, see~(\ref{alph1}). Then
we obtain  that
the first return map $T_k$, for sufficiently large {\em odd} $k$, has a fixed point with multipliers $+1$ and $-1$ and {an elliptic 2-periodic
orbit}
with double multiplier~$-1$, respectively, for such values of $\mu$:
\begin{equation}
\begin{array}{l}
\displaystyle \mu  = \tilde\mu_k^{\pm 1}  \equiv - \lambda^k y^- (\alpha+2)  - \frac{1}{d}s_0^{nor}\;\lambda^{2k} +
O(k\lambda^{3k}),\\ \\
\displaystyle \mu  = \tilde\mu_k^{2-}  \equiv - \lambda^k y^- (\alpha+2)  - \frac{1}{d}(s_0^{nor}+1)\lambda^{2k} + O(k\lambda^{3k}),\;\; k= 2m+1.
\end{array}
\label{muk+1-1lln}
\end{equation}

Thus, the values $\mu=\tilde\mu_{2m+1}^{\pm 1}$ and $\mu=\tilde\mu_{2m+1}^{2-}$
correspond to the border points of the interval $\tilde{{\sf
e}}^2_{2m+1}$ such that the first return map $T_{2m+1}$ has an elliptic {2-periodic orbit} for $\mu\in\tilde{{\sf
e}}^2_{2m+1}$.
These elliptic {2-periodic orbits} are generic (KAM-stable) for all values of $\mu\in \tilde{\sf
e}_k^2$, except for {the three values} $\mu  = \tilde\mu_k^{2,\pi/2}$, $\mu  = \tilde\mu_k^{2,2\pi/3}$ and $\mu  = \tilde\mu_k^{2,\arccos{-1/4}}$
corresponding, respectively, to $M=1/2$,   $M=3/4$ and $M=5/8$ in~(\ref{mui1norodd}).

Now we collect the results obtained in this section in the following theorem.

\begin{theorem} {\bf [On a one parameter cascade of elliptic points]} \\
Let $f_\mu$ be a one parameter family of APMs {satisfying conditions \textbf{A}, \textbf{B} and \textbf{C}}.  Then in any segment
$[-\mu_0,\mu_0]$
of $\mu$, there exist infinitely many intervals (i) ${\sf e}_k$, $k=\bar k,\bar k+1,\dots$, in the symplectic case; (ii) ${\sf e}^2_k$,
$k=\bar k,\bar k+1,\dots$, in the globally non-orientable case;
(iii) $\tilde{{\sf e}}_{2m}$ and $\tilde{{\sf e}}^2_{2m+1}$, $m=\bar m,\bar m+1,\dots$, in the locally non-orientable case
\footnote{  with border points given by formulae (i)~(\ref{muk+1n}); (ii)~(\ref{muk+1-1n}); (iii)~(\ref{muk+1-1ln}) with $k=2m$ and, respectively,
(\ref{muk+1-1lln})  with $k=2m+1$} such that

{\rm 1.a)} $f_\mu$ has a {\rm single-round elliptic periodic orbit} (of period $k+q$) either at (i) $\mu \in {\sf e}_k$  or  at
(iii) $\mu \in \tilde{{\sf e}}_{2m}$, where $k=2m$;

{\rm 1.b)} $f_\mu$ has a {\rm double-round elliptic periodic orbit} (of period $2(k+q)$ corresponding to
a {$2$-periodic point} of $T_k$) either at (ii) $\mu \in {\sf e}^2_k$  or  at (iii) $\mu \in \tilde{{\sf e}}^2_{2m+1}$, where $k=2m+1$.

{\rm 2)} These elliptic orbits are generic (KAM-stable) at almost all values of the parameter $\mu$ in the pointed out intervals (except for
those values which correspond to the points with multipliers $e^{\pm i\psi}$, where $\psi = \pi/2, 2\pi/3$ in all cases and $\psi = \arccos(-1/4)$
in the cases (ii) as well as (iii) with $k=2m+1$).

{\rm 3)} If $\alpha\neq 0$  {($\alpha$ is given in~(\ref{alph1}))}, the intervals ${\sf e}_k$ and ${\sf e}^2_k$ as well as
$\tilde{{\sf e}}_{k}$
with $k=2m$ do not intersect for
different large $k$; if $\alpha\neq -2$, the intervals $\tilde{{\sf e}}^2_{2m+1}$ do not intersect for different large $m$.

\label{th:1parcasc}
\end{theorem}

For the symplectic case, Theorem~\ref{th:1parcasc}  was established in~\cite{GG09}. {Indeed}, also in~\cite{MR97} the existence of cascade of elliptic
periodic orbits was proved, however, the problem on their coexistence was not considered.

We notice that the cases of global resonances, {that is,} maps $f_0$ with $\alpha =0$ as well as with $\alpha=-2$ for  $f_0\in H_3^{2,2},H_3^{3,2} $, are
of special interest, since elliptic periodic orbits (even infinitely many of them) can coexist. The related {phenomena} will be considered in
next section.

\section{On bifurcations of single-round periodic orbits in two parameter general unfoldings.} \label{sec:2parfam}

In principle, in Section~\ref{sec:fmu} we have studied bifurcations of single-round periodic orbits for APMs with quadratic homoclinic tangencies.
Only
some small questions remain  {unanswered}, e.g. bifurcations of strong resonances in the non-orientable cases etc. However,
we have not yet constructed more or
less complete bifurcation diagrams which include not only the results related to bifurcations of
concrete (single-round) periodic orbits but also,
{what is more important}, the problem of coexistence of
(elliptic) periodic orbits and, correspondingly, the problem of the order of
bifurcations. When the intervals of existence of elliptic orbits of different periods do not intersect, we can assume that the bifurcation
problem has been solved completely (up to some small details). However, as we saw, in the case under consideration, these intervals (from
Theorem~\ref{th:1parcasc})
may intersect and, moreover, they intersect undoubtedly when values of the parameter $\alpha$ varies near zero
(or near $\alpha = -2$ when $f_0\in H_3^{j,2}, j=2,3$). Thus, the cases of the maps $f_0$ with $\alpha =0$ ({respectively,} $\alpha=-2$)
are special and it is
clear that {these cases require} to consider {at least} two parameter families including the parameters $\mu$ and $\alpha$
as the governing ones.

Note that symplectic two-dimensional maps with quadratic homoclinic tangencies with $\alpha =0$ were studied in the papers by Gonchenko and
Shilnikov~\cite{GS01,GS03}, where the  phenomenon of ``global resonance'' was discovered. This phenomenon consists in the fact that
{for the values}
$\mu=0,\alpha=0$ the map $f_0$ can have infinitely many single-round elliptic periodic orbits of {\em all successive periods } starting at some
number. This unusual dynamical property takes place when $-3 < s_0 <1$, i.e., this is a codimension 2 effect.
This phenomenon can be considered as very interesting from various points of view, since elliptic orbits play an important r\^{o}le in conservative
dynamics and applications (including Celestial Mechanics,~\cite{Arn63}, smooth billiards,~\cite{TR},  etc).

In the present paper we also observe this phenomenon but from a different point of view: we study  bifurcation diagrams (for single-round periodic
orbits) in two parameter families, say $f_{\mu,\alpha}$, and show that all the domains of  existence of single-round elliptic orbits can
contain the
point $\mu=0,\alpha=0$. Moreover, we show that the phenomenon of global resonance takes place also in the non-symplectic case.

Note that the invariants $\tau$, see~(\ref{eq:tau}) and Section~\ref{sec3}, and $\alpha$ are closely related:
$$
\tau = \frac{1}{\ln|\lambda|} \ln\left|\alpha +1\right| \Leftrightarrow
\left\{
\begin{array}{l}
\alpha = |\lambda|^\tau -1 \;\;\mbox{if},\;\; \alpha>-1 \\
\alpha = -1 - |\lambda|^\tau \;\;\mbox{if},\;\; \alpha< -1
\end{array} \right.
$$
and, thus, both values  $\alpha=0$ and $\alpha = -2$ {are equivalent} to $\tau=0$. However, $\alpha$ and $\tau$ appear in homoclinic dynamics in
different ways: $\tau$ is a natural parameter describing the structure of nonwandering orbits of $f_0$, i.e., {when the tangency takes place},
and $\alpha$ is a natural {parameter} when studying bifurcations within the family $f_\mu$. However, in principle, they are the same. Therefore, in this
section, we will study bifurcation by means of the families $f_{\mu,\alpha}$ which unfold generally the initial tangency at $\mu=0,\alpha=0$,
except for the case $f_0\in H_3^{j,2}, j=2,3$ where the initial tangency exists at $\mu=0,\alpha = -2$.
We do this more for the simplicity of the presentation, since, in fact, we have already got, in Section~\ref{sec:fmu}, the formulae (namely,
(\ref{muk+1n}),
(\ref{muk+1-1n}),~(\ref{muk+1-1ln}) and
(\ref{muk+1-1lln}))  for the main bifurcation curves on the plane of parameters $\mu$ and $\alpha$.
 {We introduce the following notations for the bifurcation curves and the corresponding domains of existence of elliptic
periodic orbits.}

\begin{df}
\begin{itemize}
\item
For the {\em symplectic case}, we denote the curves~(\ref{muk+1n}) by $B_k^+$ and $B_k^-$, i.e., for $(\mu,\alpha)\in B_k^+$ (respectively for
$(\mu,\alpha)\in B_k^-$), the map $f_{\mu,\alpha}$ has a single-round periodic (of period $(k+q)$) orbit with double multiplier $+1$ (respectively
with double multiplier $-1$). Denote also by $E_k$ the domain between the curves $B_k^+$ and $B_k^-$, where the map $f_{\mu,\alpha}$ has a
single-round elliptic periodic orbit (of period $(k+q)$).
\item
For the {\em globally non-orientable case}, we denote the curves~(\ref{muk+1-1n}) by $B_k^{\pm 1}$ and $B_k^{2-}$, i.e., the map $f_{\mu,\alpha}$
has  a single-round periodic (of period $(k+q)$) orbit with multipliers $1$  and $-1$  {for $(\mu,\alpha)\in B_k^{\pm 1}$}
and  a double-round (of period $2(k+q)$) periodic orbit with double multiplier $-1$  {for $(\mu,\alpha)\in
B_k^{2-}$}. Denote also by $E_k^2$ the domain between the curves
$B_k^{\pm 1}$ and $B_k^{2-}$, where the map $f_{\mu,\alpha}$ has a double-round elliptic periodic orbit (corresponding to an elliptic {2-periodic
orbit} of the first return map $T_k$).
\item
For the {\em locally non-orientable case} we use the following notations. Denote the curves~(\ref{muk+1-1ln}) with even $k=2m$  by
$\tilde{B}_{2m}^{+}$ and $\tilde{B}_{2m}^{-}$, i.e., for $(\mu,\alpha)\in \tilde{B}_{2m}^+$ (respectively for $(\mu,\alpha)\in \tilde{B}_{2m}^-$),
the map $f_{\mu,\alpha}$ has a single-round periodic (of period $(2m+q)$) orbit with double multiplier $+1$ (respectively with double multiplier $-1$).
Denote also by $\tilde{E}_{2m}$ the domain between the curves $\tilde{B}_{2m}^+$ and $\tilde{B}_{2m}^-$, where the map $f_{\mu,\alpha}$ has a
single-round elliptic periodic orbit.
We denote the curves~(\ref{muk+1-1lln}) with $k=2m+1$ by $\tilde{B}_{2m+1}^{\pm 1}$ and $\tilde{B}_{2m+1}^{2-}$, i.e., the map $f_{\mu,\alpha}$
has  a single-round periodic (of period $(2m+1+q)$) orbit with multipliers~$+1$ and~$-1$ for $(\mu,\alpha)\in \tilde{B}_{2m+1}^{\pm 1}$ and
a double-round (of period $2(2m+1+q)$) periodic orbit with double multiplier~$-1$ for
$(\mu,\alpha)\in \tilde{B}_{2m+1}^{2-}$. Denote also by
$\tilde{E}_{2m+1}^2$ the domain between the curves $\tilde{B}_{2m+1}^{\pm 1}$ and $\tilde{B}_{2m+1}^{2-}$, where the map $f_{\mu,\alpha}$ has
a double-round elliptic periodic orbit (corresponding to an elliptic {2-periodic orbit} of the first return map $T_{2m+1}$).
\end{itemize}
\label{df:BkEk}
\end{df}

\begin{figure}[htb]
\centerline{\epsfig{file=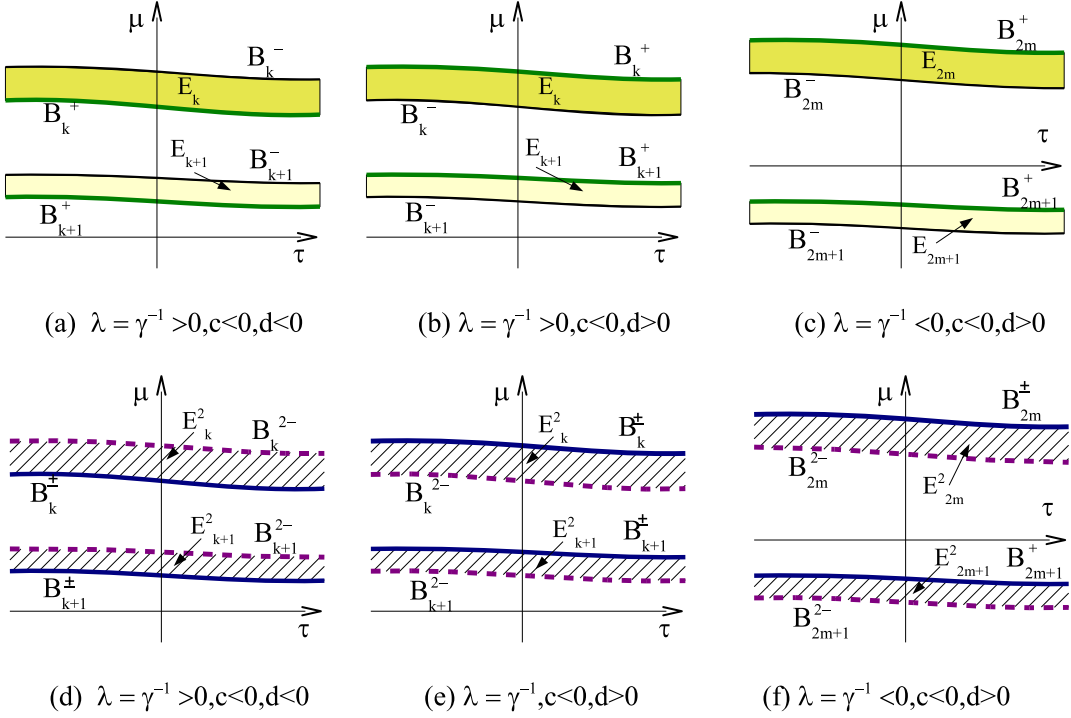, width=16cm
}} \caption{{\footnotesize Elements of the bifurcation diagrams for families $f_{\mu,\tau}$
with $c<0$: a)--c) the symplectic case; d)--f) the globally non-orientable case.} }
\label{fig:bifc<0}
\end{figure}

We note that the bifurcation diagrams in all cases with $c<0$, except for the locally non-orientable case $f_0\in H_3^{j,2}, j=2,3$,
are simple, since, by Theorem~\ref{th:1parcasc}, the intervals ${\sf e}_k$ (and ${\sf e}_k^2$) of existence of single-round (double-round)
elliptic periodic orbits do not intersect. If  we {now} consider the family $f_{\mu,\tau}$, then we obtain a picture {such} as {the one} in
Figure~\ref{fig:bifc<0}. {More precisely},
the following result is valid.

\begin{prop}
Let $f_0$ be a map of the first or second class or of the third class in $H_3^5$ and {assume} $\tau=0$ {in~(\ref{eq:tau})}.
Then, in any sufficiently small neighbourhood
of the origin of the plane $(\mu,\tau)$, the domains $E_k$ in the symplectic case or the domains $E_k^2$ in the globally non-orientable case do not
intersect and accumulate as $k\to\infty$ to the axis $\mu=0$, from one side for $\lambda=\gamma^{-1}>0$ and from both sides for
$\lambda=\gamma^{-1}<0$.
\label{prop:mutau12cl}
\end{prop}

\begin{proof}
Since  $c<0$ in the cases under consideration, we have that $cx^+-y^- = -y^-(|c|x^+/y^- +1) =
-y^-(|\lambda|^\tau +1)$, by~(\ref{eq:tau}) and~(\ref{alph1}). Then we obtain from~(\ref{muk+1n})
the following equations for the  curves $B_k^+$ and $B_k^-$:
\begin{equation}
\begin{array}{l}
\displaystyle B_k^+:\;\;\mu = \lambda^ky^-(|\lambda|^\tau +1)(1 + k\beta_1 \lambda^kx^+y^-) +
\frac{1-s_0+\dots}{d}\lambda^{2k},\\
\displaystyle B_k^-:\;\;\mu = \lambda^ky^-(|\lambda|^\tau +1)(1 + k\beta_1 \lambda^kx^+y^-) -
\frac{3+s_0+\dots}{d}\lambda^{2k},
\end{array}
\label{eq:Bk+Bk-}
\end{equation}
which take place in the symplectic case.
Accordingly, we obtain from~(\ref{muk+1-1n}) the following equations for
the curves $B_k^-$ and $B_k^{2-}$:
\begin{equation}
\begin{array}{l}
\displaystyle B_k^-:\;\; \mu  =  \lambda^ky^-(|\lambda|^\tau +1)(1 + k\beta_1
\lambda^k x^+y^-) - \frac{1}{d}(s_0^{nor} + \hat\rho_k)\lambda^{2k},\\
\displaystyle B_k^{2-}:\;\; \mu  =  \lambda^ky^-(|\lambda|^\tau +1)(1 + k\beta_1 \lambda^k x^+y^-)
- \frac{1}{d}(s_0^{nor} + 1 + \hat\rho_k)\lambda^{2k},
\end{array}
\label{eq:Bk+-}
\end{equation}
which take place in the globally non-orientable case. The proposition follows immediately from these formulae, since $|\lambda|^\tau +1>1$ and
the ``strips'' $E_k$ and $E_k^2$ have thickness of order {$\lambda^{2k}$ }.
\end{proof}

In the remaining cases of APMs with quadratic homoclinic tangencies, the structure of the domains of existence of elliptic periodic orbits
near the origin $(\mu=0,\tau=0)$ will be absolutely  {different}, {since} the domains $E_k$,  $E_k^2$,  $\tilde E_k$ or  $\tilde E_k^2$
can intersect.
Moreover, {an infinity of}  such domains can contain the origin and, hence, infinitely many single-round or double-round elliptic periodic orbits
can coexist. Below we consider {such phenomena}.

First of all, we consider the symplectic case for which we establish the following result.
\begin{figure}[htb]
\centerline{\epsfig{file=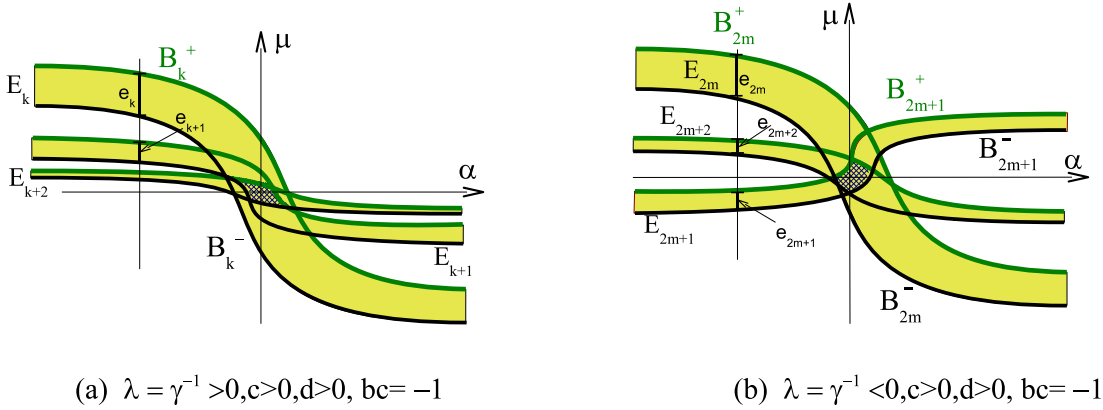, width=16cm
}} \caption{{\footnotesize Elements of the bifurcation diagrams for families $f_{\mu,\alpha}$ in the case
of symplectic maps $f_0$ of the third class in
(a) $H_3^1$ and (b) $H_3^4$. Here the  {case} of ``global
resonance'' at $\mu=0,\alpha=0$ is shown,
when all the domains $E_k$ contains the origin
$(\mu=0,\alpha=0)$. }}
\label{fig2marn}
\end{figure}

\begin{theorem}
Let $f_0 \in H_3^1\cup H_3^4$ be a symplectic map
and $f_{\mu,\alpha}$ be a two parameter general unfolding with the governing parameters $\mu$ and $\alpha$.
Then, in
any sufficiently small neighbourhood of the origin of the parameter plane $(\mu,\alpha)$, there are infinitely many
domains $E_k$, $k=\bar k, \bar k +1,...$, which accumulate to the axis $\mu=0$ as $k\to\infty$ and such that\\
{\rm 1)} all the domains $E_k$ are mutually crossed and intersect the axis $\mu=0$; \\
{\rm 2)} if $\;-3 < s_0 <1$, all the domains $E_k$ contains the origin $(\mu=0,\alpha=0)$ and, hence, the map $f_{0,0}$ has infinitely many
single-round elliptic periodic orbits of all successive periods $k=\bar k+q, \bar k +1+q,...$. Moreover, if $s_0 \neq 0, -5/4$ all these orbits
are generic.
\label{th:sympinf}
\end{theorem}

\begin{proof}
In the symplectic case, when $f_0$ belongs to $H_3^1$ (where $\lambda = \gamma^{-1}>0, c>0, d>0$) or $H_3^4$ (where
$\lambda = \gamma^{-1}<0, c>0, d>0$), the bifurcation curves $B_k^+,B_k^-$ (the boundaries of the domain $E_k$) on the parameter plane $(\mu,\alpha)$
are given  by ~(\ref{muk+1n}). It is easy to see from these formulae that all the curves $B_k^+,B_k^-$  mutually intersect
and that they cross the axis $\alpha=0$ at 
$\mu~=~\mu_k^+~= -(d)^{-1}(s_0-1 + ...)\lambda^{2k}$ and
$\mu = \mu_k^- = -(d)^{-1}(s_0+3 + ...)\lambda^{2k}$,
and the axis $\mu=0$ at the points  $\alpha = \alpha_k^+ = (dy^-)^{-1}(s_0-1 + ...)\lambda^{k}$ and
$\alpha = \alpha_k^- = (dy^-)^{-1}(s_0+3 + ...)\lambda^{k}$. Then if $-3 < s_0 <1$, all the domains $E_k$ with sufficiently large $k$ contains
the origin $(\mu=0,\alpha=0)$.

Moreover, by the Rescaling Lemma~\ref{henmain}, all the first return maps $T_k$ {have ``the same expression'' for} $\mu=0, \alpha=0$. Indeed,
we obtain from~(\ref{mui}) that the rescaled form~(\ref{henon0}) of $T_k$ in this case looks as
\begin{equation}
\bar{X} \; = \; Y  + O(k\lambda^{2k}), \bar{Y}  \;=\; (- s_0 + \rho_k^2) - X + \frac{f_{03}}{d^2}\lambda^{k} Y^3 +
O(k\lambda^{2k}) \;,
\label{henon000n}
\end{equation}
where $\rho_k^2 = O(k\lambda^k)$ is a small coefficient (a correction to $s_0$).
Then,  see Section~\ref{sec:421}, if $-3<s_0<1$ every map~(\ref{henon000n}) with sufficiently large $k$ has an elliptic fixed point which is
generic if $s_0\neq 0, - 5/4$, i.e., if the strong resonances ($\psi =\pi/2, 2\pi/3$) are absent.
\end{proof}

In Figure~\ref{fig2marn}
we give an illustration of this theorem for different cases.

A similar result takes place in the globally non-orientable case.

\begin{theorem}
Let $f_0 \in H_3^1\cup H_3^4$  in the globally non-orientable case and
$f_{\mu,\alpha}$ be a two parameter general unfolding of $f_0$ with the governing parameters $\mu$ and $\alpha$. Then, in
any sufficiently small neighbourhood of the origin of the parameter plane $(\mu,\alpha)$, there are infinitely many
domains $E_k^2$, $k=\bar k, \bar k +1,...$, which accumulate to the axis $\mu=0$ as $k\to\infty$ and such that
{\rm 1)} all the domains $E_k^2$ are mutually crossed and intersect the axis $\mu=0$;\\
{\rm 2)} if $-1 < s_0^{nor} <0$, all the domains $E_k^2$ contain the origin $(\mu=0,\alpha=0)$ and, hence, the map $f_{0,0}$ has infinitely
many double-round elliptic periodic orbits  of all successive even periods $2(\bar k+q), 2(\bar k +q+1),...$. Moreover, if
$s_0 \neq -1/2, -3/4, - 5/8$, all these orbits are generic.
\label{th:gnorinf}
\end{theorem}

\begin{proof}
In the globally non-orientable case, when $f_0$ belongs to $H_3^1$ (with $\lambda = \gamma^{-1}>0, c>0, d~>0, bc=+1$) or $H_3^4$
(with $\lambda = \gamma^{-1}<0, c>0, d>0,bc = +1$), the bifurcation curves $B_k^{\pm 1},B_k^{2-}$ on the parameter plane $(\mu,\alpha)$ are given
by~(\ref{muk+1-1n}). It is easy to see from these formulae that all the curves $B_k^{\pm 1},B_k^{2-}$ mutually intersect
and cross the axis $\alpha=0$ at the points $\mu = -(d)^{-1}(s_0^{nor}+...)\lambda^{2k}$ and $\mu = -(d)^{-1}(s_0^{nor}+1+...)\lambda^{2k}$,
and the axis $\mu=0$ at the points  $\alpha = (dy^-)^{-1}(s_0^{nor}+...)\lambda^{k}$ and $\alpha =  (dy^-)^{-1}(s_0^{nor}+1+...)\lambda^{k}$.
Then, if $-1 < s_0 <0$, all the domains $E_k^2$ with sufficiently large $k$ contain the origin $(\mu=0,\alpha=0)$.

Moreover, by the Rescaling Lemma~\ref{henmain}, all the first return maps $T_k$ {have ``the same expression'' for} $\mu=0$ and $\alpha=0$.
Indeed, we obtain from~(\ref{mui}) and~(\ref{s0}) that the rescaled form~(\ref{henon0}) of $T_k$ in this case looks as
\begin{equation}
\bar{X} \; = \; Y  + O(k\lambda^{2k}), \bar{Y}  \;=\;  (- s_0^{nor} + \rho_k^4) + X + \frac{f_{03}}{d^2}\lambda^{k} Y^3 +
O(k\lambda^{2k}) \;.
\label{henon000ng}
\end{equation}
Then, if $-1<s_0<0$, every map~(\ref{henon000ng}) with sufficiently large $k$ has an {elliptic 2-periodic orbit} which is generic if
$s_0\neq -1/2, -3/4, -5/8$, see Section~\ref{sec:422}.
\end{proof}

In Figure~\ref{fig2marng}
we give an illustration of this theorem.
\begin{figure}[htb]
\centerline{\epsfig{file=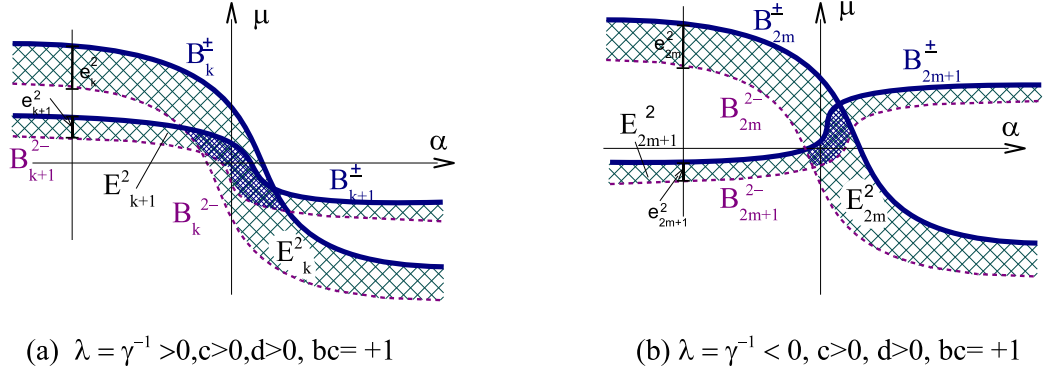, width=16cm
}} \caption{{\footnotesize Elements of the bifurcation diagrams for families $f_{\mu,\alpha}$ in the globally non-orientable maps $f_0$ of
the third class in
(a) $H_3^1$ and (b)  $H_3^4$.
}}
\label{fig2marng}
\end{figure}

We consider now the locally non-orientable case. Then we recall that $f_0 \in H_3^2\cup H_3^3$,  $\lambda\gamma = -1$ and $bc=-1$ (i.e., the
local map $T_0$ is non-orientable and the global map $T_1$ is orientable).

\begin{theorem}
{\rm (I)} Let $f_0$
belong to $H_3^{2,1}$ or $H_3^{3,1}$ and
$f_{\mu,\alpha}$ be a two parameter general unfolding
with the governing parameters $\mu$ and $\alpha$. Then in any sufficiently small neighbourhood $V$ of the point $(\mu=0,\alpha=0)$ there are
infinitely many
 domains $\tilde E_{2m}$ and $\tilde E_{2m+1}^2$,
$m=\bar m, \bar m +1,...$, which accumulate to the axis $\mu=0$ as $m\to\infty$, and the following holds.

\begin{itemize}
\item[{\rm Ia.}]
In $V$   all the domains $\tilde E_{2m}$ are crossed and intersect the axis $\mu$, whereas the domains $\tilde E_{2m+1}^2$
are not mutually crossed and do not intersect the axis $\mu$.
\item[{\rm Ib.}]
If  $-3 < s_0 <1$, all the domains $\tilde E_{2m}$ contain the origin $(\mu=0,\alpha=0)$ and, hence, the map $f_{0,0}$
has infinitely many single-round elliptic periodic orbits of all periods of the form $2m+q$, where $m=\bar m, \bar m +1,...$. Moreover, if
$s_0 \neq 0, - 5/4$, all these orbits are generic. 
\end{itemize}

{\rm (II)} Let $f_0$
belong to $H_3^{2,2}$ or $H_3^{3,2}$ and
$f_{\mu,\tilde\alpha}$ be a two parameter general unfolding with the governing parameters $\mu$ and $\tilde\alpha$.
\footnote{Recall that $\alpha = cx^+/y^- -1$ and $\tilde\alpha = cx^+/y^- +1$, i.e., $\tilde\alpha = \alpha +2$ and  both
$\alpha = 0$ and $\tilde\alpha = 0$ correspond to $\tau =0$, see~(\ref{eq:tau}).} Then in any sufficiently small neighbourhood $\tilde V$
of the origin $(\mu=0,\tilde\alpha=0)$ there are infinitely many
the domains $\tilde E_{2m}$ and $\tilde E_{2m+1}^2$,
$m=\bar m, \bar m +1,...$, which accumulate to the axis $\mu=0$ as $m\to\infty$, and the following holds.
\begin{itemize}
\item[{\rm IIa.}]
In $\tilde V$
all the domains $\tilde E_{2m+1}^2$ mutually intersect, whereas the domains $\tilde E_{2m}$ with different $m$ do not intersect and do not cross
the axis $\mu=0$.
\item[{\rm IIb.}]
If $\;-1 < s_0^{nor} <0$, all the domains $\tilde E_{2m+1}^2$ contain the point $(\mu=0,\tilde\alpha=0)$ and, hence, the map $f_{0,0}$  has
infinitely many double-round elliptic periodic orbits of all periods of the form $2(2m+1+q)$, where $m=\bar m, \bar m +1,...$. Moreover,
if $s_0 \neq -1/2, -3/4, -5/8$, all these orbits are generic.
\end{itemize}
\label{th:lnorinf}
\end{theorem}

\begin{proof}
I) If $f_0$ belongs to $H_3^{2,1}$ (where $\lambda = -\gamma^{-1}<0, c>0, d<0, bc=-1$) or $H_3^{3,1}$ (where
$\lambda = - \gamma^{-1}<0, c>0, d>0,bc = -1$), the bifurcation curves $\tilde B_{2m}^{+},\tilde B_k^{-}$ on the parameter plane $(\mu,\alpha)$ are
given by formulae~(\ref{muk+1-1ln}) with $k=2m$, where $s_0$ satisfies~(\ref{s0symp}). Then it follows from these formulae that all the
curves $\tilde B_{2m}^{+},\tilde B_{2m}^{-}$ mutually intersect and they intersect the axis $\alpha=0$ at the points
$\mu = -(d)^{-1}(s_0 -1 +...)\lambda^{4m}$ and $\mu = -(d)^{-1}(s_0+3+...)\lambda^{4m}$,
and the axis $\mu=0$ at the points  $\alpha = - (dy^-)^{-1}(s_0-1 +...)\lambda^{2m}$ and $\alpha =  -(dy^-)^{-1}(s_0+3+...)\lambda^{2m}$.
Then, if $-3 < s_0 <1$, all the domains $E_{2m}^2$ with sufficiently large $k$ contain the origin $(\mu=0,\alpha=0)$.

Moreover, by the Rescaling Lemma~\ref{henmain}, all the first return maps $T_k$ {have ``the same expression'' for} $\mu=0$ and $\alpha=0$. Indeed,
we obtain from~(\ref{mui}) that the rescaled form~(\ref{henon0}) of $T_k$  takes the form
(\ref{henon000n}) with $k=2m$.
Then, if $-3<s_0<1$ every map~(\ref{henon000ng}) with sufficiently large $k=2m$ has an elliptic fixed point which is generic if
$s_0\neq 0, -5/4$, see Section~\ref{sec:421}.

Since $c>0$ in the case under consideration, it follows from~(\ref{eq:tau}) that $\alpha > -1$. This means that the curves
$\tilde B_{2m+1}^{\pm 1}$ and $\tilde B_{2m+1}^{2-}$, given by formulae~(\ref{muk+1-1lln}) with $k=2m+1$, as well as the corresponding
domains $\tilde E_{2m+1}^2$ do not mutually intersect for different sufficiently large $m$. Moreover, they accumulate to the axis
$\mu =0$ as $m\to \infty$ from one side ($\mu>0$, since $\alpha+2>0$ and $\lambda^k <0$ for odd $k$). This completes the proof for the case I.

II. Let now $f_0$ belong to $H_3^{2,2}$ (where $\lambda = -\gamma^{-1}<0, c<0, d<0, bc=-1$) or $H_3^{3,2}$ (where
$\lambda = - \gamma^{-1}<0, c<0, d>0,bc = -1$). Therefore, since $\alpha <-1$, in contrast to the previous case, the bifurcation curves
$\tilde B_{2m}^{+}$ and $\tilde B_{2m}^{-}$ (see formula~(\ref{muk+1-1ln}) with $k=2m$) as well as the corresponding domains $\tilde E_{2m}$
do not mutually intersect in $\tilde V$. Moreover, they
accumulate to the axis $\mu =0$ as $m\to \infty$ from one side ($\mu>0$, since $\alpha < -1$ and $\lambda^k >0$ for even $k$).

Since $\alpha < -1$ in the case under consideration,  the curves $\tilde B_{2m+1}^{\pm 1}$ and $\tilde B_{2m+1}^{2-}$, given by
formulae~(\ref{muk+1-1lln}) with $k=2m+1$, can now intersect in $\tilde V$
for different sufficiently large $m$.
In this case the intersection points with the axis $\tilde\alpha=0$ have the coordinates
$\mu = -(d)^{-1}(s_0^{nor} +...)\lambda^{2k}$ and $\mu = -(d)^{-1}(s_0+1+...)\lambda^{2k}$ with $k=2m+1$; and with the axis $\mu=0$ the
coordinates $\tilde\alpha = - (dy^-)^{-1}(s_0^{nor} +...)\lambda^{2m+1}$ and $\tilde\alpha =  -(dy^-)^{-1}(s_0^{nor}+1+...)\lambda^{2m+1}$.
Then if $-1 < s_0 <0$, all the domains $\tilde E_{2m+1}^2$ with sufficiently large $m$ contains the origin $(\mu=0,\tilde\alpha=0)$ of $\tilde V$.

Moreover, by the Rescaling Lemma~\ref{henmain}, all the first return maps $T_{2m+1}$ {have ``the same expression'' for} $\mu=0$ and
$\tilde\alpha=0$ (i.e., $cx^+ = - y^-$). Indeed, we obtain from~(\ref{muinor}) that the rescaled form~(\ref{henon0}) of $T_k$  takes the form
\begin{equation}
\bar{X} \; = \; Y  + O(k\lambda^{2k}), \bar{Y}  \;=\;  (- s_0^{nor} + \rho_k^4) + X - \frac{f_{03}}{d^2}\lambda^{k} Y^3 +
O(k\lambda^{2k}) \;,
\label{henon000ngl}
\end{equation}
where $k=2m+1$ and $\rho_k^4 = O(k\lambda^k)$ is a small coefficient (a correction to $s_0^{nor}$). (This map differs from~(\ref{henon000ngl})
only by the sign {in front of} $Y^3$).
Then, if $-1<s_0^{nor}<0$ every map~(\ref{henon000ngl}) with sufficiently large $k=2m+1$ has an elliptic {2-periodic orbit} which is generic,
if $s_0\neq -1/2, -3/4, -5/8$, see Section~\ref{sec:422}.
\end{proof}

In Figures~\ref{fig23mar} and~\ref{fig33mar}
we give an illustration of this theorem for different cases.

\begin{figure}[htb]
\centerline{\epsfig{file=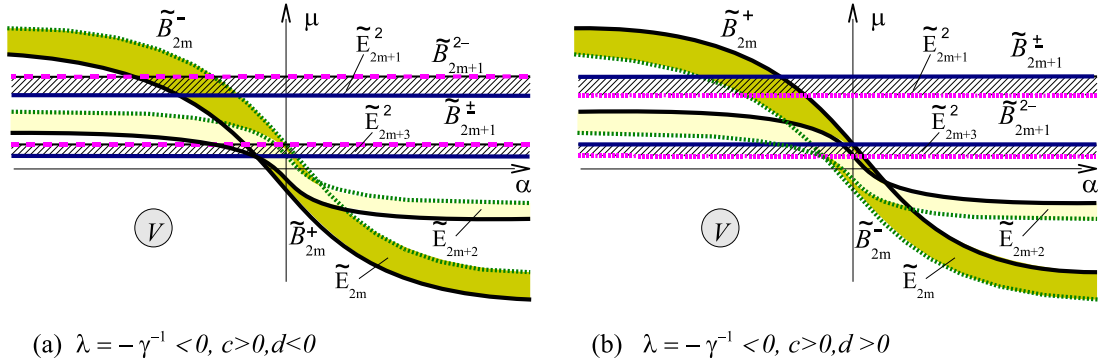, width=16cm
}} \caption{{\footnotesize Elements of the bifurcation  diagram in a neighbourhood $V(\mu=0,\alpha=0)$  for
the families $f_{\mu,\alpha}$ in the cases where (a)$f_0 \in H_3^{2,1}$; (b) $f_0 \in H_3^{3,1}$. } }
\label{fig23mar}
\end{figure}

\begin{figure}[htb]
\centerline{\epsfig{file=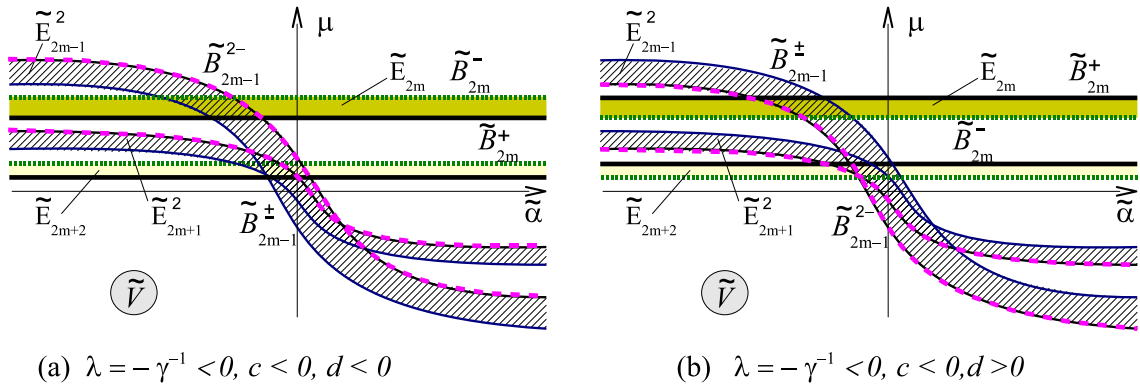, width=16cm
}} \caption{{\footnotesize Elements of the bifurcation  diagram in a neighbourhood $\tilde V(\mu=0,\tilde\alpha=0)$  for
the families $f_{\mu,\alpha}$ in the cases where (a)$f_0 \in H_3^{2,2}$; (b) $f_0 \in H_3^{3,2}$. } }
\label{fig33mar}
\end{figure}

\section{Invariants of homoclinic tangencies.}
\label{newinv}

We {have just seen} that in the case of global resonance $\tau=0$ the dynamics of APMs 
of the
third class (except for maps in $H_3^5$) depends, {indeed},  only on the  quantity $s_0$. In this
section we prove, for completeness,  the invariance of $s_0$.

First, we recall the result from~\cite{GS87} that the quantity $\tau$ is an invariant of two-dimensional
diffeomorphisms with homoclinic tangencies to a saddle with $\sigma\equiv |\lambda\gamma| = 1$. In particular, it was proved in~\cite{GS87} that
the value of $\tau$ does not depend {neither on} the  choice of pairs of homoclinic points $M^+$ and $M^-$ {nor in the}
coordinate changes conserving the
first order normal
form of the saddle map~$T_0$. This implies that, in the case of APMs, 
$\tau$ is invariant in those
$C^r$-coordinates which conserve the first order normal form~(\ref{eq:nf1}) of the saddle map $T_0$.
Note that, as it was shown in~\cite{AfrT05}, $\tau$ is also invariant under $C^1$-linearization
coordinates. In principle, this result {could} be used for proving the existence/absence of topological Smale horseshoes
near a homoclinic tangency.

{We now} prove the invariance of $s_0$. However, in contrast to $\tau$, we prove the invariance of $s_0$ in those
$C^{r-2}$-coordinates which conserve the second order normal form~(\ref{eq:nfn2}) of the local map $T_0$ (or any $n$-order normal
form~(\ref{eq:nfn}) for $n\geq 2$).
Naturally, $s_0$ ``disappears'' when a $C^1$-linearization is used, since $s_0$ depends on the coefficients
of $T_1$ in the quadratic terms which become indefinite for
$C^1$-changes.

\begin{lm}
Let $f_0\in H_3^1\cup H_3^4$ and $\tau=0$. Then, in those coordinates where the local map $T_0$ takes the second normal form~(\ref{eq:nfn2}),
the value of   $s_0$ does not depend on the choice of pairs of  {homoclinic} points of the orbit $\Gamma_0$.
\label{invs0}
\end{lm}

\begin{proof}
We take first the  pair $M^{+\prime}=T_0(M^+)$ and $M^-$ of points of
$\Gamma_0$. Then the
new global map $T_1^\prime  = T_0 T_1:\Pi^-\to T_0(\Pi^+)$ can be written, by~(\ref{eq:nfn2}), in the form
\begin{equation}
\begin{array}{l}
\bar x^\prime = \lambda \bar x (1+ \beta_1\bar x\bar y) +
O\left[\bar x^3\bar y^2\right],\;\;
\bar y^\prime = \gamma \bar y (1 - \beta_1\bar x\bar y) +
O\left[\bar x^2\bar y^3\right],\;
\end{array}
\label{01nf2}
\end{equation}
where $\bar x = x^+ + F(x,y-y^-), \bar y = G(x,y-y^-)$.
We will calculate the corresponding coefficients (that define a new $s_0^\prime$) at the homoclinic
point $M^-(x=0,y=y^-)$ using that  $\bar x = x^+, \bar y =0$,  $G_y(0,0)=0$ at this point. It follows from~(\ref{01nf2}) that
$$
\frac{\partial\bar y^\prime}{\partial \bar x} = 0, \;\; \frac{\partial\bar x^\prime}{\partial \bar
x} = \lambda,\;\; \frac{\partial\bar y^\prime}{\partial \bar y} = \gamma \;\;{\mbox{at}}\;\; \bar x
= x^+, \bar y =0
$$
and the  $O$-terms in~(\ref{01nf2}) vanish for $\bar y = 0$ along with all the
required derivatives (note that {only} the second derivatives  {of $\bar y^\prime$} are needed). Thus, we
have
$$
\begin{array}{l}
\displaystyle a^\prime = \frac{\partial\bar x^\prime}{\partial x}
= \lambda \frac{\partial F}{\partial x} + \lambda \beta_1 (\bar
x)^2 \frac{\partial G}{\partial x} + O(\bar y), \;\;
c^\prime = \frac{\partial\bar y^\prime}{\partial x} = \gamma
\frac{\partial G}{\partial x} +  O(\bar y), \\~\\
\displaystyle d^\prime = \frac{1}{2}\frac{\partial^2\bar
y^\prime}{\partial y^2} =
\displaystyle \frac{1}{2}\gamma \frac{\partial^2 G}{\partial y^2}
+ O(\bar y) + O(\partial\bar y/\partial y), \\~\\
\displaystyle f_{20}^\prime = \frac{1}{2}\frac{\partial^2\bar
y^\prime}{\partial x^2} =
\displaystyle \frac{1}{2}\gamma \left(\frac{\partial^2 G}{\partial
x^2} - 2\beta_1 \bar x\left(\frac{\partial G}{\partial
x}\right)^2\right)  + O(\bar y), \\~\\
\displaystyle f_{11}^\prime = \frac{\partial^2\bar
y^\prime}{\partial x\partial y} =
\displaystyle \gamma \frac{\partial^2 G}{\partial x\partial y} +
O(\bar y) + O(\partial\bar y/\partial y).
\end{array}
$$
Since we calculate these derivatives at the point $x=0,y=y^-$, where $\bar x=x^+,\bar y =0$ and
$\partial\bar y/\partial y \equiv G_y =0$, {we get}
\begin{equation}
{x^+}^\prime = \lambda x^+, \;
a^\prime = \lambda a + \lambda (x^+)^2\beta_1 c, \;\; c^\prime = \gamma c,\;\; d^\prime = \gamma d,
\;\; f_{20}^\prime = \gamma f_{20} - \gamma c^2\beta_1 x^+,\;\;f_{11}^\prime =  \gamma f_{11}.
\label{s0prime1}
\end{equation}

Then, by ~(\ref{s0}), we obtain that
$$
\begin{array}{l}
s_0^\prime = d^\prime{x^+}^\prime(a^\prime c^\prime + f_{20}^\prime {x^+}^\prime) + \frac{1}{2}f_{11}^\prime {x^+}^\prime\left(1 + \nu_1 -
\frac{1}{2}f_{11}^\prime {x^+}^\prime\right) = \\~\\ = \lambda\gamma dx^+\left((a + (x^+)^2\beta_1 c) c + (f_{20} - c^2\beta_1 x^+)x^+ \right) +
\lambda\gamma \frac{1}{2}f_{11} {x^+}\left(1 + \nu_1 -
\frac{1}{2}f_{11}{x^+}\right) = s_0.
\end{array}
$$

We take now the pair $M^{+\prime}=M^+$ and $M^{-\prime}=
T_0^{-1}(M^-)$ of points of $\Gamma_0$. Then
the new global map
$T_1^\prime = T_1
T_0:  T_0^{-1}(\Pi^-)\to\Pi^+$ can be written as
$$
\bar x = x^+ + F(x^\prime, y^\prime-y^-),\;\;
\bar y  = G(x^\prime, y^\prime-y^-),
$$
where $x^\prime = \lambda  x (1+ \beta_1 x y) + O(x^3 y^2)$ and $y^\prime = \gamma y (1 - \beta_1 x y) + O(x^2 y^3)$ are coordinates in $\Pi^-$
and $(x,y)\in
T_0^{-1}\Pi^-$,   $(\bar x,\bar y)\in
\Pi^+$. Thus, we have that $x^{+\prime}=x^+,
y^{-\prime}=\gamma^{-1}y^-$. Further, we calculate other coefficients
as the corresponding derivatives of $(\bar x, \bar y)$ {with respect to} $(x,y)$ calculated at the point $x=0, y= \gamma^{-1}y^-$. We
{get}
$$
\begin{array}{l}
\displaystyle a^{\prime} = \frac{\partial F }{\partial
x^\prime}\frac{\partial x^\prime}{\partial x} + \frac{\partial F
}{\partial y^\prime} \frac{\partial y^\prime}{\partial x}, \;\;\;
\displaystyle c^{\prime} = \frac{\partial G }{\partial
x^\prime}\frac{\partial x^\prime}{\partial x} + \frac{\partial G
}{\partial y^\prime} \frac{\partial y^\prime}{\partial x}, \\~\\
\displaystyle f_{11}^\prime = \frac{\partial^2 G }{(\partial
x^\prime)^2}\frac{\partial x^\prime}{\partial x}\frac{\partial
x^\prime}{\partial y} + \frac{\partial^2 G }{\partial
x^\prime\partial y^\prime}\left(\frac{\partial y^\prime}{\partial
x}\frac{\partial x^\prime}{\partial y} + \frac{\partial
x^\prime}{\partial x}\frac{\partial y^\prime}{\partial y}\right) +
\frac{\partial^2 G }{\partial y^{\prime 2}}\frac{\partial
y^\prime}{\partial x}\frac{\partial y^\prime}{\partial y} +
\frac{\partial G }{\partial x^\prime}\frac{\partial^2
x^\prime}{\partial x\partial y} + \frac{\partial G }{\partial
y^\prime}\frac{\partial^2 y^\prime}{\partial x\partial y}, \\~\\
\displaystyle d^\prime = \frac{1}{2} \left(\frac{\partial^2 G
}{(\partial x^\prime)^2}\left(\frac{\partial x^\prime}{\partial
y}\right)^2 + 2 \frac{\partial^2 G }{\partial x^\prime\partial
y^\prime} \frac{\partial y^\prime}{\partial y}\frac{\partial
x^\prime}{\partial y} + \frac{\partial^2 G }{\partial y^{\prime
2}} \left(\frac{\partial y^\prime}{\partial y}\right)^2 +
\frac{\partial G }{\partial x^\prime}\frac{\partial^2
x^\prime}{\partial y^2} + \frac{\partial G }{\partial y^\prime}
\frac{\partial^2 y^\prime}{\partial y^2}\right), \\~\\
\displaystyle f_{20}^\prime = \frac{1}{2} \left(\frac{\partial^2 G
}{(\partial x^\prime)^2}\left(\frac{\partial x^\prime}{\partial
x}\right)^2 + 2 \frac{\partial^2 G }{\partial x^\prime\partial
y^\prime} \frac{\partial y^\prime}{\partial x}\frac{\partial
x^\prime}{\partial x} + \frac{\partial^2 G }{\partial y^{\prime
2}} \left(\frac{\partial y^\prime}{\partial x}\right)^2 +
\frac{\partial G }{\partial x^\prime}\frac{\partial^2
x^\prime}{\partial x^2} + \frac{\partial G }{\partial y^\prime}
\frac{\partial^2 y^\prime}{\partial x^2}\right).
\end{array}
$$
Since
$$
\frac{\partial G }{\partial y^\prime}=0,\;\;
\frac{\partial}{\partial y}\left(x^\prime,\frac{\partial
x^\prime}{\partial x},\frac{\partial x^\prime}{\partial y}\right)
=0, \;\;\frac{\partial x^\prime}{\partial x} = \lambda, \;\;
\frac{\partial y^\prime}{\partial y} = \gamma,\;\; \frac{\partial
y^\prime}{\partial x} = -\gamma^{-1}\beta_1 (y^-)^2
$$
at the point $M^{-\prime}\left(x=0,y=\gamma^{-1}y^-\right)$, we obtain that
\begin{equation}
\begin{array}{l}
a^\prime  = \lambda a - b\beta_1\gamma^{-1}(y^-)^2, \;\; c^\prime
= \lambda c, \;\; f_{11}^\prime = \lambda\gamma f_{11} - 2d\beta_1 (y^-)^2, \;\;
d^\prime = d\gamma^2, \\
f_{20}^\prime = f_{20}\lambda^2 - f_{11} \lambda^2\beta_1 (y^-)^2
+ d\lambda^2 \beta_1^2 (y^-)^4 + c\lambda^2\beta_1 y^-,\;\;
\end{array}
\label{s0prime2}
\end{equation}
Since $\lambda\gamma =1$, we obtain, by~(\ref{s0}), that
$$
\begin{array}{l}
s_0^\prime = d^\prime{x^+}^\prime(a^\prime c^\prime + f_{20}^\prime {x^+}^\prime) + \frac{1}{2}f_{11}^\prime {x^+}^\prime\left(1 + \nu_1 -
\frac{1}{2}f_{11}^\prime {x^+}^\prime\right) = \\ = dx^+\left[ac - cb\beta_1(y^-)^2 + f_{20}x^+ -
f_{11}x^+ \beta_1 (y^-)^2 +
dx^+ \beta_1^2 (y^-)^4 + c\beta_1 x^+y^-\right] + \\
+ \frac{1}{2}(f_{11}x^+ - 2d\beta_1 (y^-)^2x^+)\left(1 + \nu_1 -
\frac{1}{2}(f_{11}x^+
- 2d\beta_1 (y^-)^2x^+)\right) = \\ =  s_0 + d\beta_1 x^+y^- (cx^+ - bc
y^- -y^-(1+\nu_1)).
\end{array}
$$
Note that $\nu_1 = -bc$ in the case $\lambda\gamma=1$ and, thus, $s_0^\prime = s_0 + d\beta_1 x^+y^- (cx^+ - y^-)$. It follows that
$s_0^\prime = s_0$ at $cx^+ = y^-$ which is equivalent to $\tau=0$ if $c>0$.
\end{proof}

{In the locally non-orientable case, $s_0$ remains invariant with respect to
the choice of
any pair of homoclinic points of the needed type (see the definition of homoclinic
points of the needed time just before the condition \textbf{D} in
Section~\ref{sec:globmap}).}

\begin{lm}
Let  $f_0\in H_3^2\cup H_3^3$. Then, in coordinates where the  map $T_0$ takes the second normal form~(\ref{eq:nfn2}), the value of
$s_0$ does not depend on the choice of pairs of homoclinic points of the needed type.
\label{invs01}
\end{lm}

\begin{proof}
By condition~\textbf{D}, the pair $M^+$ and $M^-$ of homoclinic points is of the needed type (i.e., the corresponding map $T_1$ is orientable).
We prove the invariance of $s_0$ for the pairs a) $T_0^2(M^+)$ and $M^-$; b) $M^+$ and $T_0^{-2}(M^-)$ and c) $T_0(M^+)$ and $T_0^{-1}(M^-)$, which
are {all} of the needed type. Note that in the case $\lambda\gamma =-1$ the
calculations become much simpler, since $\beta_1=0$.

a) For the pair ${M^+}^\prime = T_0(M^+)$ and ${M^-}^\prime = M^-$ of homoclinic points, we obtain from~(\ref{s0prime1}) that
\begin{equation}
{x^+}^\prime = \lambda x^+, \;
a^\prime = \lambda a, \;\; c^\prime = \gamma c,\;\; d^\prime = \gamma d,
\;\; f_{20}^\prime = \gamma f_{20},\;\;f_{11}^\prime =  \gamma f_{11}.
\label{s0prime1a}
\end{equation}
Since $\lambda\gamma = -1$, we obtain then that $s_0^\prime = - s_0$. Analogously, for the pairs ${M^+}^{''} = T_0({M^+}^\prime)$ and
${M^-}^{''} = M^-$, we obtain that
$s_0^{''} =  - s_0^\prime$ and, hence $s_0^{''} =   s_0$.

b) For the pair ${M^+}^\prime = M^+$ and ${M^-}^\prime = T_0^{-1}(M^-)$ of homoclinic points, we have that ${x^+}^\prime = x^+$ and, by
(\ref{s0prime2}),
\begin{equation}
\begin{array}{l}
a^\prime  = \lambda a, \;\; c^\prime
= \lambda c, \;\; f_{11}^\prime =  - f_{11}, \;\;
d^\prime = d\gamma^2, \;\;
f_{20}^\prime = f_{20}\lambda^2.
\end{array}
\label{s0prime2b}
\end{equation}
Since $\lambda\gamma = - 1$, we obtain by~(\ref{s0}) that
\begin{equation}
\left[dx^+(ac + f_{20}x^+)\right]^\prime = dx^+(ac + f_{20}x^+) \;\;\mbox{and}\;\; \left[f_{11}x^+\right]^\prime = - f_{11}x^+.
\label{s0prime2bn}
\end{equation}
This implies evidently that $s_0^{''} = s_0$ for the needed type pair ${M^+}^{''} = M^+$ and ${M^-}^{''} = T_0^{-2}(M^-)$ of homoclinic points.

c) Consider first the pair ${M^+}^\prime = T_0(M^+)$ and ${M^-}^\prime = M^-$ for which formula~(\ref{s0prime1a}) holds {with a negative}
coordinate ${x^+}^\prime$ of the point  ${M^+}^\prime$. Therefore, we make the coordinate change $x\to -x, y \to y$ after which the
new map $T_1^\prime$ will have the following coefficients
$$
{x^+}^\prime = - \lambda {x^+}, \;\;
a^\prime  = \lambda a, \;\; c^\prime
= - \gamma c, \;\; f_{11}^\prime = -\gamma f_{11}, \;\;
d^\prime = d\gamma, \;\;
f_{20}^\prime = f_{20}\gamma ,
$$
which gives relation~(\ref{s0prime2bn}). Evidently, at the further transition to the pair ${M^+}^{''} = T_0(M^+)$ and ${M^-}^{''} = T_0^{-1}(M^-)$,
this gives the required equality $s_0^{''} =  s_0$.
\end{proof}

\section{The proof of Lemma~\ref{lemma:FSNF}.}
\label{appendix}

We start from the well-known fact that  the local stable and unstable manifolds of $O$ can be
straightened {out} by means of a certain $C^r$-symplectic change of coordinates,\footnote{Let us recall
some details of this. We can always write the local map in the form $\bar x = \lambda(\varepsilon)
x + h_1(x,y,\varepsilon) \;,\; \bar y = \gamma(\varepsilon)y + h_2(x,y,\varepsilon)$, where
$|\lambda\gamma|=1$, $h_i(0,0,\varepsilon)= 0$.
Let
$y=\varphi(x,\varepsilon)$ be the equation of $W^s_{loc}$. Then, by the change $\xi = x, \eta = y -
\varphi(x,\varepsilon)$, we straighten {out} $W^s_{loc}$. Moreover, this change is symplectic, since it
is produced by the generating function $V(x,\eta,\varepsilon)=x \eta + \int
\varphi(x,\varepsilon) dx$. The manifold $W^u_{loc}$ is straightened {out} analogously.} i.e., the map $T_0$
can be written in the following form
\begin{equation}
\begin{array}{l}
\bar x = \lambda(\varepsilon) x + f(x,y,\varepsilon)x \;,\; \bar y = \gamma(\varepsilon)y +
g(x,y,\varepsilon)y \;, \label{eq:form1}
\end{array}
\end{equation}
where $f(0,0,\varepsilon)\equiv 0, g(0,0,\varepsilon)\equiv 0$.
In these coordinates, the fixed point $O_\varepsilon$ is in the origin and the equations of
$W^s_{loc}$ and $W^u_{loc}$ are $y=0$ and $x = 0$, respectively, for all sufficiently small
$\varepsilon$.

We consider  the map $T_\varepsilon$ in the initial form~(\ref{eq:form1}). This map is $C^r$ and can
be represented in the following ``n-th order extended form''
\begin{equation}
\begin{array}{l}
\bar x = \lambda(\varepsilon) x\{1 +
[\varphi_1^{(0)}(x,\varepsilon) + \psi_1^{(0)}(y,\varepsilon)] +
[\beta_1^{(1)} +
\varphi_1^{(1)}(x,\varepsilon) + \psi_1^{(1)}(y,\varepsilon)]\cdot xy + \\
\quad + [\beta_1^{(2)} + \varphi_1^{(2)}(x,\varepsilon) +
\psi_1^{(2)}(y,\varepsilon)]\cdot (xy)^2
+ \cdots + \\
\quad + [\beta_1^{(n)} + \varphi_1^{(n)}(x,\varepsilon) +
\psi_1^{(2)}(y,\varepsilon)]\cdot (xy)^n \} +
O(x^{n+2}y^{n+1})  \;,\;  \\
\bar y = \gamma(\varepsilon) y\{1 +
[\varphi_2^{(0)}(x,\varepsilon) + \psi_2^{(0)}(y,\varepsilon)] +
\quad + [\beta_2^{(1)} + \varphi_2^{(1)}(x,\varepsilon) +
\psi_2^{(1)}(y,\varepsilon)]\cdot xy\\
\quad + [\beta_2^{(2)} +
\varphi_2^{(2)}(x,\varepsilon) + \psi_2^{(2)}(y,\varepsilon)]\cdot
(xy)^2
+ \cdots  + \\
\quad + [\beta_2^{(n)} + \varphi_2^{(n)}(x,\varepsilon) +
\psi_2^{(2)}(y,\varepsilon)]\cdot (xy)^n\} + O(x^{n+1}y^{n+2})
\end{array}
\label{can3.1}
\end{equation}
where $|\lambda\gamma|=1$, $\beta_1^{(i)}$ and $\beta_2^{(i)}$ are  {constants}, 
$i=1,\ldots,n,$ $\varphi_k^{(i)}(0,\varepsilon)=\psi_k^{(i)}(0,\varepsilon)\equiv 0 \;,\; k=1,2$.
Denote $\alpha_{ki} \equiv [\varphi_k^{(i)}(x,\varepsilon) + \psi_k^{(i)}(y,\varepsilon)]$. Since
$T_\varepsilon \in C^r$, we have, due to the  {expansion} 
in~(\ref{can3.1}),  that $\alpha_{ki}\in C^{r-2i-1}$.

{Lemma~\ref{lemma:FSNF}} states that there exist canonical changes which {cancel the} functions $\alpha_{ki}$ and
transform constants $\beta_1^{i}$ and $\beta_2^{i}$ into the ``Birkhoff-Moser coefficients'' $\beta_i$
and $\tilde\beta_i$~, respectively. {In making} these changes {we will} see that
the change {cancelling} the term $\alpha_{ki}$ is $C^{r-2i-2}$~, while the next term
$\alpha_{k,i+1}$ is $C^{r-2(i+1)-1} = C^{r-2i-3}$~. That is, such a change {will} not change the
smoothness of the high order terms (in the sense of the  {expansion} 
in~(\ref{can3.1}). Thus, the
final smoothness will be equal to the smoothness of the last coordinate transformation.

Now we prove the lemma by induction on $i$. Note that Lemma~\ref{lemma:NF1order} can be considered
here as ``the first step of induction''.

Suppose that for some $i \leq n$ we have brought the map $T_\varepsilon$ to the form
\begin{equation}
\begin{array}{l}
\bar x = \lambda(\varepsilon) x\{1 + \beta_1(\varepsilon) \cdot xy
+ \beta_2(\varepsilon) \cdot (xy)^2 + ... +
\beta_{i-1}(\varepsilon) \cdot (xy)^{i-1} + \\
\quad + \beta_1^{(i)} + [\varphi_1^{(i)}(x,\varepsilon) +
\psi_1^{(i)}(y,\varepsilon)]\cdot (xy)^i \} +
O(x^{i+2}y^{i+1})  \;,\;  \\
\bar y = \gamma(\varepsilon) y\{1 + \tilde\beta_1(\varepsilon) \cdot xy +
\tilde\beta_2(\varepsilon) \cdot (xy)^2 + ...
+ \tilde\beta_{i-1}(\varepsilon) \cdot (xy)^{i-1} + \\
\quad + \tilde\beta_2^{(i)} + [\varphi_2^{(i)}(x,\varepsilon) + \psi_2^{(i)}(y,\varepsilon)]\cdot
(xy)^i\} + O(x^{i+1}y^{i+2})
\end{array}
\label{can3.2}
\end{equation}

Let us show that there exists a canonical change {cancelling} the terms $\alpha_{1i}$ and
$\alpha_{2i}$ and that the smoothness of such a change is equal to the smoothness of functions
$\alpha_{k,i}$ minus one. Then, the lemma will be proven.

For this goal we make two consecutive canonical changes with the following generating functions
\begin{equation}
V_1^{(i)}(x,\eta) = x\eta + (x\eta)^{i+1} v_1^{(i)}(x,\varepsilon) \;\;{\rm and}\;\; V_2^{(i)}(x,\eta) = x\eta +
(x\eta)^{i+1} v_2^{(i)}(\eta,\varepsilon),
\label{canv1-v2}
\end{equation}
where $v_k^{(i)}(0,\varepsilon)=0\;,\;k=1,2$. By means of these changes one can vanish functions
$\varphi_1^{(i)}$ and $\psi_2^{(i)}$ in~(\ref{can3.2}), respectively. After this, we show that the new
functions $\tilde\varphi_2^{(i)}$ and $\tilde\psi_1^{(i)}$  vanish due to equality to one of
$|J(T_\varepsilon)|$.

First, we make the change {associated to} the generating function
$V_1^{(i)}$ where $v_1^{(i)}(0,\varepsilon)=0$~. Thus, this change
is
\begin{equation}
\begin{array}{l}
\xi = x + (i+1)  x^{i+1}\eta^i v_1^{(i)}(x,\varepsilon) \;\;,\;\;
y = \eta +  x^{i}\eta^{i+1} \tilde v_1^{(i)}(x,\varepsilon)
\end{array}
\label{can3.3}
\end{equation}
where $\tilde v_1^{(i)}(x,\varepsilon) = (i+1)
v_1^{(i)}(x,\varepsilon) + x \cdot\partial v_1^{(i)}/\partial x$
and $\tilde v_1^{(i)}(0,\varepsilon) \equiv 0$~.

The first equation of~(\ref{can3.2}) is transformed to
\begin{equation}
\begin{array}{l}
\bar \xi = \bar x + (i+1) \bar x^{i+1}\bar\eta^i v_1^{(i)}(\bar x,\varepsilon) =
 \lambda x\{1 + \beta_1 \cdot xy +
\beta_2 \cdot (xy)^2 + \cdots  \\ \quad + \beta_{i-1} \cdot (xy)^{i-1} +
\beta_1^{(i)}\cdot (xy)^i +
\varphi_1^{(i)}(x,\varepsilon)\cdot (xy)^i + \\ \quad
+ \psi_1^{(i)}(y,\varepsilon)\cdot (xy)^i \} + (i+1) \lambda^{i+1}x^{i+1}\gamma^{i}y^{i}
v_1^{(i)}(\lambda x,\varepsilon) + O(\xi^{i+2}\eta^{i+1}) \\ =   \lambda\xi  +
 x^{i+1}y^i
\left[-(i+1) \lambda v_1^{(i)}(x,\varepsilon)) + (i+1) \lambda \delta_i v_1^{(i)}(\lambda
x,\varepsilon) +
\lambda \varphi_1^{(i)}(x,\varepsilon)\right] + \\ \quad
+ \lambda \xi\{\beta_1 \cdot \xi\eta + \beta_2 \cdot (\xi\eta)^2 +
\cdots   + \beta_{i-1} \cdot (\xi\eta)^{i-1}
+ \beta_1^{(i)}\cdot (\xi\eta)^i) +  \\ \quad
+ \psi_1^{(i)}(\eta,\varepsilon)\cdot \xi(\xi\eta)^i\}  + O(\xi^{i+2}\eta^{i+1}),
\end{array}
\label{can3.4}
\end{equation}
where $\delta_i = \mbox{sign}\;(\lambda\gamma)^i$. Now we take a function
$v_1^{(i)}(x,\varepsilon)$ to {cancel} the expression inside the square brackets in~(\ref{can3.4}),
i.e.,
\begin{equation}
\begin{array}{l}
v_1^{(i)}(\lambda x,\varepsilon) = \delta_i v_1^{(i)}(x,\varepsilon))
-\frac{1}{i+1}\varphi_1^{(i)}(x,\varepsilon)
\end{array}
\label{can3.5}
\end{equation}

Note that this equation has a solution in the class of functions (of variable $x$) whose smoothness
coincides with the smoothness of the function $\varphi_1^{(i)}(x,\varepsilon)$ (recall that
$\varphi_1^{(i)} \in C^{r-2i-1}$). The sought solution, $u = v_1^{(i)}(x,\varepsilon)$, can be
viewed as the equation
of the strong stable invariant manifold
$W^{ss}_i$ containing the point $(0,0)$ of the following planar map
\begin{equation}
\begin{array}{l}
\bar u = \delta_i u - \frac{1}{i+1}\varphi_1^{(i)}(x,\varepsilon) \;\;,\;\; \bar x =
\lambda(\varepsilon)x
\end{array}
\label{can3.6}
\end{equation}
(since  $W^{ss}$ is invariant, its equation $u = \phi_{ss}(x,\varepsilon)$ has to satisfy the
following homological equation:  $\phi_{ss}(\lambda x,\varepsilon) = \delta_i
\phi_{ss}(x,\varepsilon) -\frac{1}{i+1}\varphi_1^{(i)}(x,\varepsilon)$ that is,~(\ref{can3.5}).
Since $\delta_i = \pm 1$, such a manifold exists, it is $C^{r-2i-1}$ and, thus, the change
(\ref{can3.5}) is $C^{r-2i-2}$.

We can see from~(\ref{can3.3}) that the sought change is of the form
$$
x = \xi + O((\xi\eta)^{i+1}) \;,\; y = \eta + O((\xi\eta)^{i+1}).
$$
This means that, in the second equation of~(\ref{can3.2}), such a change can {affect}  only the
function $\lambda^{-1}\varphi_2^{(i)}(x,\varepsilon)x^iy^{i+1}$ from the explicitly shown ones in
(\ref{can3.2}): $\varphi_2^{(i)} \Rightarrow \tilde\varphi_2^{(i)}$~.

Thus, after change~(\ref{can3.3}), the map $T_\varepsilon$ {has} the form~(\ref{can3.2}) where
\begin{equation}
\begin{array}{l}
\varphi_1^{(i)}(x,\varepsilon) \equiv 0 \;,\; \varphi_2^{(i)} \equiv \tilde\varphi_2^{(i)},
\end{array}
\label{can3.75}
\end{equation}
and the other explicitly given functions are the same. Note
that the function $\psi_2^{(i)}(y,\varepsilon)$ does not {change}.

It is evident that the second coordinate transformation, {associated to} the second generating
function $V_2^{(i)} = x\eta + (x\eta)^{i+1} v_2^{(i)}(\eta,\varepsilon)$ with
$v_2^{(i)}(0,\varepsilon)=0$, is carried out quite similarly, 
due to the condition
$|\lambda\gamma|\equiv 1$, see also~\cite{GG09}.

Thus, after the canonical changes with {associated} generating functions $V_1^{(i)}$ and $V_2^{(i)}$ from
(\ref{canv1-v2}),  the map $T_\varepsilon$ takes the following form
\begin{equation}
\begin{array}{l}
\bar x = \lambda(\varepsilon) x\{1 + \beta_1(\varepsilon) \cdot xy
+ ... +
\beta_{i}(\varepsilon) \cdot (xy)^{i}\} 
+ \tilde\psi_1^{(i)}(y,\varepsilon)\cdot x^{i+1}y^i  +
O(x^{i+2}y^{i+1})  \;,\;  \\
\bar y = \gamma(\varepsilon) y\{1 + \tilde\beta_1(\varepsilon) \cdot xy + ...
+ \tilde\beta_{i}(\varepsilon) \cdot (xy)^{i}\} 
+ \tilde\varphi_2^{(i)}(x,\varepsilon)\cdot x^iy^{i+1} +
O(x^{i+1}y^{i+2})
\end{array}
\label{can3.12}
\end{equation}

Let us show that the equality $J(T_\varepsilon) \equiv 1$ implies
$\tilde\psi_1^{(i)} \equiv 0$ and $\tilde\varphi_2^{(i)} \equiv
0$. Indeed, we can represent the map~(\ref{can3.12})  as
\begin{equation}
\begin{array}{l}
\bar x = \lambda(\varepsilon) x B_i(xy) +
\tilde\psi_1^{(i)}(y,\varepsilon)\cdot x^{i+1}y^i  +
O(x^{i+2}y^{i+1})  \;,\;  \\
\bar y = \gamma(\varepsilon) y B_i^{-1}(xy) \tilde\varphi_2^{(i)}(x,\varepsilon)\cdot x^iy^{i+1} +
O(x^{i+1}y^{i+2})
\end{array}
\label{can3.13}
\end{equation}
where $B_i$ and $B_i^{-1}$ are the {truncations} of the Bikhoff-Moser normal form. Then the Jacobian of~(\ref{can3.13}) has the following form
$$
J = \pm 1 + (i+1)(\lambda\tilde\varphi_2^{(i)}(x,\varepsilon) +
\gamma\tilde\psi_1^{(i)}(y,\varepsilon))\cdot x^{i}y^i  + O((xy)^{i+1}),
$$
{from which} it follows that $\tilde\varphi_2^{(i)}\equiv 0$ and $\tilde\psi_1^{(i)}\equiv 0$.

In the non-orientable case $\lambda\gamma = -1$, the monomials of the form
$\beta_i x (xy)^i$ in the equation for $\bar x$ and $\tilde\beta_i y(xy)^i$ in the equation for $\bar y$ with odd $i$ are {non}-resonant.
Therefore, they can be {cancelled} (inside
every corresponding  step of the proof) by the canonical polynomial coordinate transformations with
generating functions $\tilde V_i = x\eta + \nu_i (x\eta)^{i+1}$. One can check that if
in~(\ref{can3.2}) all terms~$\beta_i$ and~$\tilde\beta_i$ vanish for odd $i$, except for the
last ones  $\beta_n$ and $\tilde\beta_n$ for odd $n$, then $\beta_n = -\tilde\beta_n$. Then the
change with the generating functions $\tilde V_n$ {cancels} both these terms simultaneously.

This completes the proof of the lemma.

\subsection*{Acknowledgments}

The authors would like to thank D.V.~Turaev and L.M.~Lerman for very fruitful discussions. The authors also thank
the two anonymous referees for the careful reading of the manuscript and the valuable comments and corrections which greatly improved
the final version of the paper.
This work has been partially supported by the Russian Scientific Foundation
Grant 14-41-00044.
Sections~5--7 has been carried out by the RSciF-grant (project No.14-12-00811). SG was partially supported
by the RFBR grants 13-01-00589 and 14-01-00344.
AD and MG have been
partially supported by the Spanish MINECO-FEDER Grant MTM2012-31714 and the Catalan
Grant 2014SGR504. MG has also been supported by the DFG~Collaborative Research Center TRR~109 ``Discretization in
Geometry and Dynamics''.

\end{document}